\documentclass[10pt,reqno]{amsart} 
\usepackage{amssymb,latexsym,  
amscd, 
amsthm, 
graphicx, epsfig}
\usepackage{cite} 

\theoremstyle{plain}
\newtheorem{theorem}{Theorem}
\newtheorem{lemma}{Lemma}
\newtheorem{corollary}{Corollary}
\newtheorem{proposition}{Proposition}

\theoremstyle{definition}

\newcommand{\Z}{\mathbb{Z}}
\newcommand{\W}{\mathcal{W}} 
\newcommand{\U}{\mathcal{U}}

\newcommand{\Oo}{\mathcal{O}}
\newcommand{\C}{\mathbb{C}}

\numberwithin{equation}{section} 

\newcommand{\nn}{\nonumber \\}

 \newcommand{\ord}{\mbox{\rm ord}}

\newcommand{\wt}{\mbox{\rm wt} }

\newcommand{\N}{\mathbb{N}}

\newcommand{\F}{\mathcal{F}}

\newcommand{\one}{\mathbf{1}}

\begin{document}
\title[Sequences of multiple products and cohomology classes]   
{Sequences of multiple products and 
cohomology classes for foliations of complex curves} 
\author{A. Zuevsky}
\address{Institute of Mathematics \\ Czech Academy of Sciences\\ Praha}

\email{zuevsky@yahoo.com}

\begin{abstract}
The idea of transversality is explored in the construction of 
 cohomology theory associated to adapted 
sequences of multiple products of rational functions 
associated to vertex algebra 
  cohomology of codimension one foliations on complex curves. 
Explicit formulas for cohomology invariants results from 
consideration transversality conditions applied to sequences of multiple 
products for elements of cochain transversal complexes defined 
for codimension one foliations. 


\end{abstract}


\maketitle
\begin{center}
{Conflict of Interest and Data availability Statements:}
\end{center}
The author states that:  

1.) The paper does not contain any potential conflicts of interests. 

2.) The paper does not use any datasets. No dataset were generated during and/or analysed 
during the current study. 

3.) The paper includes all data generated or analysed during this study. 

4.) Data sharing is not applicable to this article as no datasets were generated or analysed during the current study.

5.) The data of the paper can be shared openly.  

6.) No AI was used to write this paper. 

\section{Introduction}
\label{introduction}
In this paper we develop algebraic and functional-analytic 
methods of the cohomology theory of foliations 
on complex curves. 
The cohomology techniques applied to smooth manifolds are 
represented both by geometric 
\cite{G, Kaw, N0, Ko, Lawson, Thur} 
and algebraic \cite{GF1}   
approaches to characterization of foliation leaves.   
In the long list of works including 
\cite{Bott,  BS,  CM, Fuks1,  Losik, BG, BGG, Khor, Lawson}
can only partially reflect the contemporary theory of foliations 
involving a variety of approaches.   
As for the theory of vertex algebras 
\cite{BZF, DL, K, FHL}, 
it is represented now by 
a mixture of algebraic, conformal field theory, automorphic forms 
and several other fields of mathematics related studies. 
In the conformal field theory  
algebraic nature of vertex algebra methods applied  
\cite{FMS},  
provides extremely powerful tools to compute correlation functions. 
Geometric sewing constructions of higher genus 
 Riemann surfaces \cite{Y} 
provide models spaces for the construction of 
sequences of multiple products 
while the analytic part stems from the theory of 
vertex algebra correlation functions
and vertex operator algebra bundles defined 
 on complex 
curves \cite{BZF}. 

The idea of a characterization of the space of leaves of a foliation 
in terms of adapted sequences of rational functions  
with specific properties 
originates from  
 conformal field theory methods \cite{FMS, BZF, H2,  Zhu} 
and the algebraic structure of vertex algebra matrix elements. 
To introduce a sequence of multiple products  
for elements of families of cochain complexes
we use the rich algebraic and geometric structure of  
of vertex algebra matrix elements \cite{TZ, TZ1, TZ2, MTZ}.   
Computation of higher order cohomology invariant including powers of 
rational functions originating from vertex algebra matrix elements 
and generalizing  
 the classical cohomology classes \cite{G} 
constitutes the main result of the paper in addition to 
the general construction of a vertex algebra cohomology theory for 5
foliations and the machinery 
of multiple products for corresponding cochain complexes. 
Our approach to formulation of the foliation cohomology 
makes connection to the classical Lie-algebraic approach \cite{Fuks1} 
since vertex algebra represent, in particular, generalizations of Lie algebras. 
In comparison to the classical 
${\rm \check{C}}$ech-de Rham cohomology of foliations \cite{CM}, 
our approach involves deep algebraic properties related to vertex algebras 
 to establish new higher order cohomology classes. 
 
Let $W^{(i)}$, $1 \le i \le l$, be a set of 
grading-restricted generalized modules for a  
 grading-restricted vertex algebra $V$.  
In Section \ref{spaces}  
the families of cochain complexes
and corresponding coboundary operators 
associated to algebraic completions $\overline{W}^{(i)}$ of 
 grading-restricted vertex algebra modules $W^{(i)}$ 
are constructed to describe algebraic invariants for 
 a codimension one foliation $\F$ on a complex curve. 
In \cite{Huang}, for a grading restricted vertex algebra $V$, and its 
grading-restricted generalized module $W$  
the notion of $\overline{W}$-valued rational function 
was introduced. 
In this paper we denote by $\W_{z_1, \ldots, z_n}$ the space of 
$\overline{W}$-valued differential forms with specific properties
combining $\overline{W}$-valued functions and invariant differentials. 
That notion we describe in Section \ref{transversal}. 
The transversality conditions established  
for sequences of multiple product defined on 
 the families of vertex algebra cochain complexes 
result in sequences of general higher invariants of higher orders 
of functions and their derivatives.    
\subsection{The main result of the paper}
Let $F \in C^{k_i}_{m_i}\left(V, \W^{(i)}, \F\right)$.
Let us introduce the set of cohomology classes, for $k$, $m \in \N$, and 
$\beta=0$, $1$,  
\begin{eqnarray}
\label{pampadur}
\left[Sym \cdot_{\rho_1, \ldots, \rho_l}  
\left( 
 \left(\delta^{k_i}_{m_i}  F^{(i)}\right)^m,  
  \left(\partial_t F^{(i')} \right)^\beta, 
   \left( F^{(i'')} \right)^k    
   \right) \right], && 
\end{eqnarray}
where the symmetrization is performed over all 
possible positions of the differentials and elements in 
the multiple product. 
We consider also a smoothly varying one real parameter $t$ families
 of transversely oriented codimension one foliations on $M$, 
with $F$ depending on $t$. 
The main statement of this paper consists in  
the following Theorem proven in Section \ref{gv} 
and generalizing classical results of \cite{G} on 
codimension one foliation invariants: 
\begin{theorem}
\label{talasa} 
For families of complexes 
$\left\{ C^{k_i}_{m_i}\left(V, \W^{(i)}, \F \right) \right\}$, 
$1 \le i \le n$, 
 the sequence of multiple products \eqref{gendef},  
the coboundary operators \eqref{deltaproduct},   
\eqref{halfdelta}, 
 the transversality condition \eqref{ortho}
 applied to the families of cochain  complexes \eqref{hat-complex}, and   
\eqref{hat-complex-half}, 
and  satisfying the mutual orders condition 
$\ord \left(\delta^{k_{i_s}}_{m_{i_s}} \Phi^{(i_s)}, \Psi^{(i_{s'})}\right)<m+k-1$, 
generate an non-vanishing infinite series of 
cohomology classes of invariants \eqref{pampadur}  
for
$(2-m)k_i -m +1   - \beta k_{i'} - k k_{i''}<0$, 
and 
$(2-m)m_i + m-1 - \beta m_{i'} - k m_{i''} <0$. 
where $\beta=0$, $1$; $k$, $m \ge 0$; 
$k_i$, $k_{i'}$, $k_{i''}$, $m_i$, $m_{i'}$, $m_{i''} \ge 0$. 
The invariants are 
independent on the choice of $F^{(i)}$, $F^{(i')}$, $F^{(i'')}$  
satisfying the transversality conditions \eqref{prontosan}.    
Similar for 
 the families of short complexes \eqref{shorta} 
for an infinite series of pairs $(k_{i_s}, m_{i_s}, )= ((1, i_s), (2, i_s), l)$,
 $((0, i_s),(3, i_s))$, $((1_s), t)$, $i_s=i, i', i''$, $0 \le t \le 2$.  
\end{theorem}

Results of this paper promise to be developed in various directions. 
In particular, papers \cite{Huang,  CE, FF1988} suggest 
several approaches to cohomology formulation and computation for 
vertex algebra related structures. 
The general theory of characteristic classes for arbitrary 
codimension foliations, and, in particular, possible classification 
of foliations leaves remain the most desirable problems 
in the contemporary theory of foliations. 
The algebraic and geometric origin of problems considered in this paper 
hint natural directions to generalize constructions 
associated with vertex algebras and applications. 
In particular, the problem to distinguish 
\cite{BG, BGG} types of compact and non-compact leaves
of foliations, requires a further development of algebraic and analytical 
methods to compute higher order cohomology invariants discussed in this paper. 
In \cite{Losik} the author introduced a foliation theory in terms of frames. 
We would be interested in a development of results of that paper 
with the vertex algebra theory applied to 
 smooth structures on the space of leaves for foliations.
For smooth manifolds, 
a completely intrinsic cohomology theory formulated 
in terms of vertex operator algebra bundles \cite{BZF}
would lead to further applications for classification 
of foliation leaves \cite{BG, BGG}. 
In relation to the classical paper \cite{BS}, 
one would be interested in clarifying 
the idea of auxiliary vertex operator algebra bundles construction 
in order to compute cohomology of foliations.   
In a separate paper we will 
 consider a cohomology theory for vertex operator algebra bundles \cite{BZF}    
defined on  arbitrary codimension foliations on smooth manifolds. 

The plan of the paper is the following. 
Section \ref{transversal} contains a description of the transversal structures  
 for foliations.    
In Subsection \ref{intep} a vertex algebra interpretation for the local geometry 
of foliations is described.  
In Subsection \ref{composable} 
the definition and properties of maps 
adapted transversal to a number of vertex operators are given. 
In Section \ref{product} we introduce 
sequences of multiple products of elements of $\W^{(i)}$-spaces 
and study their properties. 
Subsection \ref{geoma} contains a geometric motivation leading to 
the notion of sequences of multiple products. 
In Subsection \ref{kapusta} the elimination of coinciding  
vertex algebra elements and corresponding formal parameters is described. 
Subsection \ref{nahuy} constructs the adaptation operation 
for special type of matrix elements leading to rational functions. 
The definition of 
the sequence of multiple products of elements of spaces of differential forms
is introduced. 
In Subsection \ref{glavna} 
 we prove that the sequence of multiple products map to the tensor product 
$\W^{(1, \ldots, l)}$-space.   
In Subsection \ref{symprop} 
we prove that 
 a sequence of multiple products
satisfies a symmetry property \eqref{shushu}. 
In Subsection \ref{exderconj}
it is shown that sequences of multiple products satisfy 
 $L_V(-1)$-derivative and $L_V(0)$-conjugation properties.  
In Subsection \ref{properties1}   
invariance of sequences of multiple products 
under the action of the group of independent transformations 
of coordinates is proven. 
The spaces for families of 
chain complex associated to a vertex algebra on a foliation 
are introduced in 
 Section \ref{spaces}. 
In Subsection \ref{noprop} 
properties of spaces for vertex algebra complexes are studied. 
Subsection \ref{cobcob} introduces the coboundary operators for
 the families complexes in 
our formulation.     
Sequences of multiple products for families of complexes 
 are defined in Section \ref{productc}.
In Subsection \ref{adapta} the geometric interpretation  
 of multiple products for a foliation
is discussed.
The properties of the product are studied in Section \ref{pyhva}.   
In Subsection \ref{toto} an analogue of Leibniz rule is proven for 
sequences of multiple products for spaces of complexes.  
Section \ref{gv} contains the proof of Theorem \ref{talasa}, 
 the main result of this paper. 
Explicit formulas for multiple products 
 cohomology invariants for a codimension one foliation  
on a smooth complex curve are found. 
In Subsection{cohomological} the notions related to a vertex operator algebra cohomology 
are introduced. 
Subsection \ref{multpartprod} defines the transversality conditions for 
multiple products. 
Subsection \ref{comult} introduces the series of multiple parametric 
commutator products for elements 
of the families of cochain complex spaces. 
Finally, Subsection \ref{proofthe} contains the proof of Theorem \ref{talasa}. 
In the Appendix we provide the material required 
for the construction of the vertex algebra 
cohomology of foliations. 
Properties of matrix elements for spaces $\W^{(i)}$ are listed. 
\section{Transversal structures for a foliation}
\label{transversal}  
We refer to \cite{CM} for the definitions and properties of 
a basis of transversal sections for 
foliations and corresponding holonomy of a foliation.  
In \cite{zucmp} the notion of a 
holomorphic multi-point connections on a smooth complex variety was introduced. 
The factor space $H^n= {\mathcal Con}^n_{cl }/G^{n-1}$ 
 of closed multi-point connections 
with respect to the space of connection forms 
determines the cohomology.  
A construction of a vertex algebra cohomology of foliations 
in terms of connections 
related to \cite{BS}
will be given in a separate paper. 
The formulation of a vertex algebra cohomology of a foliation 
given in the Section 
\ref{productc} 
is partially motivated by the construction of the 
 ${\rm \check C}$ech-de~Rham cohomology \cite{CM}. 

 Let us we provide several definitions and properties from \cite{Huang}.
\label{shuffles}
For the permutation group $S_q$, the elements of 
$J_{l, s}=\{\sigma\in S_l\;|\;\sigma(1)<\ldots <\sigma(s), \;   
 \sigma(s+1)<\ldots <\sigma(l)\}$. 
are called shuffles. 
Here $l \in \N$ and $1\le s \le l-1$, let $J_{l; s}$ is the set of elements of 
$S_l$ which preserves the order of the first $s$ and the last  
$l-s$ numbers. 
We denote also  
$J_{l; s}^{-1}=\{\sigma\;|\; \sigma\in J_{l; s}\}$.  
For $n\in \Z_+$, the configuration space is defined by  
$F_n\C=\{(z_1, \ldots, z_n)\in \C^n\;|\; z_i\ne z_j, i\ne j\}$.    
In the Appendix we review 
the notion of a grading-restricted 
vertex algebra $V$, and its grading-restricted generalized $V$-module $W$.   
The algebraic completion 
$\overline{W}=\prod_{n\in \mathbb C}W_{(n)}=(W')^*$.   
of $W$ will be denoted as $\overline{W}$ in what follows. 
We notate by $Rf(z_1, \ldots, z_n)$ 
a rational function   
 if a meromorphic function $f(z_1, \ldots, z_n)$ defined  
  on the configuration space 
$F_{n}\mathbb{C}$ is  
 analytically extendable to a rational function in $(z_1, \ldots, z_n)$.  
For any $w'\in W'$, a map 
 $f: F_n\C 
\to
 \overline{W}$,   
 $(z_1, \ldots, z_n)  
\mapsto 
 f(z_1, \ldots, z_n)$,      
is called a $\overline{W}$-valued rational  
function in $(z_1, \ldots, z_n)$  
with the only possible poles at 
$z_i=z_j$, $i\ne j$,  
the bilinear pairing 
(see 
the Appendix)   
 $\langle w', f(z_1, \ldots, z_n) \rangle$ defined for $W$ 
is a rational function $f(z_1, \ldots, z_n)$ in $(z_1, \ldots, z_n)$   
with the only possible poles  at 
$z_i=z_j$, $i\ne j$.  
We denote by $\widetilde{W}_{z_1, \ldots, z_n}$ 
 the space of $\overline{W}$-valued rational functions. 
Since it does not bring any misunderstanding, 
we will use the same notation $\langle ., . \rangle$ for 
bilinear pairings for different modules of $V$.  
The complex-valued bilinear pairing with 
an element $f$ of the algebraic completion $\overline{W}$  
inserted characterizes 
a $\widetilde{W}_{z_1, \ldots, z_n}$-valued rational function.  

Let ${\rm Aut}(V)$ be a group of automorphisms of $V$ 
with elements 
 $g \in {\rm Aut}(V)$.  
We define that the action of $g$ on the tensor product $V^{\otimes n}$ as 
$g.(v_1 \otimes ... \otimes v_n) = g.v_1 \otimes ... \otimes g.v_n$. 
Then $g$ automatically commutes with the action of $S_n$. 
Let $g$ commute also with $L_V(-1)$ and $L_W(-1)$.    
In \cite{MTZ, TZ1, TZ2} we considered various versions of orbifolding 
$n$-point correlation constructions for vertex operator algebras. 
In paricular, it included presence of an automorphisms group element 
$g\in Aut(V)$ in expressions for correlations functions. E.g., in particular, 
such twisted torus correlation functions had the form 
$Tr_V\left(g Y(v_1, z_1) \ldots Y(v_n, z_n)\right)$. 
Similarly, for a function $\Phi \in \widetilde{W}_{z_1, \ldots, z_n}$ 
 we include of automorphism element 
$\Phi(g; v_1, z_1;  \ldots ; v_n,  z_n)$ 
 acting on elements of the corresponding module $W$. 
As we know from, e.g., \cite{MTZ},  
 that enriches the analytic structure of a vertex operator algebra matrix elements. 
Since matrix elements are then involved in determination of cohomology invariants 
it is also useful to include them in our considerations.  

Now let us define the space of 
$\widetilde{W}_{z_1, \ldots, z_n}$-valued differential forms for 
a quasi-conformal grading-restricted vertex algebra $V$. 
This space is used in the construction of 
 families of cochain complexes  
describing the vertex algebra cohomology 
of foliations on complex curves. 
 The weight $\wt(v)$ of a homogeneous element $v$ of a vertex algebra with respect  
 to Virasoro algebra $L_V(0)$-mode is defined in the Appendix.     
In \cite{BZF} it was proven that, for a primary $v \in V$ element, 
i.e., $L_V(n) v = 0$ for every $n > 1$, 
 a vertex operator $Y(v, z)$ multiplied by 
  $\wt(v)$-power of the corresponding differential $dz^{\wt(v)}$  
is invariant (see the Appendix)    
with respect to the formal parameter $z$ changes  
(for a quasi-conformal vertex operator algebra).  
We consider $\widetilde{W}_{z_1, \ldots, z_n}$-valued   
 functions $\Phi$ for primary $v_i \in V$, $1 \le i \le n$, 
and formal parameters $z_i$, 
endowed with $\wt(v_i)$-powers of the corresponding differentials: 
\begin{eqnarray*} 
&& \Phi \left(g; dz_1^{\wt(v_1)} \otimes v_1, z_1; \ldots; 
 dz_n^{\wt (v_n)} \otimes  v_n, z_n\right)
\nn
&& \quad = \Phi \left(g; v_1, z_1; \ldots; 
  v_n, z_n\right) dz_1^{\wt(v_1)} \ldots dz_n^{\wt (v_n)}. 
\end{eqnarray*}
Let us underline that this notation is sometimes 
useful for further considerations. 
In what follows,
we denote these forms as 
$\Phi(g; v_1, z_1; \ldots; v_n, z_n)$
abusing notations. 

For $n\in \Z_+$, $v_i\in V$,  $1 \le i \le n$, and arbitrary $w'\in W$,  
$\Phi(g; v_1, z_1;  \ldots ; v_n,  z_n)$     
 is said to have
the $L_V(-1)$-derivative property if 
\begin{eqnarray}
 &&\langle w', \partial_{z_i} \Phi(g; v_1, z_1;  \ldots ; v_n,  z_n)\rangle=  
\langle w',  \Phi(g; v_1, z_1;  \ldots; L_V(-1)v_i, z_i; \ldots ; v_n,  z_n) \rangle,   
\nn 
\label{lder1}
  &&\sum\limits_{i=1}^n\partial_{z_i} \langle w', 
\Phi(g; v_1, z_1;  \ldots ; v_n,  z_n)\rangle= 
\langle w', L_W(-1).\Phi(g; v_1, z_1;  \ldots ; v_n,  z_n) \rangle. 
\end{eqnarray}  
For $\sigma\in S_n$, and $v_i \in V$, $1 \le i \le n$, 
\begin{equation}
\label{sigmaction}
\sigma(\Phi)(g; v_1, z_1; \ldots; v_n, z_n)  
=\Phi(g; v_{\sigma(1)}, v_{\sigma(1)};  \ldots; v_{\sigma(n)}, z_{\sigma(n)}), 
\end{equation}   
defines the action of the symmetric group $S_n$.   
The permutation given by $\sigma_{i_1, \ldots, i_n}(j)=i_j$,
 will be notated as $\sigma_{i_1, \ldots, i_n}\in S_n$ 
for $1 \le j \le n$. 
For $v_j \in V$, $1 \le j \le n$,  
$w'\in W'$, $(z_1, \ldots, z_n)\in F_n\C$ and $z\in \C^\times$, 
$(zz_1, \ldots, zz_n)\in F_n\C$,   
$\Phi$ 
 satisfies the $L_V(0)$-conjugation property if 
\begin{eqnarray}
\label{loconj}
\langle w', z^{L_W(0)}    
\Phi \left(g; v_1, z_1; \ldots; v_n, z_n \right) \rangle  
=\langle w', 
 \Phi(g; z^{L_V(0)} v_1, zz_1;  \ldots ;  z^{L_V(0)} v_n,  zz_n)\rangle. 
\end{eqnarray} 

From considerations of \cite{BZF}  it follows 
\begin{proposition}
\label{pupa0}
For primary  elements $v_j \in V$, $1 \le j \le n$,  
  of a quasi-conformal grading-restricted vertex algebra $V$,  
$\Phi(g; v_1, z_1; \ldots; v_n, z_n)$
 is canonical with respect to the action of the group  
$\left( {\rm Aut} \; \Oo\right)^{\times n}_{z_1, \ldots, z_n}$    
of independent $n$-dimensional changes  
\begin{eqnarray}
\label{zwrho}
&&(z_1, \ldots, z_n)   
\mapsto 
(\widetilde{z}_1, \ldots, \widetilde{z}_n)  
= 
(\varrho(z_1), \ldots, \varrho(z_n)). 
\end{eqnarray}    
\end{proposition}
\label{wspace}
We define the space $\W_{z_1, \ldots, z_n}$ of forms 
$\Phi$ $\left(g \right.$;$dz_1^{\wt(v_1)}$ $\otimes$ $v_1$, $z_1$; $\ldots$;  
 $dz_n^{\wt (v_n)}$ $\otimes$  $v_n$, $\left. z_n\right)$ 
satisfying $L_V(-1)$-derivative \eqref{lder1}, $L_V(0)$-conjugation 
\eqref{loconj} properties, and the symmetry property with 
 respect to the action of the symmetric group $S_n$:  
\begin{equation}
\label{shushu} 
\sum_{\sigma\in J_{l; s}^{-1}}(-1)^{|\sigma|}
 \Phi(g; v_{\sigma(1)}, z_{\sigma(1)}; 
\ldots; v_{\sigma(n)},  z_{\sigma(n)})=0. 
\end{equation} 
\subsection{Geometric setup for a foliation in terms of a vertex algebra} 
\label{intep}
Let us first recall \cite{CM} the notion of a basis of transversal sections for foliations.  
 Let $\mathcal M$ be a complex curve endowed 
 with a foliation $\F$ of codimension one.   
A transversal section of a foliation $\F$ is an embedded 
 one-dimensional 
 submanifold $U\subset M$ which  
is everywhere transverse to the leaves of $\F$.    
If $\alpha$ is a path between $p_1$ and $p_2$ on the same leaf of $\F$,  
and 
$U_1$ and $U_2$ are transversal sections through $p_1$ and $p_2$, 
 then $\alpha$ defines a transport 
 along the leaves from a neighborhood of $p_1$ in $U_1$ to a neighborhood of $p_2$ in $U_2$.  
 That gives a germ of a diffeomorphism   
$hol (\alpha): (U_1, p_1)\hookrightarrow  (U_2, p_2)$,  
which is called the holonomy of the path $\alpha$. 
Two homotopic paths always define the same holonomy. 
If the above transport along $\alpha$ is defined in all of $U_1$ and embeds $U_1$ into $U_2$, this embedding
$h: U_1\hookrightarrow U_2$,  
is called the holonomy embedding.    
A composition of paths induces a composition of holonomy embeddings.  
Transversal sections $U$ through $p$ as above should be thought of 
as neighborhoods of the leaf through $p$ in
the space of leaves. 
A transversal basis for the space of leaves $\mathcal M/\F$ of a foliation $\F$ is a  
family $\U$ of transversal sections $U \subset \mathcal M$ 
with the following property. 
If $U_p$ is any transversal section through a given $p\in \mathcal M$, 
then there exists a 
holonomy embedding  
$h: U\hookrightarrow U_p$,   
 with $U\in \U$ and $p\in h(U)$. 
A transversal section is a one-dimensional disk given by a chart of $\F$.   
Accordingly, we can construct 
a transversal basis $\U$ out of a basis 
$\widetilde{\U}$ of $\mathcal M$ by domains of foliation charts 
 $\phi_U: \widetilde{U}\tilde{\hookrightarrow } \mathbb{R} \times U$,    
$\widetilde{U}\in \widetilde{\U}$,
 with $U=\mathbb{R}$. 

We consider a $(n,k)$-set of points, 
$n \ge 1$, $k \ge 1$,    
$\left(p_1, \ldots, p_n; p'_1, \ldots, p'_k \right)$,   
on a smooth complex curve $M$.  
Let us denote the set of the corresponding local coordinates by  
$\left( c_1(p_1), \ldots, c_n(p_n); c'_1(p'_1), \ldots, c'_k(p'_k) \right)$.   
For a grading-restricted vertex algebra $V$, 
 we consider a set $\left\{W^{(l)}, l\ge 1 \right\}$ 
of its grading-restricted generalized modules.   

For the first $n$ grading-restricted vertex algebra $V$ elements of 
\begin{equation}
\label{vectors}
\left(v_1, \ldots, v_n; v'_1, \ldots, v'_k\right),   
\end{equation}
we consider the linear maps 
\begin{equation} 
\label{maps}
\Phi: V^{\otimes n} \rightarrow \W_{c_1(p_1), \ldots, c_n(p_n)},   
\end{equation}
\begin{eqnarray}
\label{elementw}
&&\Phi \left(g; dc_1(p_1)^{\wt (v_1)} \otimes v_1, c_1(p_1); \ldots; dc_n(p_n)^{\wt (v_n)} 
\otimes v_n,  c_n(p_n) \right)
\nn
&& \quad =    
\Phi \left(g;  v_1, c_1(p_1); \ldots; v_n,  c_n(p_n) \right) 
dc_1(p_1)^{\wt (v_1)} \ldots dc_n(p_n)^{\wt (v_n)}. 
\end{eqnarray}
In our setup, 
we identify formal parameters   
$(z_1, \ldots , z_n)$ of $\W_{z_1, \ldots , z_n}$, 
with local coordinates $(c_1(p_1), \ldots, c_n(p_n))$ 
 around points $p_i$, $0 \le i \le n$, on $M$.  
%
%
In \cite{zucmp} we proved, that  
 for arbitrary sets of vertex algebra elements 
$v_i$, $v'_j \in V$, $1 \le i \le n$, $1 \le j \le k$,  
 arbitrary sets of points $p_i$
 endowed with local coordinates $c_i(p_i)$ on $M$, 
 and arbitrary sets of points 
  $p'_j$ endowed with local coordinates $c'_j(p'_j)$ 
on the transversal sections $U_j\in \U$ of $M/\F$, 
the element \eqref{elementw} 
 as well as the vertex operators
 \begin{eqnarray}
\omega_W\left(dc'_j({p'_j})^{\wt(v'_j)} \otimes v'_j,  c'_j({p'_j})\right)
&=& Y_W\left( d(c'_j(p'_j))^{\wt(v'_j)} \otimes v'_j,  c'_j({p'_j})\right) 
\nn
&=& Y_W\left( v'_j,  c'_j({p'_j})\right) \; d(c'_j(p'_j))^{\wt(v'_j)}
\label{poper} 
\end{eqnarray} 
 are invariant under the action of the group of independent transformations of coordinates.  
We have already commented in the beginning of this Section on notations
 of differential forms of $\W_{z_1, \ldots, z_n}$.  
 The form \eqref{poper} represents   
a convenient way to notate ordinary vertex operators
 multiplied by $\wt(v_i)$-powers 
of the corresponding differential $\left(dz_i\right)^{\wt(v_i)}$.  

In the construction of spaces for families of cochain complexes  
associated to a grading-restricted vertex algebra   
we consider sections $U_j$, $j \ge 0$ of a transversal basis $\U$ of $\F$,  
and mappings $\Phi$ that belong to the space $\W_{c(p_1), \ldots, c(p_n)}$
for local coordinates $(c(p_1), \ldots, c(p_n))$ on $M$   
at intersections $(p_1, \ldots, p_n)$ of $U_j$ with leaves of $M/\F$ of $\F$. 
Consider a collection of $k$ transversal sections $U_j$, $1 \le j \le k$ of $\U$.    
In order to define the vertex algebra cohomology of $M/\F$,   
we assume that  
mappings $\Phi$ are adapted transversal to $k$ vertex operators.   
We choose one point $p'_j$ with a local coordinate $c'_j(p'_j)$ 
on each transversal section $U_j$, $1 \le j \le k$. 
Let us assume that $\Phi$ is adapted transversal to $k$ vertex operators. 
We denote by $c'_j(p'_j)$, $1 \le j \le k$ 
the formal parameters of $k$ vertex operators
adapted transversal to a map $\Phi$. 
The notion of a adapted transversal map $\Phi$
 to a number of vertex operators
consists of two conditions on $\Phi$.
 The adapted transversal conditions require 
the existence of positive 
integers $N^n_m(v_i, v_j)$, 
depending on vertex algebra elements $v_i$ and $v_j$ only,  
restricting orders of poles 
for the corresponding sums \eqref{Inm}. 
\subsection{The adaptation of transversal operators}
\label{composable}
In the construction of the families of 
 cochain complexes 
 we will use linear maps from tensor powers of $V$ to the space 
$\W_{z_1, \ldots, z_n}$.  
 For that purpose, in particular, 
 to define a family of 
 coboundary operators, 
 we have to adapt compositions of the vertex operator 
transversal structure of cochains 
 associated to a transversal basis for a foliation, 
with vertex operators.
To make the adaptation mentioned above 
one considers \cite{H2} series obtained by projecting 
 elements of a $V$-module algebraic completion 
to their homogeneous components. 
Recall definitions and notations of the Appendix. 
For a generalized grading-restricted $V$-module  
$W=\coprod_{n\in \C} W_{(n)}$,   
 and $q\in \C$, 
let
$P_q: \overline{W}\to W_{(q)}$,   
 be the projection from 
$\overline{W}$ to $W_{(q)}$.  
 Let $v_t \in V$, $m\in \N$, $1 \le t \le m+n$, $w'\in W'$, and 
 $l_1, \ldots, l_n\in \Z_+$ be such that $l_1+\ldots +l_n=m+n$.   
 Define  
$\Xi_s 
= 
E^{(l_s)}_V(v_{k_1}, z_{k_1}- \varsigma_s;  
v_{k_s}, z_{k_s}- \varsigma_s 
 ; \one_V)$,       
where 
 $k_1=l_1+\ldots +l_{s-1}+1$, 
 $\ldots$, 
$k_i=l_1+\ldots +l_{s-1}+l_s$,    
for $s=1, \ldots, n$.  

For a linear map $\Phi: V^{\otimes n} \to \W_{z_1, \ldots, z_n}$, 
the adapted transversal to $m$ vertex operators 
for $v_{1+m}, \ldots, v_{n+m} \in V$, 
is given by the adaptation procedure $R$ that takes 
an analytic extension of the matrix elements 
\begin{eqnarray}
\label{Inm}
\mathcal R^{1, n}_m(\Phi)=
R \sum_{r_1, \ldots, r_n\in \Z}\langle w',  
\Phi\left(g; P_{r_1} \Xi_1; \varsigma_1; 
 \ldots; 
P_{r_n} \Xi_n, \varsigma_n\right)  
\rangle, 
\end{eqnarray} 
\begin{eqnarray*} 
\mathcal R^{2, n}_m(\Phi)=  
R \sum_{q\in \C}\langle w', E^{(m)}_W \left(v_1, z_1; \ldots; v_m, z_m;   
P_q( \Phi(g; v_{1+m}, z_{1+m}; \ldots v_{n+m}, z_{n+m} \right)\rangle 
\end{eqnarray*}
to the rational functions 
in $z_1, \ldots, z_{m+n} \in \C$, 
 independent of $\varsigma_1, \ldots, \varsigma_n \in \C$,  
 absolutely convergent on the domains 
\begin{eqnarray}
\label{granizy1}
|z_{l_1+\ldots +l_{i-1}+p}-\varsigma_i| 
+ |z_{l_1+\ldots +l_{j-1}+q}-\varsigma_i|< |\varsigma_i 
-\varsigma_j|, 
\nn
 1 \le i\ne j \le k, \quad  p=1, 
\ldots,  l_i, \quad 
  q=1, \ldots, l_j,  
\end{eqnarray} 
\begin{eqnarray}
\label{granizy2}
z_{i'}\ne z_{j'}, \quad {i'}\ne {j'}, \quad  
|z_{i'}|>|z_k|>0, \; i'=1, \ldots, m; \;  k=m+1, \ldots, m+n, 
\end{eqnarray}
correspondingly,  
with the pole singular points restricted to 
$z_i=z_j$, of order less than or equal existing 
$N^n_m(v_i, v_j)\in \Z_+$,   
depending only on $v_i$ and $v_j$.  
\section{Sequences of multiple products}    
\label{product}
Let $\left\{W^{(i)}, 1 \le i \le l \right\}$ 
be a set of grading-restricted generalized $V$-modules.  
 In this Section we introduce  
 the sequences of products of elements for  
 a few $\W^{(i)}_{x_{1, i}, \ldots, x_{k_i, i}}$-spaces,     
and study their properties.  
A sequence of multiple products 
defines an element of the tensor product of several 
$\W$-spaces characterized by a converging adapted rational function
resulting from the product of matrix elements of the corresponding $V$-modules. 
\subsection{The geometric motivation for multiple  
products of $\W$-spaces}  
\label{geoma}
By using geometric ideas,  
we will introduce sequences of multiple products for elements of  
$\W^{(i)}_{x_{1, i}, \ldots, x_{k_i, i}}$-spaces 
though their algebraic structure is quite complicated.
 Let us associate a certain model space to each of 
$\W^{(i)}_{x_{1, i}, \ldots, x_{k_i, i}}$-spaces. 
Then a geometric model for a sequence of products should be defined,
and a sequence of algebraic products of 
$\W^{(i)}_{x_{1, i}, \ldots, x_{k_i, i}}$-spaces should be introduced. 
%
For a (not necessary finite) set of 
 $\W^{(i)}_{x_{1, i}, \ldots, x_{k_i, i}}$-spaces, 
$1 \le i \le l$, 
 $k_l \ge 0$,   
 we first associate formal complex parameters  
in sets 
$(x_{1, i}, \ldots, x_{k_i, i})$ 
to parameters of $i$ auxiliary   
spaces.  
The formal parameters of the algebraic product of $l$ spaces  
$\W^{(l)}_{z_1, \ldots, z_{k_1+ \ldots + k_l} }$, 
should be then identified with  
parameters of resulting model space.  
 We take the Riemann sphere $\Sigma^{(0)}$ as our  
 initial auxiliary geometric model space to form  
a sequence of multiple products of spaces of differential forms $\W^{(l)}_{}$  
constructed from matrix elements (see Subsection \ref{nahuy}). 
 The resulting auxiliary/model 
space is formed by 
 a Riemann surface $\Sigma^{(l)}$ of genus $l$ 
obtained by the multiple $\rho_i$-sewing procedures of attaching $l$ handles  
to the initial Riemann sphere $\Sigma^{(0)}$   
where $\rho_i$ are complex parameters, $1 \le i \le l$.  
The local coordinates of $k_1 + \ldots + k_l$ points 
on the Riemann surface $\Sigma^{(l)}$ 
are identified with 
the formal parameters $(x_{1, 1}, \ldots, x_{k_l, l})$, $l \ge 1$.  

We now recall the $\rho$-sewing  
 construction \cite{Y} of a Riemann surface $\Sigma^{(g+1)}$ 
 formed by self-sewing a handle 
to a Riemann surface $\Sigma^{(g)}$ of genus $g$.  
Consider a Riemann surface $\Sigma^{(g)}$  
of genus $g$,
 and let $\zeta_1$, $\zeta_2$ 
be local coordinates in the neighborhood 
of two separated points $p_1$ and $p_2$
on $\Sigma^{(g)}$. 
 For $r_a>0$, $a=1$, $2$, 
 consider two disks  
$\left\vert \zeta_a\right\vert \leq r_a$. 
To ensure that the disks do not intersect 
the radia $r_1$, $r_2$ must be sufficiently small.  
Introduce a complex parameter $\rho$ 
where $|\rho |\leq r_1r_2$, 
 and excise the disks
\begin{equation}
\label{disk}
\{\zeta_a:\, \left\vert \zeta_a\right\vert <|\rho |r_{\bar{a}}^{-1}\}
\subset \Sigma^{(g)},  
\end{equation} 
to form a twice-punctured surface  
$\widehat{\Sigma}^{(g)}=\Sigma^{(g)} \backslash 
\bigcup_{a=1,2}\{z_a :\, \left\vert \zeta_a \right\vert <|\rho |r_{\bar{a}}^{-1}\}$.   
We use the notation  
$\bar 1=2$, $\bar 2=1$.
The annular regions 
$\mathcal{A}_a\subset \widehat{\Sigma}^{(g)}$ 
are defined though the relation 
\begin{equation}
\label{zhopki}
\mathcal{A}_a=\{\zeta_a:\, |{\rho }|r_{\bar{a}}^{-1}
\leq \left\vert \zeta_a \right\vert \leq r_a\}, 
\end{equation}  
and identify them as a single region 
$\mathcal{A}=\mathcal{A}_1 \simeq \mathcal{A}_2$ 
 via the sewing relation 
\begin{equation}
\zeta_1 \zeta_2=\rho, 
 \label{rhosew}
\end{equation}
to form a compact Riemann surface  
$\Sigma^{(g+1)}=\widehat{\Sigma}^{(g)}\backslash \{\mathcal{A}_1\cup \mathcal{A}_2\}
\cup \mathcal{A}$,   
of genus $g+1$.
The multiple sewing procedure repeats the above construction several times
with complex sewing parameters $\rho_i$, $1 \le i \le l$.  
Thus, starting from the Riemann sphere it forms a genus $l$ Riemann surface. 
As a parameterization of a cylinder 
connecting the punctured Riemann surface to itself 
we can consider 
 the sewing relation \eqref{rhosew}.  
%
When we identify the annuluses \eqref{zhopki}
in the $\rho$-sewing procedure, 
certain $r$ points among points 
$(p_1, \ldots, p_{k_1+\ldots+k_l})$  
may coincide.    
This corresponds to the singular case of coincidence of $r$ formal parameters. 
\subsection{The elimination of coinciding parameters in multiple products}  
\label{kapusta}
Let us now give a formal algebraic definition of the sequence of products of 
$\W^{(i)}_{x_{1, i}, \ldots, x_{k_i,i}}$-spaces.  
Let $f_i$ and $g_i$ be elements of the automorphism groups 
of $V'$ (the dual space to $V$ with respect 
the bilinear pairing $\langle.,.\rangle_\lambda$, 
(cf. the Appendix)    
and generalized grading-restricted $V$-modules  
$W^{(i)}$, $1 \le i \le l$ correspondingly.
It is assumed that on each of $W^{(i)}$ 
there exist a non-degenerate bilinear pairing $\langle.,.\rangle$.  
Note that we do not consider twisted modules \cite{DL}.   
It will be dealt in a separate paper.  

Note that 
according to our assumption, 
$(x_{1, i}, \ldots, x_{k_i, i}) \in F_{k_l l}\C$, $1 \le i \le l$, i.e.,    
belong to the 
corresponding configuration space.  
 As it follows from the definition of $F_n\C$, 
 any coincidence of 
formal parameters should be excluded from the set of parameters for a product of 
$\W^{(i)}_{x_{1, i}, \ldots, x_{k_i,i}}$-spaces. 
In general, it may happend that some formal parameters 
of $(x_{1, 1}, \ldots, x_{k_1, 1}, \ldots, x_{1, l}, \ldots,  x_{k_l, l})$, $l \ge 1$,   
coincide. 
 In  the definition 
of the products below     
we keep only one of several coinciding formal parameters.  
%
Suppose in \eqref{parasha}
we have $k$ groups of coinciding formal parameters,  
$x_{j_{1, q}, i_1}= x_{j_{2, q}, i_2}= \ldots=x_{j_{s_q, q}, i_{s_q}}$, $1 \le q \le k$
$1 \le i_1 < i_2 < \ldots < i_{s_q} \le l$.    
Here $s_q$ denotes the number of coinciding parameters in $q$-th group.  
Introduce the operation $\; \widehat{}\; $ of exclusion 
of all 
$(x_{j_{2,q}, i_2}, \ldots, x_{j_{s_q, q}, i_{s_q}})$, $1 \le q \le k$,   
except of the first ones $x_{j_{1,q}, i_1}$ in each of $k$  
groups of coinciding formal parameters 
of the right hand side of \eqref{parasha}.   
Let us denote $\theta_i=k_1+\ldots+k_i$, and 
$r_i$, $1 \le i \le l$, 
the number of excluded formal parameters in \eqref{parasha}, 
and by $r=\sum_{i=1}^l r_i$ the total number of omitted parameters.  
In the whole body of the paper, we will denote by 
$(v_1, z_1; \ldots; v_{\theta_l-r}, z_{\theta_l-r})$    
the set of vertex algebra elements 
and formal parameters which excludes coinciding ones, i.e., 
\begin{eqnarray*}
 (v_1, z_1; \ldots; v_{\theta_i-r_i}, z_{\theta_i-r_i})   
\qquad \qquad \qquad \qquad \qquad \qquad \qquad \qquad \qquad \qquad \qquad \qquad &&
\nn
  \qquad =(v_{1, 1}, x_{1, i}; \ldots; 
v_{j_{1, 1}, i_1}, x_{j_{1, 1}, i_1}; 
\ldots; \widehat{v}_{j_{2, 1}, i_2},  \widehat{x}_{j_{2, 1}, i_2},  
\ldots; 
 v_{j_{s_1, 1}, i_{s_1}}, x_{j_{s_1, 1}, i_{s_1}}; \qquad \qquad && 
\end{eqnarray*}
\begin{eqnarray} 
\label{parasha}
 \qquad v_{j_{1, k}, i_1}, x_{j_{1, k}, i_1}; 
\ldots; \widehat{v}_{j_{2, k}, i_2},  \widehat{x}_{j_{2, k}, i_2},  
\ldots;  
 v_{j_{s_k, k}, i_{s_k}}, x_{j_{s_k, k}, i_{s_k}};
 v_{k_l, l}, x_{k_l, l}).   &&
\end{eqnarray}
We will require that the set of all formal parameters 
 $(z_1, \ldots, z_{\theta_l-r})$   
would belong to $F_{\theta_l-r}\C$. 
Let us introduce the new enumeration 
of elements of $v_j$ and $z_j$, $1 \le j \le \theta_l-r$. 
 Put $k_0=1$, $r_0=0$,  then set
 $n_i=\sum\limits_{s=0}^{i-1}(k_s-r_{s-1})$, $1 \le i \le l$. 
Recall the notion of an intertwining operator \eqref{wprop} $Y^W_{WV'}(w, z)$, 
 for $w\in W$, $z\in \C$ given in the Appendix.  
We use elimination of coinciding formal parameters in order to satisfy 
    the conditions for resulting configuration spaces when we multiply elements 
    that belong to subspaces of the complex we construct. We drop corresponding 
    vertex operator algebra elements simultaneously. 
    Since the whole picture of cohomology introduced through maps $\Phi$ depending 
    on elements of our vertex algebra $V$ (composibility conditions apply further 
    restrictions both on vertex algebra elements and formal parameters), the resulting 
    cohomology has already restrictions on choices of elements of $V$. As a result,             eliminations of vertex algebra elements corresponding to coinsiding formal parameters 
    does not drop information much. But we have to take that into account to avoid 
    overcounting and satisfy conditions of configuration spaces. 
\subsection{The adaptation of multiple product sequences} 
\label{nahuy}
In order to define appropriately a sequence of multiple products,
we have to introduce the operation of adaptation which we 
denote by $\mathcal R$.  
 A product (see the formula \eqref{tupoo} for the product) of individual matrix elements for  
$\Phi^{(i)}(g_i;  v_{1, i}, x_{1, i} $; $ \ldots; 
v_{k_i, i}, x_{k_i, i})  \in  \W^{(i)}_{x_{1, i}, \ldots, x_{k_i, i} }$, $1 \le i \le l$,    
results in an $\Phi$ of elements given by a matrix element 
\begin{equation}
\label{def}
\langle w', \Phi\rangle = R\langle w', \Phi\rangle,  
\end{equation} 
 with $w_i'\in W^{(i)}{}'$.  
According to \eqref{tupoo} (see below), 
 $\Phi$ is represented by a series in powers of a complex parameter, 
   we assume that (3.5) converges absolutely (on a certain domain) to a
   singular-valued rational function which we denote by $R \langle w', \Phi\rangle$. 
For elements 
$\Phi^{(i)} = \Phi^{(i)}(g_i;  v_{1, i}, x_{1, i} $; $ \ldots; 
v_{k_i, i}, x_{k_i, i})  \in  \W^{(i)}_{x_{1, i}, \ldots, x_{k_i, i} }$, 
$1 \le i \le l$,
 we assume that \eqref{def}
converges absolutely (on a certain domain) 
to a singular-valued 
 rational function which we denote by $R\langle w', \Phi\rangle$. 
In what follows, with $1 \le i \le l$, the notation 
$\left( \Phi^{(i)} \right.$ $(g_i$;  $v_{1, i}$, $x_{1, i}$; $\ldots$;    
$v_{k_i, i}$, $x_{k_i, i})$  $\in$  $\left. \W^{(i)}_{x_{1, i}, \ldots, 
 x_{k_i, i} }\right)$ 
will mean the set $(\Phi^{(1)}, \ldots, \Phi^{(l)})$.  
As we will see below, we will use this also to denote the multiples product. 

For an arbitrary element 
$\Phi \in \widehat{W}_{z_1, \ldots, z_n}$  
with the matrix element $\langle w', \Phi w\rangle$, 
let $\mathcal S$ be the operation which 
chooses a single-valued meromorphic branch of $\langle w', \Phi w\rangle$.  
Consider $L$ grading-restricted vertex operator algebra modules $W^{(i)}$, 
$1 \le i \le L$. 

For $1 \le l \le L$, $w_i'\in W^{(i)}{}'$, $u \in V_{(k)}$,  and  
$\Phi^{(i)}(g_i;  v_{1, i}, x_{1, i} $; $ \ldots; 
v_{k_i, i}, x_{k_i, i})  \in  
\W^{(i)}_{x_{1, i}, \ldots, x_{k_i, i} }$, $1 \le i \le l \le L$,   
a sequence of ordered 
$(\rho_1$, $\ldots, \rho_l)$-products, $k \in \Z$,
  is defined by 
 the meromorphic functions   
\begin{eqnarray}
\label{gendef}
 \cdot_{\rho_1, \ldots, \rho_l}  
\; \left(\Phi\left(g_i; v_{1, i}, x_{1, i}; 
\ldots; v_{k_i, i}, x_{k_i, i}\right) \right)_k  
\qquad \qquad \qquad \qquad \qquad \qquad \qquad \qquad  &&
\nn
=
 \widehat{\mathcal S}  
\prod_{i=1}^l  \rho^k_i  
 \langle w_i', Y^{W^{(i)}}_{W^{(i)}V'} \left(   
\Phi^{(i)}(g_i; v_{1, i}, x_{1, i}; \ldots;  
 v_{k_i, i}, x_{k_i, i}; u, \zeta_{1, i} ), 
\zeta_{2, i} \right)\; f_i.\overline{u} \rangle,  &&
\end{eqnarray}
extendable to a rational function 
$\Theta$ $\left(f_1 \right.$, $\ldots$, $f_l$;  
$g_1$, $\ldots$, $g_l$; $v_1$, $z_1$; $\ldots$; 
$v_{\theta_l-r}$, $z_{\theta_l-r}$;  
$\rho_1$, $\ldots$, $\rho_l$;     
$\zeta_{1, 1}$, $\zeta_{2, 1}$; $\ldots$; $\zeta_{1, l}$,  
$\left. \zeta_{2, l} \right)_k$ 
on the domain $F\C_{\sum_{i=1}^l k_i}$. 
In \eqref{gendef}  $Y^{W^{(i)}}_{W^{(i)}V'}$
 is an intertwining operator {interop} defined in the Appendix.    
Note that the order of matrix elements in the sequence of products \eqref{gendef}
is ordered with respect of the sequence of $V$-modules $W^{(i)}$.  
In \eqref{gendef}, $f_i$, $1 \le i \le l$, represents another collection of 
$V$-automorphism group elements. 
Together with automorphisms $g_i$, they constitute the whole set of 
transformations deforming matrix elements for $V$ \cite{MTZ, TZ2}. 
As we mentioned in the remark earlier,  
deformations of matrix elements are useful for cohomology descriptions. 
 The expression \eqref{gendef}
 is parametrized by $\zeta_{1, i}$, $\zeta_{2, i}  \in \C$, 
 related by the sewing relation \eqref{rhosew}. 
Here $k \in \Z$, 
$u \in V_{(k)}$ is an element of any $V_{(k)}$-basis, 
$\overline{u}$ is the dual of $u$
with respect to a non-degenerate bilinear pairing 
 $\langle .\ , . \rangle_\lambda$ over $V$ 
(see 
the Appendix).     
The construction (3.6) is quite similar to the contruction of higher genus correlation 
functions for vertex operator algebras in the Schottky geometric procedure (see, e.g.,  
\cite{T}). In that construction they associate individual matrix element to handles 
attached to Riemmann sphere in order to construct a higher genus Riemann surface. 

Here the operation $\widehat{\mathcal S}$ 
combines the adaptation operation $\mathcal S$
with the elimination of 
coinciding parameters described in Subsection \ref{kapusta}. 
The elements $u$ of a vertex algebra grading subspace $V_{(k)}$,  
their duals $\overline{u}$, as well as 
formal parameters $\zeta_{a, i}$, $a=1$, $2$, $1 \le i \le l$, 
bear implicit nature and can be incorporated into the definition 
of the bilinear pairing (see the Appendix).    
Thus we assume in what follows that the action of the transformation operators 
as well as vertex algebra operators is taken into account in the definition 
of a bilinear pairing. 
For simplicity, for a fixed set $(\rho_1, \ldots, \rho_l)$, 
let us denote the sequence of products 
depending on $l$ elements of $\Phi^{(i)}$, $1 \le i \le l$, 
$\cdot_{(\rho_1, \ldots, \rho_l)}
 (\Phi^{(1)}, \ldots, \Phi^{(l)})$  
as $(\Phi^{(1)}, \ldots, \Phi^{(l)})$.  
 Partial products with the number of parameters different to 
$L$ will be noted explicitly.  
Note that the products \eqref{gendef} are associative and additive 
by construction. 

For a fixed vertex operator algebra $V$ element $u \in V_{(k)}$, 
 the sequence \eqref{gendef} of multiple products
contains a product of matrix elements of intertwiners of $\Phi^{(i)}$
multiplied by the corresponding $k$-power of $\phi_i$.
%
In the simplest case $l=1$ of the product \eqref{gendef} 
defines another element 
$\Psi(v'_1, z_1; \ldots; v'_k, z_k)  
\in \W_{z_1, \ldots, z_k}$, 
$k \in \Z$, 
\begin{eqnarray}
\label{mangal} 
&&  \Theta\left(f_1, \ldots, f_l; 
g_1, \ldots, g_l; v_1, x_1; \ldots; v_k,  x_k; 
\rho;   
\zeta_1, \zeta_2 \right)_k    
\nn
& & \;  =   
 \mathcal S 
 \rho^k  
 \langle w', Y^W_{WV'} \left(     
\Phi \left(g; v_1, x_1;  \ldots;  
v_k, x_k; u, \zeta_1 \right), \zeta_2 \right)\; f_i.\overline{u} \rangle,   
\end{eqnarray}
Let us introduce now 
the adaptation operation $\mathcal R$ 
to recurrently define a sum 
of products for all $k \in \Z$.  
 Starting with the product \eqref{gendef}
 for some particular $k_0 \in \Z$,  
we define, for $k_0\pm 1$ 
\begin{eqnarray}
\label{tupoo}
  && 
 \left(\Phi\left(g_i; v_{1, i}, x_{1, i}; 
\ldots; v_{k_i, i}, x_{k_i, i}\right) \right)_{k_0\pm 1}   
= 
\; \left(\Phi\left(g_i; v_{1, i}, x_{1, i}; 
\ldots; v_{k_i, i}, x_{k_i, i}\right) \right)_{k_0} 
\\
\nonumber
 && \quad  + \widehat{\mathcal R}    
\prod_{i=1}^l  \rho^{k_0\pm 1}_i    
 \langle w_i', Y^{W^{(i)}}_{W^{(i)}V'} \left(   
\Phi^{(i)}(g_i; v_{1, i}, x_{1, i}; \ldots;  
 v_{k_i, i}, x_{k_i, i}; u, \zeta_{1, i} ), 
\zeta_{2, i} \right)\; f_i.\overline{u} \rangle,   
\end{eqnarray}
with $u \in V_{(k_0\pm 1)}$.    
We can then recurrently extend that to both directions for $k \in \Z$. 
Here the adaptation $\mathcal R$ is defined as the following operation. 
Since the product \eqref{gendef} contains intertwining operators 
for the corresponding grading-restricted $V$-modules $W_i$, $1 \le i \le l$,  
the dependence of the corresponding matrix elements 
 contains \cite{DL} 
rational powers of parameters  
of elements 
$\Phi^{(i)}$ $(g_i$; $v_{1, i}$, $x_{1, i}$; $\ldots$;   
 $v_{k_i, i}$, $x_{k_i, i}$; $u$, $\zeta_{1, i} )$. 
%
 Due to the rational power structure 
 it is clear that 
for a fixed $k \in \Z$,    
the action of the adaptation operation 
is it always possible to 
choose a branch of possible multiply-valued form 
$\prod_{i=1}^l$  $\rho^k_i$    
 $\langle w_i'$, $Y^{W^{(i)}}_{W^{(i)}V'}$  
$\left( \Phi^{(i)} \right.$ $(g_i$; $v_{1, i}$, $x_{1, i}$; $\ldots$;   
 $v_{k_i, i}$, $x_{k_i, i}$; $u, \zeta_{1, i} )$,  
$\left. \zeta_{2, i} \right)$\; $f_i.\overline{u} \rangle$, 
such that 
its singularities would be at a minimal distance $\epsilon(k)$, 
(such that $\lim_{k\to \pm \infty}\epsilon(k) \ne 0$), 
from singularities of the same product for $k-1$. 
%
In our particular case of the intertwining operators \cite{DL} in \eqref{gendef}
that means that we choose appropriate values of rational powers of 
the corresponding parameters. 
%
Concerning singularities of the products in \eqref{gendef}
a change of a vertex algebra $V$-element $u \in V_{(k-1)}$
to $u \in V_{(k)}$ results in a change of the rational power 
of the product dependence. 
%
By continuing the process for further $k \in \Z$, 
and applying the adaptation procedure on each step for each $k$, 
we obtain the sequence of multiple products for fixed $l$  
will always give a function with non-accumulating 
singularities with $k \to \pm \infty$. 

As a result or the recurrence procedure,  
we find the multiple product defining a rational function 
 $\left(\Phi^{(i)}\left(g_i; v_{1, i}, x_{1, i}; 
\ldots; v_{k_i, i}, x_{k_i, i}\right) \right)_{[k_1, \ldots, k_n]}$,      
for a set of multiple products \eqref{gendef} 
for several consequent values of $k \in \Z$ limited by the strip 
$[k_1, \ldots, k_n]$.  
We also define the total sequence of products \eqref{gendef} considered 
for all $k\in \Z$, 
\begin{eqnarray} 
\nonumber
 && \left(\Phi\left(g_i; v_{1, i}, x_{1, i};
 \ldots; v_{k_i, i}, x_{k_i, i}\right) \right)  
 \mapsto 
\nn
\nonumber
 &&\Theta \left(f_1, \ldots, f_l; 
g_1, \ldots, g_l; v_1, z_1; \ldots; v_{\theta_l-r}, z_{\theta_l-r}; 
\rho_1, \ldots, \rho_l;    
\zeta_{1, 1}, \zeta_{2, 1}; \ldots; \zeta_{1, l}, \zeta_{2, l} \right) 
\nn
\nonumber 
&& = \widehat{\Theta}\left(f_1, \ldots, f_l;  
g_1, \ldots, g_l; v_{1, 1}, x_{1, 1};  \ldots; 
v_{k_1, 1}, x_{k_1, 1};
 v_{1, l}, x_{1, l}; \ldots; v_{k_l, l}, x_{k_l, l};  
\right. 
\nn
\label{gendefgen} 
&& \qquad \qquad \qquad \qquad  \qquad \qquad \qquad \qquad  \left.
\rho_1, \ldots, \rho_l; \zeta_{1, 1}, \zeta_{2, 1}; 
\ldots; \zeta_{1, l}, \zeta_{2, l}\right).   
\end{eqnarray}
Recurrently continuing the construction of \eqref{tupoo} 
it is clear that \eqref{gendefgen} has meromorphic properties. 

Numerous constructions in conformal field theory \cite{FMS}, in particular, 
by constructions of  partition and correlation functions
\cite{  MT, MTZ, TZ, TZ1, TZ2,  Y} 
on higher genus Riemann surfaces, 
support the definitions \eqref{gendef}, \eqref{gendefgen}
 of the sequences of multiple products. 
The geometric nature of the genus $l$ Riemann surface sewing 
construction as a model for multiple product,
requires intertwining operators in \eqref{gendef}, \eqref{gendefgen}. 
Taking into account properties of 
the corresponding bilinear pairing defined for a vertex operator algebra $V$, 
it is natural \cite{TZ1} to associate a $V$-basis  
$\left\{u \in V_{(k)}\right\}$ and 
 complex parameters 
 $\zeta_{a, i}$, $a=1$, $2$, $1\le i \le l$, 
with the attachment of a handle to a Riemann surface. 
The attachment of   
a twisted handle to the Riemann sphere $\Sigma^{(0)}$
to form a torus $\Sigma^{(1)}$ \cite{TZ}, 
corresponds to the construction of simplest one $\rho$-parameter
product of $\W$-spaces described in Subsection \ref{geoma}, 
\eqref{mangal} in the geometric model.
The element \eqref{mangal} defines an automorphism of $\W_{z_1, \ldots, z_k}$. 
The geometric description and a reparametrization    
of the original Riemann sphere is obtained via 
the shrinking the parameter $\rho$. 

With some $\varphi$, $\kappa \in \C$ 
 related \cite{TZ} to twistings of 
attached handles in the $\rho$-sewing procedure, 
it is convenient to parametrize the automorphism group elements as  
$g_i=e^{2\pi i \varphi}$,  $f_i=e^{2\pi i \kappa}$. 
An example of the bilinear pairing $\langle .,. \rangle$ 
can be given by \eqref{def} (see also \cite{LiH}). 
The type of a vertex operator algebra $V$ 
determines 
the nature of the $V$ automorphisms group 
(see, e.g., \cite{MTZ}). 
By means of the redefinition of the bilinear pairing $\langle .,.\rangle$,     
in particular via the sewing relations \eqref{rhosew}, 
it is possible to relate ( e.g., \cite{TZ1, TZ2})   
 the sewing parameters 
$(\rho_1, \ldots, \rho_l)$ to parameters 
$\zeta_{1, i}$, $\zeta_{2, i}  \in \C$, $1 \le i \le l$.    
We will omit the $\zeta_{1, i}$, $\zeta_{2, i}$ from notations in what follows 
due to this reason. 

The construction of correlation functions  
for vertex algebras on Riemann surfaces of genus $g\ge 1$ \cite{MTZ, TZ}
inspires the forms of \eqref{gendef}, \eqref{gendefgen}. 
One would be interested in consideration of alternative forms of products such as 
multiple $\epsilon$-sewing \cite{Y} products leading to a different system of 
invariants for foliations. 
That material will be covered in a separate paper. 

Note that \eqref{gendefgen}
 does not depend on the choice of a basis of $u \in V_{(k)}$, $k \in \Z$.
by the standard reasoning \cite{FHL, Zhu}.   
In the case when the forms $\Phi^{(i)}$, $1 \le i \le l$, 
that we multiply    
do not contain $V$-elements,  
  \eqref{gendef} defines the following products
 $\cdot_{\rho_1, \ldots, \rho_l} \left(\Phi^{(i)} \right)$   
\begin{eqnarray}
\label{pipas} 
&& \Theta(f_1, \ldots, f_l;  g_1, \ldots, g_l; \rho_1, \ldots, \rho_l;  
\zeta_{1, 1}, \zeta_{2, 1}; \ldots; \zeta_{1, l}, \zeta_{2, l} )_k   
\nn
&&
\qquad =    
\prod_{i=1}^l  
 \rho_i^k  
\langle w_i', Y^{W^{(i)}}_{W^{(i)}V'}\left(   
\Phi^{(i)}(g_i; u, \zeta_{1, i}), \zeta_{2, i} \right) \; f_i.\overline{u} \rangle.    
\end{eqnarray}  
The right hand side of \eqref{gendefgen}  
is given by a 
formal series of bilinear pairings summed over a vertex algebra basis. 
To complete this definition we have to show 
that a differential form 
that belongs to the space $\W^{(1, \ldots, l)}_{z_1, \ldots, z_{\theta_l-r}}$ 
is defined by  
the right hand side of \eqref{gendefgen}. 
%
As parameters for 
elements of $\W^{(i)}$-spaces, 
we could take $\zeta_{1, i}$   
in \eqref{gendef}, \eqref{gendefgen}. 
%
Note that due to \eqref{wprop}  
it is assumed that  
$\Phi^{(i)}(g_i; v_{1, 1}, x_{1, 1}$;  $\ldots$;     
  $v_{k_i, i}, x_{k_i, i}; u, \zeta_{1, i})$   
are adapted transversal to 
 the grading-restricted generalized
$V$-module $W^{(j)}$,  $1 \le j \le l$,   
 vertex operators 
$Y_{W^{(j)}}(u, -\zeta_{1, j})$.   
 (cf. Subsection \ref{composable}).     
The products \eqref{gendefgen} are actually defined by the sum of 
products of matrix elements of generalized grading-restricted 
$V$-modules $W^{(i)}$, $1 \le i \le l$.  
The parameters $\zeta_{1, i}$ and $\zeta_{2, i}$ satisfy \eqref{rhosew}.
The vertex algebra elements  
$u \in V$ and $\overline{u} \in V'$ 
are related by the bilinear pairing. 
In terms of the theory of correlation functions for 
vertex operator algebras \cite{FMS, Zhu},  
the form of the sequences of multiple 
 products defined above is a natural one.    
\subsection{The product of $\W$-spaces} 
 The main statement of this Section is given by 
\begin{proposition}
\label{glavna}
For $l \ge 1$, elements of the spaces 
$\W^{(1)}_{x_{1, 1}, \ldots, x_{k_1, 1}}$, $\ldots$,  
  $\W^{(l)}_{x_{1, l}, \ldots, x_{k_l, l}}$   
such that the products defined by \eqref{gendefgen} are given by converging expressions, 
define the correspond to maps  
$\cdot_{\rho_1, \ldots, \rho_l}:  
\W^{(1)}_{x_{1, 1}, \ldots, x_{k_1, 1}} \times
 \ldots \times \W^{(l)}_{x_{1, l}, \ldots, x_{k_l, l}}   
\rightarrow \W^{(1, \ldots, \l)}_{z_1, \ldots, z_{\theta_l-r}}$,    
where 
$\W^{(1, \ldots, \l)}_{z_1, \ldots, z_{\theta_l-r}}
= \bigotimes_{i=1}^l \widehat{\W^{(i)}}_{x_{1, i}, \ldots, x_{k_i, i}}$.     
\end{proposition}   
The rest of this Section is devoted to the proof of Proposition \ref{glavna}. 
Under conditions stated in Proposition \ref{glavna},  
we show that the right hand side of \eqref{gendef}, \eqref{gendefgen}    
belongs to the space $\W^{(1, \ldots, l)}_{z_1, \ldots, z_{\theta_l-r}}$.  
In the view of Proposition \ref{glavna},  
let us denote by  
$\Phi^{(1, \ldots, l)}$ 
 an element of the tensor product $\W^{(1, \ldots, l)}$-valued function  
which would correspond to a rational function 
\begin{eqnarray*}
&&
  \Theta(f_1, \ldots, f_l; g_1, \ldots, g_l;
 v_1, z_1; \ldots; v_{\theta_l-r}, z_{\theta_l-r} )_k  
\nn
&&
\qquad 
=\langle w_i', \Phi^{(1, \ldots, l)} 
(f_1, \ldots, f_l; g_1, \ldots, g_l; v_1, z_1; \ldots; v_n, z_n) \rangle,
\end{eqnarray*}  
obtained as a result of the product \eqref{gendef}.
 
For more general situation discussing 
convergence and well-behavior problem for products of 
of the classical coboundary operators,  
the main approach is the construction of differential equations 
 that products and approximations by using Jacobi identity.
 For the ordinary cohomology theory of grading-restricted vertex algebras,
 such techniques do not work because cochains do 
not satisfy Jacobi identity. 
%
 We will apply the general constructions of \cite{H2, Gui} to 
study properties of products of coboundary operators in another paper.    
\subsection{Convergence of multiple products sequences}   
\label{pilito} 
In \cite{Gui} it was established that the correlation functions for a 
 $C_2$-cofinite vertex operator algebra of conformal field theory type 
are absolutely and locally uniformly convergent
 on the sewing domain since it is a multiple
sewing of correlation functions associated with genus zero conformal blocks.  
In this paper we give an alternative proof 
though one can use the results of \cite{Gui} 
 to prove Proposition \eqref{derga}.  
%
A $\W^{(i)}_{x_{1, i}, \ldots, x_{k_i, i}}$-space 
is defined via of matrix elements of the form \eqref{def}. 
This corresponds \cite{FHL} 
 to matrix element of a number of a vertex algebra $V$-vertex operators 
with formal parameters identified with local coordinates on the Riemann sphere. 
The product of $l$ $\W^{(i)}_{x_{1, i}, \ldots, x_{k_i, i}}$-spaces  
can be geometrically associated with 
 a genus $l \ge 0$ Riemann surface $\Sigma^{(l)}$  
 with a few marked points with 
local coordinates vanishing at these points  
\cite{H2}. 
The center of an annulus used in order 
to sew another handle to a Riemann surface  
is identified with an additional point.    
We have then a geometric interpretation for 
the products \eqref{gendef}, \eqref{gendefgen}.    
A genus $l$ Riemann surface $\Sigma^{(l)}$   
formed in the multiple-sewing procedure represents  
the resulting model space. 
%
  Matrix elements for a number of vertex operators 
are usually associated \cite{FHL, FMS} 
with a vertex algebra correlation functions on the sphere. 
Let us extrapolate this notion to the case of 
$\W^{(i)}_{x_{1, i}, \ldots, x_{k_i, i}}$-spaces, $1 \le i \le l$.   
 We  use the $\rho$-sewing procedure  
for the Riemann surface  
 with attached handles 
in order to supply an appropriate geometric construction of the products 
to obtain a matrix element   
associated with the definition of the multiple products \eqref{gendef}, \eqref{gendefgen}.   

Similar to \cite{H2, Huang0, H20, Y, Zhu,  FMS, BZF}
let us identify local coordinates of the corresponding sets of points 
on the resulting model genus $l$ Riemann surface 
with the sets $(x_{1, i}, \ldots, x_{k_i, i})$, $1 \le i \le l$ of   
complex formal parameters.  
The roles of coordinates \eqref{disk} of 
the annuluses \eqref{zhopki} 
can by played by the complex parameters $\zeta_{1, i}$ and $\zeta_{2, i}$ 
of \eqref{gendef}, \eqref{gendefgen}. 
 Several groups of coinciding coordinates may occur
on identification of annuluses $\mathcal A_{a, i}$ and $\mathcal A_{\overline{a}, i}$.  
As a result of the $(\rho_1, \ldots, \rho_l)$-parameter sewing \cite{Y}, 
the sequence of products \eqref{gendef}, \eqref{gendefgen}  
describes a differential form that belongs to the space $\W^{(1, \ldots, l)}$ defined  
on a genus $l$ Riemann surface $\Sigma^{(l)}$.      
Since $l$ initial spaces $\W^{(i)}_{x_{1, i}, \ldots, x_{k_i, i}}$ contain  
 $\widetilde{W^{(i)}}_{x_{1,i}, \ldots, x_{k_i, i}}$-valued differential forms   
expressed by 
 matrix elements of the form \eqref{def},  
it is then proved (see Proposition \ref{derga} below), 
that the resulting products define 
elements of the space 
$\W^{(1, \ldots, l)}_{z_1, \ldots, z_{\theta_l-r}}$   
 by means of absolute convergent matrix elements
 on the resulting genus $l$ Riemann surface.    
 The sequences of multiple products of  
$\W^{(i)}_{x_{1, i}, \ldots, x_{k_i, i}}$-spaces 
as well as the moduli space of the resulting genus $l$ 
 Riemann surface $\Sigma^{(l)}$ are described by 
the complex sewing parameters $(\rho_1, \ldots, \rho_l)$.  
\begin{proposition}
\label{derga}
The total sequence of products \eqref{gendef}, \eqref{gendefgen}  
of elements of the spaces
 $\W^{(i)}_{x_{1, i}, \ldots, x_{k_i, i}}$, $1 \le i \le l$,  
 corresponds to rational functions  
absolutely converging in all complex parameters $(\rho_1, \ldots, \rho_l)$     
with only possible poles at $x_{j, m'}=x_{j', m''}$,  
  $1 \le j\le k_{m'}$, $1 \le j' \le k_{m''}$, $1 \le m'$,  $m'' \le l$, $l \ge 1$.  
\end{proposition}
\begin{proof}
\end{proof}
\subsection{Symmetry properties}
\label{symprop}
Let us assume that $g_i$, $f_i$ commute with $\sigma(i) \in S_l$, $l \ge 1$.

 The action of an element $\sigma \in S_{\theta_l-r}$ 
on the sequence of products of 
$\Phi^{(i)}$ $(g_i$; $v_{1,1}, x_{1,1}$;  $\ldots; 
v_{k_i, i}, x_{k_i, i}) \in \W^{(i)}_{x_{1,1}, \ldots, x_{k_i, i}}$, 
$l \ge 1$, is defined as  
\begin{eqnarray}
\label{bardos} 
&& \sigma(\Theta) 
\left(f_1, \ldots, f_l; g_1, \ldots, g_l;
v_1, z_1; \ldots; v_{\theta_l-r}, z_{\theta_l-r}; \rho_1, \ldots, \rho_l \right)_k  
\\
\nonumber 
  && =  \Theta \left(f_1, \ldots, f_l; g_1, \ldots, g_l; v_{\sigma(1)}, z_{\sigma(1)}; 
 \ldots; v_{\sigma(\theta_l-r)},  
z_{\sigma (\theta_l-r)}; \rho_1, \ldots, \rho_l\right)_k, 
\end{eqnarray}
and the total multiple product \eqref{gendefgen} correspondingly. 

Note that \eqref{bardos} assumes that $\sigma \in S_{\theta_l-r}$ does not 
act on $\zeta_{a, i}$, $a=1$, $2$, $1\le i \le l$ in the products \eqref{gendef}, 
\eqref{gendefgen}.   
The results of this Section 
below extend to corresponding total multiple products. 
%
Next, we prove 
\begin{lemma}
\label{tarusa}
The products \eqref{gendef}, \eqref{gendefgen} satisfy  
 \eqref{shushu} for $\sigma \in S_{\theta_l-r}$, i.e., 
\begin{eqnarray*} 
&& \sum_{\sigma\in J_{\theta_l-r; s}^{-1} } (-1)^{|\sigma|}
  \Theta\left(f_1, \ldots, f_l; g_1, \ldots, g_l; \right. 
\nn
&& \qquad \qquad \qquad  \left. 
v_{\sigma(1)}, z_{\sigma(1)}; \ldots; v_{\sigma(\theta_l-r)}, z_{\sigma(\theta_l-r)};   
 \rho_1, \ldots, \rho_l \right)_k 
=0. 
\end{eqnarray*}
\end{lemma}
\begin{proof}
 For arbitrary $w'_i \in W^{(i)}{}'$, $1 \le i \le l$, 
\begin{eqnarray*}
 \sum_{\sigma\in J_{\theta_l-r; s}^{-1}}(-1)^{|\sigma|} 
  \Theta \left(f_1, \ldots, f_l; g_1, \ldots, g_l;
v_{\sigma(1)}, z_{\sigma(1)}; \ldots; v_{\sigma(\theta_l-r)}, z_{\sigma(\theta_l-r)};  
 \rho_1, \ldots, \rho_l \right)_k \qquad \qquad \qquad \qquad  &&  
\nn
=
\sum_{\sigma\in J_{\theta_l-r; s}^{-1}}(-1)^{|\sigma|}  \;    
 \mathcal R \prod_{i=1}^l
 \rho^k_i 
 \langle w'_i, Y^{W^{(i)}}_{W^{(i)}V'} 
\left(  
\Phi^{(i)}(g_i; 
v_{\sigma(n_i+1)}, z_{\sigma(n_i+1)}; \ldots; \right.  \qquad \qquad \qquad \qquad &&
\nn
\left. v_{\sigma(n_i+k_i-r_i)}, z_{\sigma(n_i+k_i-r_i)};     
u, \zeta_{1, i}), 
  \zeta_{2, i} \right)\; f_i.\overline{u} \rangle \qquad \qquad  \qquad \qquad &&
\nn
= \sum_{\sigma\in J_{\theta_l-r; s}^{-1}}(-1)^{|\sigma|}    
 \mathcal R \prod_{i=1}^l 
  \langle w'_i, e^{\zeta_{2, i} L_{W^{(i)}}(-1)} \; Y_{W^{(i)}}(f_i.\overline{u}, -\zeta_{2, i}) \; 
 \qquad \qquad  \qquad \qquad  \qquad \qquad && 
\nn
 \Phi^{(i)}(g_i; 
v_{\sigma(n_i+1)}, z_{\sigma(n_i+1)}; 
\ldots; v_{\sigma(n_i+k_i-r_i)}, z_{\sigma(n_i+k_i-r_i)}; 
u, \zeta_{1, i})   
\rangle.   \qquad \qquad \qquad \qquad &&
\end{eqnarray*}
We obtain 
for an element $\sigma \in S_{\theta_l-r}$ 
 inserted inside the intertwining operator 
\begin{eqnarray*}  
 \mathcal R\prod_{i=1}^l
 \rho^k_i 
  \langle w'_i, e^{\zeta_{2, i} L_{W^{(i)}}(-1)} 
\; Y_{W^{(i)}}(f_i. \overline{u}, -\zeta_{2, i})  
 \qquad \qquad  \qquad \qquad \qquad \qquad 
\qquad \qquad \qquad \qquad \qquad \qquad && 
\nn
\sum_{\sigma\in J_{k_i-r_i; s}^{-1}}(-1)^{|\sigma|} 
\Phi^{(i)}(g_i;
v_{\sigma(n_i+1)}, z_{\sigma(n_i+1)}; \ldots; 
v_{\sigma(n_i+k_i-r_i)}, z_{\sigma(n_i+k_i-r_i)}; 
u, \zeta_{1, i})    
\rangle=0, \qquad \qquad \qquad \qquad \qquad  && 
\end{eqnarray*}
since,
$J^{-1}_{\theta_l-r; s}= J^{-1}_{k_1-r_1;s} \times \ldots \times  J^{-1}_{k_l-r_l;s}$,    
and due to the fact that 
$\Phi^{(i)}(g_i$; $v_{1,1}, x_{1, 1}$ ; $\ldots$;     
  $v_{k_1, 1}, x_{k_1, 1}$;  
$v_{1, i},  x_{1, i}$; $\ldots$; $v_{k_i, i}, x_{k_i, i}$;     
 $u, \zeta_{1, i})$  
satisfy \eqref{sigmaction}.   
\end{proof}
\subsection{The existence, $L_V(-1)$-derivative, and $L_V(0)$-conjugation properties}
\label{exderconj}
In this subsection 
 we prove the existence of an appropriate differential form that belongs to 
$\W^{(1, \ldots, l)}_{z_1, \ldots, z_{\theta_l-r}}$     
corresponding to an absolute convergent  
 $\Theta(f_1, \ldots, f_l$; $g_1, \ldots, g_l$;  
 $v_1, z_1$; $\ldots$; $v_{\theta_l-r}$,  $z_{\theta_l-r})$   
defining the $(\rho_1, \ldots, \rho_l)$-product of elements of 
the spaces 
$\W^{(i)}_{ x_{1, 1}, \ldots, x_{k_i, i}}$. 
%
\begin{lemma}
\label{baskal}
For all choices of sets of elements of the spaces 
$\W^{(i)}_{x_{1,i}, \ldots, x_{k_i, i}}$, 
$1 \le i \le l$,  
there exists a differential form characterized by the element 
 $\Theta(f_1, \ldots, f_l;  g_1, \ldots, g_l$;   
 $v_1, z_1; \ldots; v_{\theta_l-r}, z_{\theta_l-r}$; $\rho_1$, $\ldots, \rho_l)_k    
\in \W^{(1, \ldots, l)}_{z_1, \ldots, z_{\theta_l-r}}$ 
such that the product \eqref{gendefgen}  
converges to a rational function  
\begin{eqnarray*}
&&
R(v_1, z_1; \ldots; v_{\theta_l-r}, z_{\theta_l-r}; 
 \rho_1, \ldots, \rho_l)  
\nn
&&
\qquad =   \Theta(f_1, \ldots, f_l;  g_1, \ldots, g_l; 
v_1, z_1; \ldots; v_{\theta_l-r}, z_{\theta_l-r}; \rho_1, \ldots, \rho_l)_k.    
\end{eqnarray*} 
\end{lemma}
The action of  
 $\partial_s=\partial_{z_s}={\partial}/{\partial_{z_s}}$, $1\le s \le \theta_l-r$, 
on $\widehat{\Theta}$ is defined as 
\begin{eqnarray*}
 \partial_s  \Theta 
 (f_1, \ldots, f_l; g_1, \ldots, g_l; 
v_1, z_1;   \ldots;  v_{\theta_l},  z_{\theta_l} 
   \rho_1, \ldots, \rho_l)_k    
\qquad \qquad \qquad \qquad \qquad \qquad && 
\nn
=     
 \mathcal R \prod_{i=1}^l
\rho^k_i
 \langle w'_i, \partial_s 
Y^{W^{(i)}}_{W^{(i)}V'} 
\left(  
\Phi^{(i)}(g_i; v_{n_i+1},  z_{n_i+1};  \ldots; \right. 
\qquad \qquad \qquad \qquad &&
\nn
\left.  v_{n_i+k_i-r_i}, z_{n_i+k_i-r_i};   
 u, \zeta_{1, i}), 
  \zeta_{2, i} \right) f_i.\overline{u} \rangle. \qquad  && 
\end{eqnarray*}
\begin{proposition}
\label{katas1}
The products \eqref{gendef}, \eqref{gendefgen} satisfy  
the properties \eqref{lder1} and \eqref{loconj}.  
\end{proposition}
\begin{proof}
 By using \eqref{lder1} for 
$\Phi^{(i)}(g_i; v_{1,i},  x_{1,i};  \ldots; v_{k_i,i}, x_{k_i,i})$ 
 we consider 
\begin{eqnarray}
\label{ohaio} 
&& \partial_s 
   \Theta(f_1, \ldots, f_l; g_1, \ldots, g_l; 
v_1, z_1; \ldots; v_{\theta_l-r}, z_{\theta_l-r};
  \rho_1, \ldots, \rho_l)_k \qquad \qquad \qquad 
\\
&&\qquad  = 
\mathcal R    
 \prod_{i=1}^l
\rho^k_i 
\langle w'_i, \partial_s 
   \left( e^{\zeta_{2, i}  L_{W^{(i)}}(-1)} Y_{W^{(i)}} 
\left( f_i.\overline{u}, - \zeta_{2, i} \right)  \right. 
\nn
\nonumber
&&  \qquad \qquad \qquad
\left.  
 \Phi^{(i)}(g_i;
v_{n_i+1},  z_{n_i+1};  \ldots;  v_{n_i+k_i-r_i}, z_{n_i+k_i-r_i};  u, \zeta_{1, i})\right) 
\rangle   
\end{eqnarray}
\begin{eqnarray*}
 =  
 \mathcal R \prod_{i=1}^l
 \rho^k_i
 \langle w', 
Y^{W^{(i)}}_{W^{(i)}V'}  
\left( \sum\limits_{j=1}^{k_i-r_i} \partial_s^{\delta_{s,j}}
\Phi^{(i)}\left(g_i; 
 v_{n_i+1},  z_{n_i+1};  \ldots; \right. \right. &&
\nn
 \left. \left. \qquad \qquad v_{n_i+k_i-r_i}, z_{n_i+k_i-r_i};  
 u, \zeta_{1, i}\right), 
  \zeta_{2, i} \right) f_i.\overline{u} \rangle && 
\end{eqnarray*}
\begin{eqnarray*}
&& =   
 \mathcal R \prod_{i=1}^l
 \rho^k_i
 \langle w', 
Y^{W^{(i)}}_{W^{(i)}V'} 
\left(   \sum\limits_{j=1}^{k_i-r_i} 
\Phi^{(i)}(g_i; v_{n_i+1},  z_{n_i+1};  \ldots;  \right. 
\nn
&&
\qquad  \left. \left(L_V{(-1)}\right)^{\delta_{s,j}}.v_s, x_s;  
\ldots; v_{n_i+k_i-r_i}, z_{n_i+k_i-r_i};
 u, \zeta_{1, i}, 
  \zeta_{2, i} \right) f_i.\overline{u} \rangle  
\end{eqnarray*}
\begin{eqnarray*}
=    
  \Theta (f_1, \ldots, f_l;  
g_1, \ldots, g_l;  
v_1, z_1;  \ldots;  \left(L_V{(-1)}\right)_s;  \ldots;  v_{\theta_l-r},  z_{\theta_l-r};  
\rho_1, \ldots, \rho_l)_k. &&    
\end{eqnarray*}
By summing over $s$ we obtain
\begin{eqnarray*}
\sum\limits_{s=1}^{\theta_l-r} \partial_s   
 \Theta (f_1, \ldots, f_l; g_1, \ldots, g_l;  
v_1, z_1;  \ldots;  v_{\theta_l-r},  z_{\theta_l-r}; \rho_1, \ldots, \rho_l)_k
\qquad \qquad \qquad 
&&
\nn
 = 
\sum\limits_{s=1}^{\theta_l-r} 
   \Theta (f_1, \ldots, f_l; g_1, \ldots, g_l; 
v_1, z_1;  \ldots;  \left(L_V{(-1)}\right)_s;  \ldots;  v_{\theta_l-r},  z_{\theta_l-r}
; \rho_1, \ldots, \rho_l)_k &&
\nn
= L_{W^{(i)}}{(-1)}.   \Theta(f_1, \ldots, f_l; g_1, \ldots, g_l;  
v_1, z_1;  \ldots;  v_{\theta_l-r},  z_{\theta_l-r}; \rho_1, \ldots, \rho_l)_k. &&
\end{eqnarray*}  
\end{proof}
We define also 
\begin{eqnarray}
\label{musaka}
  \widehat{\Theta} 
 \left(y_1^{  L_{W^{(1)}}  (0)}, \ldots, y_l^{  L_{W^{(l)}}  (0)}; 
 f_1, \ldots, f_l; g_1, \ldots, g_l; \right. \qquad \qquad \qquad &&
\nn
\left. 
 v_1, z_1;  \ldots;  v_{\theta_l-r}, x_{\theta_l-r}; \rho_1, \ldots, \rho_l \right)_k 
\quad &&
\nn
 =
 \mathcal R \prod_{i=1}^l
\rho^k_i   \langle w_i, 
 \Phi^{(i)} \left(g_i;  y_i^{ L_{W^{(i)}}  (0)} v_{n_i+1},  y_i\; z_{n_i+1};     
\ldots;  \right. \qquad \qquad  &&
\nn
 \left. 
 y_i^{ L_{W^{(i)} }  (0)} v_{n_i+k_i-r_i}, y_i \; z_{n_i+k_i-r_i};   
  u,  \zeta_{1, i}, 
  \zeta_{2, i} \right) f_i.\overline{u} \rangle. \; &&
\end{eqnarray} 
\begin{proposition}
\label{katas2}
The products \eqref{gendef}, \eqref{gendefgen}
  satisfy 
the properties \eqref{loconj}.  
\end{proposition}
\begin{proof}
For $y_i \ne 0$, $1\le i \le l$,  
 due to \eqref{loconj} 
  and \eqref{aprop},  
\begin{eqnarray*}
&& 
 \widehat{\Theta} (f_1, \ldots, f_l; g_1, \ldots, g_l;  
y_1^{L_V(0)} v_1, y_1 \; z_1;  \ldots;  y_l^{L_V(0)} v_{\theta_l-r},  y_l \;  x_{\theta_l-r, l};    
 \rho_1, \ldots, \rho_l)_k 
\nn
&& =  
 \widehat{\mathcal R} 
 \prod_{i=1}^l
\rho^k_i 
 \langle w'_i, Y^{W^{(i)}}_{W^{(i)}V'}  
\left(   
\Phi^{(i)}(g_i; y_i^{L_V{(0)} }  v_{n_i+1},  y_i\; z_{n_i+1}; \ldots;
\right. 
\nn
&&
\left. \qquad \qquad \qquad  \qquad \qquad 
 y_i^{L_V{(0)} }  v_{n_i+k_i-r_i}, y_i\; z_{n_i+k_i-r_i};   
  u,  \zeta_{1, i}), 
  \zeta_{2, i} \right) f_i.\overline{u} \rangle
\nn
&&
=\widehat{\Theta} \left(y_1^{  L_{W^{(1)}}  (0)}, \ldots, y_l^{  L_{W^{(l)}}  (0)}; 
 f_1, \ldots, f_l; g_1, \ldots, g_l; \right. 
\nn
&&
\qquad \qquad \qquad \qquad \qquad \qquad \qquad \left. 
 v_1, z_1;  \ldots;  v_{\theta_l-r}, x_{\theta_l-r}; \rho_1, \ldots, \rho_l\right)_k. 
\end{eqnarray*}
\end{proof}
As an upshot, we obtain the proof of Proposition \ref{glavna} 
by taking into account 
the results of Proposition \eqref{derga},  
Lemma \eqref{tarusa}, Lemma \eqref{baskal}, and Proposition \eqref{katas2}.  
\subsection{Canonical properties of the $\W$-products} 
\label{properties1}
In this Subsection we study  
  properties of  
the products   
$\Theta$ $(f_1$, $\ldots$, $f_l$; $g_1$, $\ldots$, $g_l$;    
$v_1$, $z_1$; $\ldots$; $v_{\theta_l-r}$, $z_{\theta_l-r}$;   
 $\rho_1$, $\ldots$, $\rho_l)_k$  of 
\eqref{gendef}, \eqref{gendefgen}
with respect of changing of formal parameters.  
\begin{proposition}
\label{pupa}
 Under the action  
$(\varrho(z_1), \ldots, \varrho(z_{\theta_l-r}))$
of the group  
$\left( {\rm Aut} \; \Oo\right)^{\times (\theta_l-r)}_{z_1, \ldots, z_{\theta_l-r}}$    
of independent  
$\theta_l-r$-dimensional   
changes of formal parameters 
\begin{eqnarray}
\label{zwrhokl}
&&(z_1, \ldots, z_{\theta_l-r})   
\mapsto 
(\widetilde{z}_1, \ldots, \widetilde{z}_{\theta_l-r})   
= 
(\varrho(z_1), \ldots, \varrho(z_{\theta_l-r})). 
\end{eqnarray}   
the products \eqref{gendef}, \eqref{gendefgen}  
are canonical 
for generic elements $v_j \in V$, $1 \le j \le \theta_l-r$, $l \ge 1$,  
  of a quasi-conformal grading-restricted vertex algebra $V$. 
\end{proposition}
\begin{proof}
Due to Proposition \ref{pupa0},  
\begin{eqnarray*}
&&\Phi^{(i)}(g_i; v_{n_i+1}, \widetilde{z}_{n_i+1}; 
 \ldots; v_{n_i+k_i-r_i}, \widetilde{z}_{n_i+k_i-r_i})  
\nn
&&\qquad =\Phi^{(i)}(g_i; v_{n_i+1}, z_{n_i+1}; 
 \ldots; v_{n_i+k_i-r_i, i}, z_{n_i+k_i-r_i}).  
\end{eqnarray*}
\begin{eqnarray*}
&& \Theta(f_1, \ldots, f_l; g_1, \ldots, g_l; 
 v_1,  \widetilde{z}_1; \ldots; v_{\theta_l-r}, 
\widetilde{z}_{\theta_l-r}; \rho_1, \ldots, \rho_l)_k     
\nn
 && =    
 \mathcal R \prod_{i=1}^l \rho^k_i 
 \langle w'_i,  Y^{W^{(i)}}_{W^{(i)}V'}  
\left(  
\Phi^{(i)}(g_i; v_{n_i+1},   \widetilde{z}_{n_i+1};  
\ldots; \right. 
\nn
&&
\left. \qquad \qquad \qquad \qquad \qquad 
v_{n_i+k_i-r_i},  \widetilde{z}_{n_i+k_i-r_i};   
  u,  \zeta_{1, i}),  
   \zeta_{2, i} \right) f_i.\overline{u} \rangle 
\end{eqnarray*}
\begin{eqnarray*}
 && =  
\mathcal R \prod_{i=1}^l \rho^k_i 
 \langle w'_i,  Y^{W^{(i)}}_{W^{(i)}V'}  
\left(  
\Phi(g_i; v_{n_i+1,i},   z_{n_i+1};  
\ldots;
\right. 
\nn
&&
\left. \qquad \qquad \qquad \qquad \qquad
v_{n_i+k_i-r_i},  z_{n_i+k_i-r_i};   
  u,  \zeta_{1, i}),  
   \zeta_{2, i} \right) f_i.\overline{u} \rangle 
\nn
&& 
=  \Theta(f_1, \ldots, f_l; g_1, \ldots, g_l;   
 v_1,  z_1; \ldots; v_{\theta_l-r}, z_{\theta_l}; \rho_1, \ldots, \rho_l).   
\end{eqnarray*}
 The products \eqref{gendef}, \eqref{gendefgen} 
are therefore invariant under \eqref{zwrho}. 
\end{proof}
\section{Spaces for families of complexes} 
\label{spaces}
In this Section we introduce the definition of spaces  
for the families of complexes 
associated to a grading-restricted vertex algebra $V$
 $V$-modules suitable for the construction of 
a codimension one foliation  
cohomology defined on a complex curve.  
 Several grading-restricted generalized modules $W^{(i)}$ 
as well as the corresponding spaces $\W^{(i)}_{x_{1, i}, \ldots, x_{k_i, i}}$ 
are involved in the constructions of this paper. 

 Consider a configuration of $2l$ sets of vertex algebra $V$ elements, 
$(v_{1, i}, \ldots, v_{k_i, i})$,  
  $(v'_{1, i}, \ldots, v'_{m_i, i})$, 
$1 \le i \le l$, 
  and points 
 $(p_{1, i}, \ldots, p_{k_i, i})$, 
$(p'_{1, i}, \ldots, p'_{m_i, i})$,   
 with the local coordinates 
$(c_{1, i}(p_{1, i}), \ldots, c_{k_i, i}(p_{k_i, i}))$ 
$(c_{1, i}(p'_{1, i}), \ldots, c_{k_i, i}(p'_{m_i, i}))$ 
 taken on the intersection of the $i$-th leaf of the leaves space $M/\F$ 
with the $j$-th transversal section $U_j \in \U$, $j\ge 1$, 
of a foliation $\F$ transversal basis $\U$ on a complex curve. 
%
Denote by $C^{k_i}_{(m_i)}\left(V, \W^{(i)}, \F\right)$ $(U_{p, i})$, 
$0 \le p \le m_i$, $k_i \ge 1$, $m_i \ge 0$,         
the space of all linear maps \eqref{maps}. 
 $\Phi: V^{\otimes k_i } \rightarrow 
\W^{(i)}_{ c_{1, i}(p_{1, i}), \ldots, c_{k_i, i}(p_{k_i, i}); \atop
c_{1, i}(p'_{1, i}), \ldots, c_{k_i, i}(p'_{m_i, i})}$,     
adapted transversal to 
 $m_i$ of vertex operators \eqref{poper} 
equipped with the formal parameters identified 
with the local coordinates $c'_{j, i}(p'_{j, i})$ around the points $p'_{j, i}$  
on each of the transversal sections $U_j$, $1 \le j \le m_i$.   

We assume that each section of a transversal basis  
$\U$ has a coordinate chart 
 induced by a coordinate chart of $M$ \cite{CM}.   
%
A holonomy embedding 
maps  
a coordinate chart on the first section 
into a coordinate chart on the second transversal section, 
and a section into another section of a transversal basis.  
Let us now introduce the following spaces for the families of complexes
associated with grading-restricted generalized $V$-modules.  
This definition is motivated by 
the definition of the spaces for ${\rm \check C}$ech-de Rham complex in \cite{CM}. 

For $k_i\ge 0$, $m_i \ge 0$,   
 introduce the spaces   
\begin{equation}
\label{ourbicomplex}
 C^{k_i}_{m_i}\left(V, \W^{(i)}, \U, \F\right) =  \bigcap_{ 
U_1 \stackrel{h_1, i} {\hookrightarrow }  
\ldots \stackrel {h_{p-1, i}}{\hookrightarrow } U_{p,i},    
\;  1 \le p \le m_i }  
 C^{k_i}_{(m_i)}\left(V, \W^{(i)}, \F\right) (U_{p, i}),    
\end{equation}
where the intersection ranges over all possible $(p-1, i)$-tuples 
of holonomy embeddings $h_{p, i}$, $1 \le p \le  m_i-1$,  
between transversal sections of a basis $\U$  for $\F$. 

We skip $\F$ from further  
notations of complexes 
since a foliation $\F$ is fixed in our considerations. 
\subsection{Properties of spaces for families of complexes} 
\label{noprop}
In \cite{zucmp} we have proven the following facts about 
spaces for families of vertex algebra complexes for foliations.  
The spaces \eqref{ourbicomplex} are non only zero spaces. 
The family \eqref{ourbicomplex} is the transversal basis $\U$ independent.   
According to that, 
 we will denote  $C^{k_i}_{m_i}\left(V, \W^{(i)}, \U, \F\right)$ as 
$C^{k_i}_{m_i}\left(V, \W^{(i)} \right)$ in what follows. 
In the Appendix the definition of 
 a quasi-conformal 
grading-restricted vertex algebra is given.   
The following Proposition was proven in \cite{zucmp}.  
The construction \eqref{ourbicomplex} is canonical, i.e., 
 does not depend on the foliation preserving choice of local coordinates on $M/\F$        
for a quasi-conformal grading-restricted vertex algebra $V$ 
and its grading-restricted generalized  
modules $W^{(i)}$, $1 \le i \le l$.    

In what follows,  
we will always assume the quasi-conformality \cite{BZF} 
of $V$ for the spaces \eqref{ourbicomplex}.
The condition is necessary in the proof of elements invariance of the spaces   
$\W^{(i)}_{z_{1, i}, \ldots, z_{k_i, i}}$, $1 \le i \le l$,  
with respect to a vertex algebraic representation 
(cf. the Appendix) of the group 
$\left({\rm Aut}\; \Oo\right)^{\times k_i}_{z_{1, i}, \ldots, z_{k_i, i} }$.  

Let $W^{(i)}$, $1 \le i \le l$ be a set of grading-restricted generalized $V$ modules.  
Due to the definition 
of the adapted transversal,
 with $k_i=0$ the maps $\Phi^{(i)}$ do not include variables. 
 Let us set 
$C_{m_i}^0\left(V, \W^{(i)}\right)= W^{(i)}$,     
for $m_i\ge 0$. 
 According to the definition, 
such mappings 
are assumed to be adapted transversal 
 to a number of vertex operators depending on local coordinates 
of $m_i$ points on $m_i$ transversal sections.  
%
In \cite{zucmp} we proved that  
\begin{equation}
\label{susus}
C_{m_i}^{k_i}\left(V, \W^{(i)}\right)\subset C_{m_i-1}^{k_i}\left(V, \W^{(i)}\right).       
\end{equation}
\subsection{Connections as coboundary operators}
\label{cobcob}
In this Subsection we introduce the coboundary operators acting on
 the families of spaces \eqref{ourbicomplex}.  
Consider the vector of $E$-operators: 
\begin{eqnarray}
\label{mathe}
\mathcal E^{(i)} &=&  \left( E^{(1)}_{W^{(i)}}.,\; 
\sum\limits_{j=1}^n (-1)^j \; E^{(2)}_{V; \one_V}(j)., 
 \; E^{W^{(i)}; (1)}_{W^{(i)} V'}. \right).     
\end{eqnarray} 
The definition of the $E$-operators given in the Appendix.  
 When acting on 
a map $\Phi^{(i)} \in C^{k_i}_{m_i}\left(V, \W^{(i)} \right)$, 
each entry of \eqref{mathe} increases the number of the vertex algebra elements 
$(v_{1, i}, \ldots, v_{k_i, i})$ 
with a vertex algebra element $v_{k_i+1, i}$.   
According to Proposition of \cite{Huang}  
the number of adapted transversal vertex operators 
with the vertex algebra elements $(v'_{1, i}, \ldots, v'_{m_i, i})$ 
decreases to $(m_i-1)$  
as the result of the action of each entry of \eqref{mathe} on $\Phi^{(i)}$.   

The coboundary operators  
$\delta^{k_i}_{m_i}$ acting on elements 
$\Phi^{(i)} \in C^{k_i}_{m_i}$ $\left(V\right.$, $\left.\W^{(i)} \right)$
 of the families of spaces \eqref{ourbicomplex},  
are defined by 
\begin{equation} 
\label{deltaproduct1}
\delta^{k_i}_{m_i} \Phi^{(i)}={\mathcal E^{(i)}}.\Phi^{(i)}.     
\end{equation} 

Here $.$ represents the action of each element of ${\mathcal E}^{(i)}$ 
of the vector on a single element $\Phi^{(i)}$. 
Note that 
${\mathcal E^{(i)}}.\Phi^{(i)} \in C^{k_i+1}_{m_i-1}\left(V, \W^{(i)} \right)$
due to \eqref{mathe} and \eqref{deltaproduct1}. 
A vertex operator added by $\delta^{k_i}_{m_i}$ has a formal parameter associated with 
an extra point $p_{k_i+1, i}$ on $M$ with a local coordinate $c_{k_i+1}(p_{k_i+1, i})$. 
 The right hand side of \eqref{deltaproduct1} 
is adapted transversal to $m_i-1$ vertex operators.  
Let us mention, that the foliation cohomology
is affected by the particular choice of 
$m_i$ vertex operators excluded.   
In \cite{zucmp} we proved 
\begin{lemma}
For arbitrary $w'_i \in W'_i$ dual to $W^{(i)}$,  
the definition \eqref{deltaproduct1} is equivalent to 
 a multi-point vertex algebra connection 
\begin{equation}
\label{deltaproduct}
\delta^{k_i}_{m_i} \Phi^{(i)}(g_i; v_{1, i}, x_{1, i}; \ldots; v_{1, i}, x_{1, i}) 
 = G(g; p_{1, i}, \ldots, p_{k_i+1, i}).     
\end{equation}
\hfill $\square$
\end{lemma}
The explicit form of $G(g; p_{1, i}, \ldots, p_{k_i+1, i})$ was derived in \cite{zucmp}. 
According to the construction of the families of complexes spaces \eqref{ourbicomplex}   
the action 
 of $\delta^{k_i}_{m_i}$ on an element of 
$C^{k_i}_{m_i}\left(V, \W^{(i)} \right)$ give rise a coupling  
as differential forms of $\W^{(i)}_{x_{1, i}, \ldots, x_{k_i, i}}$. 
These are the vertex operators with the local coordinates 
$c_{j, i}(z_{p_{j, i}})$, $0 \le j \le m_i$,   
at the vicinities of the same points $p_{j, i}$
 taken on transversal sections for $\F$,   
with elements of $C^{k_i}_{m_i-1}\left(V, \W^{(i)} \right)$
considered at the    
points 
$c_{j, i}(z_{p_j, i})$, $0 \le j \le n$ on $M$  
at $p_{j, i}$. 

 There exists an additional family of exceptional short complexes 
 which we call the family of transversal connection complexes 
in addition to the families of complexes  
$\left(C^{k_i}_{m_i} \left(V \right. \right.$, 
$\left. \W^{(i)} \right)$, $\left. \delta^{k_i}_{m_i}\right)$   
 given by \eqref{ourbicomplex}   
and \eqref{deltaproduct}. 
In \cite{zucmp} we proved 
\begin{lemma}
\label{lemmo}
 For $k_i=2$, and $m_i=0$, there exist subspaces   
$C_{m_i}^{2, i}$ $\left(V, \W^{(i)}  \right)$  
 $\subset$ $C^{0, i}_{ex}$ $\left(V, \W^{(i)} \right)$ 
$\subset C_{0, i}^2\left(V, \W^{(i)}\right)$,    
 for all $m_i \ge 1$, with the action of the coboundary operator  
$\delta^{2, i}_{m_i}$ defined by \eqref{deltaproduct}.   
\hfill $\square$
\end{lemma}

 The coboundary operators  
\begin{equation}
\label{halfdelta}
\delta^{2, i}_{ex, i}: C_{ex, i}^{2, i}\left(V, \W^{(i)} \right) 
\to C_{0, i}^{3, i}\left(V, \W^{(i)} \right), 
\end{equation}
are defined 
 by the corresponding 
three point connections.  
In \cite{zucmp} we proved  
\begin{proposition}
\label{cochainprop}
The operators \eqref{deltaproduct} and 
\eqref{halfdelta} form the cochain complexes    
\begin{equation}
\label{conde}
\delta^{k_i}_{m_i}: C_{m_i}^{k_i}\left(V, \W^{(i)} \right)   
\to C_{m_i-1}^{k_i+1}\left(V, \W^{(i)} \right),   
\end{equation}  
\begin{equation}
\label{deltacondition}  
\delta^{k_i+1}_{m_i-1} \circ \delta^{k_i}_{m_i}=0,  
\end{equation} 
\begin{equation}
\label{porto1}
\delta^{2, i}_{ex, i} \circ \delta^{1, i}_{2, i}=0,  
\end{equation}
\begin{equation}
\label{hat-complex}
0\longrightarrow C_{m_i}^0\left(V, \W^{(i)} \right)  
\stackrel{\delta^0_{m_i} }{\longrightarrow}
 C_{m_i-1}^1\left(V, \W^{(i)} \right)    
\stackrel{\delta^1_{m_i-1}}{\longrightarrow}\ldots 
\stackrel{\delta^{m_i-1}_1}{\longrightarrow}
 C_0^{m_i}\left(V, \W^{(i)} \right)\longrightarrow 0,  
\end{equation}
\begin{eqnarray}
\label{hat-complex-half}
0\longrightarrow C^{0, i}_{3, i}\left(V, \W^{(i)} \right) 
\stackrel{ \delta_{3, i}^{0, i} }{\longrightarrow} 
 C^{1, i}_{2, i}\left(V, \W^{(i)} \right) 
\stackrel{ \delta_{2, i}^{1, i} }{\longrightarrow} 
C_{ex, i}^{2, i}\left(V, \W^{(i)} \right)  &&
\nn
\label{shorta} 
\stackrel{\delta_{ex}^2}{\longrightarrow}  
 C_{0, i}^{3, i}\left(V, \W^{(i)} \right)\longrightarrow 0, &&
\end{eqnarray}
with the spaces \eqref{ourbicomplex}.  
\noindent
With  
$\delta_{2, i}^{1, i} C_{2, i}^{1, i}\left(V, \W^{(i)} \right) \subset  
 C_{1, i}^{2, i}\left(V, \W^{(i)} \right)\subset  
 C_{ex, i}^{2, i}\left(V, \W^{(i)} \right)$,    
$\delta^{2, i}_{ex, i}\circ  \delta^{1, i}_{2, i} 
= \delta^{2, i}_{1, i}\circ  \delta^{1, i}_{2, i}   
=0$. 
\hfill $\square$ 
\end{proposition}
 The cohomology series 
 $H^{k_i}_{m_i}\left(V, \W^{(i)}, \F\right)$ of $M/\F$ 
with coefficients in     
$\W^{(i)}_{z_1, \ldots, z_n}$
 containing maps adapted transversal
 to $m_i$ vertex operators on $m_i$ transversal sections, 
as the factor space 
 $H_{m_i}^{k_i}\left(V, \W^{(i)}, \F\right)
= {\mathcal Con}_{m_i; \; cl}^{k_i}/G^{k_i-1}_{m_i+1}$.      
of closed 
multi-point connections with respect to the space of connection forms. 
  It is easy to see that the definition of cohomology in terms of 
 multi-point connections 
is equivalent to the standard cohomology definition  
 $H_{m_i}^{k_i}\left(V, \W^{(i)}, \F\right)=
 \mbox{\rm Ker} \; \delta^{k_i}_{m_i}/\mbox{\rm Im}\; \delta^{k_i-1}_{m_i+1}$.    
\section{Sequences of multiple products for complexes} 
\label{productc}
In this Section 
the material of Section \ref{product} 
is applied to   
 the families of cochain complex spaces $C^{k_i}_{m_i}(V, \W^{(i)})$ 
 defined in Section \ref{spaces} 
for a foliation $\F$ on a complex curve. 
We introduce the product of a few cochain  complex spaces  
with the image in another cochain  complex space coherent with respect 
to the original coboundary operators \eqref{deltaproduct} and \eqref{halfdelta}, 
and the symmetry property \eqref{shushu}.
 We prove the canonical property of the product, 
 and derive an analogue of Leibniz formula. 
\subsection{Sequences of multiple products defined for foliation complexes}  
\label{adapta}
In this Subsection
 we extend the definition 
of the $\W^{(i)}_{z_1, \ldots, z_n}$-spaces multiple 
product to  
$C^{k_i}_{m_i}\left(V, \W^{(i)} \right)$-spaces 
for a codimension one foliation on a complex curve.  
Recall the definition \eqref{ourbicomplex} of 
$C^{k_i}_{m_i}\left(V, \W^{(i)} \right)$-spaces given in Section \ref{spaces}.  
In order to introduce the product 
of a few elements $\Phi^{(i)} \in C^{k_i}_{m_i}\left(V, \W^{(i)} \right)$ 
that belong to    
several cochain complex spaces \eqref{ourbicomplex} for a foliation $\F$
We then use the geometric multiple $\rho$-scheme
 of a Riemann surface self-sewing.  
We assume that each of the cochain  complex spaces
 $C^{k_i}_{m_i}\left(V, \W^{(i)} \right)$ 
is considered on the same fixed transversal basis $\U$
since the construction is again local. 
Moreover, we assume that the marked points used in the definition \eqref{ourbicomplex} of 
the spaces $C^{k_i}_{m_i}\left(V, \W^{(i)} \right)$  
are chosen on the same transversal section.  
 %
 Recall the setup for a few cochain complex spaces  
$C^{k_i}_{m_i}\left(V, \W^{(i)} \right)$.  
Let $(p_{1, i}, \dots, p_{k_i, i})$, $1 \le i \le l$,       
be sets of points with the local coordinates 
$(c_{1, i}(p_{1, i})$, $ \dots $, $c_{k_i, i}(p_{k_i, i}))$    
 taken on the $j$-th transversal section $U_{j, i} \in \U$, $j\ge 1$,    
of the transversal basis $\U$.  
For $k_i \ge 0$,    
 let $C^{k_i}_{(m_i)}\left(V, \W^{(i)} \right)(U_m)$, $0 \le j \le m$,       
be as before the spaces of all linear maps \eqref{maps} 
\begin{eqnarray}
\label{mapy11}
 && \Phi^{(i)}: V^{\otimes k_i} \rightarrow 
\W^{(i)}_{c_{1, i}(p_{1, i}), \dots, c_{k_i, i}(p_{k_i, i}) 
c_{1, i}(p_{1, i}), \dots, c_{k_i, i}(p_{k_i, i}); 
\atop 
c_{1, i}(p'_{1, i}), \dots, c_{m_i, i}(p'_{m_i, i}) 
c_{1, i}(p'_{1, i}), \dots, c_{k_i, i}(p'_{m_i, i})},     
\end{eqnarray}
adapted transversal to 
 vertex operators
 \eqref{poper} with the formal parameters identified 
with the local coordinate functions $c'_{j, i}(p'_{j, i})$ around points $p_{j, i}$,  
on each of the transversal sections $U_{j, i}$, $1 \le j \le l_1$, $1 \le i \le l$.     
According to the definition \eqref{ourbicomplex}, 
 for $k_i\ge 0$, $1 \le m_i \le l_1$,  
the spaces $C^{k_i}_{m_i}(V, \W^{(i)})$ are:    
\begin{equation}
\label{ourbicomplex111}
 C^{k_i}_{m_i}\left(V, \W^{(i)} \right) =  \bigcap_{  
U_1 \stackrel{h_1, i} {\hookrightarrow }  
\ldots \stackrel {h_{m_i-1, i}}{\hookrightarrow } U_{m_i, i}, \;       
1 \le i \le m_i}  
 C^{k_i}_{(m_i)}\left(V, \W^{(i)} \right) (U_{j, i}),    
\end{equation}
where the intersection ranges over all possible $m_i$-tuples of the holonomy embeddings  
$h_{j, i}$, $1 \le j \le m_i-1$,    
between the transversal sections $(U_{1, i}, \ldots, U_{m_i, i})$ of the basis $\U$ for $\F$.  
Let $t$ be the number of the coinciding vertex operators for the mappings  
that are adapted transversal to
$\Phi^{(i)}(g_i; v_{1, i}, x_{1, i}$;  $\ldots$; $v_{k_i, i}, x_{k_i, i})  
\in C^{k_i}_{m_i}\left(V, \W^{(i)} \right)$, $1 \le i \le l$.    
Denote $\mu_i=m_1+\ldots+m_i$. 
%
Elements $\Phi^{(1, \ldots, l)}$ of the tensor product
 $\W^{(1, \ldots, l)}_{z_1, \ldots, z_{\theta_l-r}}$ 
correspond to the choice of a set of leaves of $M/\F$. 
Thus, the collection of matrix elements of \eqref{ourbicomplex111}
 identifies the space 
$C_{\mu_l - t}^{\theta_l-r}\left(V, \W^{(1, \ldots, l)} \right)$. 
 Let us formulate the main proposition of this Section.  
\begin{proposition}
\label{tolsto}
For $\Phi^{(i)}(g_i; v_{1, i}, x_{1, i}; \ldots; v_{k_i, i}, x_{k_i, i})   
\in C_{m_i}^{k_i}  \left(V, \W^{(i)} \right)$     
the sequence of products \eqref{gendef}  
$\widehat{\Theta}$ $\left(f_1, \ldots, f_l; 
g_1, \ldots, g_l; v_{1, 1}, x_{1, 1};  \ldots; v_{k_l, l},  x_{k_l, l};      
 \rho_1, \ldots, \rho_l; \zeta_{1, i}, \zeta_{2, i} \right)_k$ \eqref{bardos}    
belongs to the space 
$C^{\theta_l-r}_{\mu_l - t}$  $\left(V, \W^{(1, \ldots, l)}, \right)$, i.e.,  
\begin{equation}
\label{toporno}
\cdot_{\rho_1, \ldots, \rho_l}:
 \times_{i=1}^l C^{k_i}_{m_i}\left(V,\W^{(i)} \right)  
\to C_{\mu_l - t}^{\theta_l-r}\left(V, \W^{(1, \ldots, l)} \right).    
\end{equation}
\end{proposition}
\begin{proof}
In Proposition \ref{derga} it was proven that 
$\widehat{\Theta}$ $\left(f_1 \right.$, $\ldots$, $f_l$;   
$g_1, \ldots, g_l$; $v_{1, 1}, x_{1, 1}$;  
 $\ldots$; $v_{k_l, l}, x_{k_l, l}$;    
 $\rho_1, \ldots$, $\rho_l$; $\zeta_{1, i}$, $\left. \zeta_{2, i} \right)_k$    
 $\in \W^{(1, \ldots, l)}_{z_1, \ldots, z_{\theta_l-r}}$.      
Namely, the differential 
forms corresponding to the sequence  
multiple product $\widehat{\Theta}$ $(f_1$, $\ldots$, $f_l$; $g_1$, $\ldots$, $g_l$; 
 $v_{1, 1}$, $x_{1, 1}$;  $\ldots$; $v_{k_l, l}$,  $x_{k_l, l}$;     
 $\rho_1$, $\ldots$, $\rho_l$; $\zeta_{1, i}$, $\zeta_{2, i})_k$
 converge in $\rho_i$ individually, and are subject to    
  \eqref{shushu},  the $L_V(0)$-conjugation \eqref{loconj} and 
the $L_V(-1)$-derivative \eqref{lder1} properties. 
The formula \eqref{sigmaction} gives 
 the action of $\sigma \in S_{k_l-r}$ on the product 
$\widehat{\Theta}$ $\left(f_1, \ldots, f_l; g_1, \ldots, g_l \right. $ ; 
$v_{1, 1}, x_{1, 1};  \ldots; v_{k_l, l}, x_{k_l, l};  
\rho_1, \ldots, \rho_l $ ; $\left. \zeta_{1, i}, \zeta_{2, i}\right)_k$    
 \eqref{bardos}.   
Then we see that for the sets of points 
$(p_{1, i}$, $\dots$, $p_{k_i, i})$,  
taken on the same 
transversal section $U_{j, i} \in \U$, $j\ge 1$, 
by Proposition \ref{derga} we obtain a map 
$\widehat{\Theta}$ $\left(f_1 \right.$, $\ldots$, $f_l$; $g_1$, $\ldots$, $g_l$; 
  $v_{1, 1}$, $x_{1, 1}$; $\ldots$;  $v_{k_l, l}$, $x_{k_l, l}$;  
 $\rho_1$, $\ldots$, $\rho_l$; $\zeta_{1, i}$, $\left. \zeta_{2, i} \right)_k$,    
$:$ $V^{\otimes (\theta_l) } \rightarrow    
 \W^{(1, \ldots, l)}_{c_1(p_1), \dots, c_{k_1+\ldots+k_l-r}(p_{k_1+\ldots+k_l-r})}$,     
 with the non-coinciding formal parameters $(z_1, \ldots, z_{\theta_l-r})$   
identified with the local coordinates  
$(c_1(p_1), \ldots $, $ c_{\theta_i-r_i}(p_{\theta_i-r}))$,   
of the points 
$(p_{1, 1}, \ldots,  p_{k_1, 1},  \ldots, p_{1, l}  \ldots, p_{k_l, l})$.  
%
Let us show that 
\begin{eqnarray*}
 && \sum_{q_1, \ldots, q_l \in \C}
\langle w', E^{(m_1+\ldots+m_l)}_{W^{(1, \ldots, l)}} \Big( 
v_1, z_1; \ldots; 
v_{m_1+\ldots+m_l}, z_{m_1+\ldots+m_l}; \qquad \qquad  \qquad \qquad  \qquad \qquad  
\nn
 && \qquad 
P_{q_1, \ldots, q_l}  \Big( \Phi^{(1, \ldots, l)}(f_1, \ldots, f_l; g_1, \ldots, g_l;  
v_{m_1+\ldots+m_l+1}, z_{m_1+\ldots+m_l+1};  
\ldots; \qquad \qquad  \qquad \qquad \qquad \qquad  \qquad 
\nn
 && 
 \qquad \qquad v_{m_1+\ldots+m_l+k_1+\ldots+k_l}, z_{m_1+\ldots+m_l+k_1+\ldots+k_l}; 
\rho_1, \ldots, \rho_l  
\Big) \Big)
 \rangle \qquad \qquad  \qquad  
\nn
 && =
\sum_{u\in V_{(k)} \atop k \in \Z}    
\widehat{\mathcal R}\prod_{i=1}^l
 \rho^k_i 
\langle w'_i,   
E^{(m_i)}_{W^{(i)}} \Big(
v_{k_i+1, i}, x_{k_i+1, i}; \ldots; 
v_{k_i+m_i, i}, x_{k_i+m_i, i}; \qquad \qquad  \qquad \qquad  \qquad \qquad 
\nn
 && \qquad  P_{q_i} \Big( Y^{W^{(i)}}_{W^{(i)}V'}
\left(   
\Phi^{(i)} (g_i; v_{1, i}, x_{1, i};  
\ldots; v_{k_i, i}, x_{k_i, i}, u, \zeta_{1, i}),
 \zeta_{2, i} \right) f_i.\overline{u} \Big) \Big)\rangle. \qquad \qquad  
\end{eqnarray*}
Indeed,
in the Appendix 
 the definition \eqref{xusnya} 
of 
 $E^{(m_1+\ldots+m_l)}_{W^{(1, \ldots, l)}}$ was given. 
Consider
\begin{eqnarray*}
&&
\sum_{u\in V_{(k)} \atop k \in \Z}     
\mathcal R \prod_{i=1}^l 
 \rho^k_i  
\langle w'_i,    
E^{(m_1+\ldots+m_l)}_{ W^{(1, \ldots, l)} } \Big(
v_1, z_1; \ldots;  
v_{m_1+\ldots+m_l}, z_{m_1+\ldots+m_l}; 
\nn
&& \qquad P_{q_1, \ldots, q_l} \Big(   Y^{ W^{(i)} }_{ W^{(i)}V' }
\left( \Phi^{(i)} (g_i; v_{m_1+\ldots+m_i +1}, z_{m_1+\ldots+m_i +1};  \right. 
\nn
&&
\qquad \qquad  \left. 
\ldots; v_{m_1+\ldots+m_i +k_i}, z_{m_1+\ldots+m_i +k_i}, u, \zeta_{1, i}), 
 \zeta_{2, i} \right) f_i.\overline{u} \Big) \Big)\rangle. 
\end{eqnarray*}
\begin{eqnarray*}
=\sum_{u\in V_{(k)} \atop k \in \Z}    
\mathcal R \prod_{i=1}^l
 \rho^k_i 
\langle w'_i,   
E^{(m_1+\ldots+m_l)}_{W^{(1, \ldots, l)}} \Big(v_1, z_1; \ldots;  
v_{m_1+\ldots+m_l}, z_{m_1+\ldots+m_l};
 \qquad \qquad  \qquad \qquad  \qquad \qquad &&
\nn
 P_{q_1, \ldots, q_l} \Big(
e^{\zeta_{2, i} L_{W^{(i)}}(-1) }\; Y_{W^{(i)}} \left( f_i.\overline{u}, -\zeta_{2, i} \right) 
\qquad \qquad  \qquad \qquad  \qquad \qquad \qquad \qquad \qquad  \qquad
&& 
\nn
\Phi^{(i)} (
g_i; v_{m_1+\ldots+m_i +1}, z_{m_1+\ldots+m_i +1};  
\ldots; v_{m_1+\ldots+m_i +k_i}, z_{m_1+\ldots+m_i +k_i}; u, \zeta_{1, i}),
 \Big) \Big)\rangle. \qquad \qquad \qquad && 
\end{eqnarray*}
The action of 
a grading-restricted generalized $V$-module $W^{(i)}$ 
vertex operators $Y_{W^{(i)}} \left( f_i.\overline{u}, -\zeta_{a, i}\right)$,
and the exponentials $e^{\zeta_{a, i} L_{W^{(i)}}(-1)}$, $a=1$, $2$,
 of the differential operator $L_{W^{(i)}}(-1)$,     
 shifts the grading index $q$ of the $W^{(i)}_{q_i}$-subspaces  
by $\alpha_i \in \C$ which can be later rescaled 
to $q_i$.   
 Thus, the last expression transforms to 
\begin{eqnarray*} 
&& \sum_{q\in \C}  
\sum_{u\in V_{(k)} \atop k \in \Z}    
\mathcal R \prod_{i=1}^l
 \rho^k_i 
\langle w'_i,   
E^{(m_1+\ldots+m_l)}_{W^{(1, \ldots, l)}} \Big(v_1, z_1; \ldots;   
v_{m_1+\ldots+m_l}, z_{m_1+\ldots+m_l};
\nn
  && \qquad e^{\zeta_{2, i} L_{W^{(i)}}(-1)}\; 
 Y_{W^{(i)}} \left( f_i.\overline{u}, -\zeta_{2, i} \right)   
\nn
 && \qquad \qquad  P_{q_1+\alpha_1, \ldots, q_l+\alpha_l}  
\left(  
\Phi^{(i)} (
g_i; v_{m_1+\ldots+m_i +1}, z_{m_1+\ldots+m_i +1};  \right. 
\nn
&& \qquad \qquad \qquad \left. \ldots; v_{m_1+\ldots+m_i +k_i}, z_{m_1+\ldots+m_i +k_i}; 
u, \zeta_{1, i}),  \right)  \rangle 
\end{eqnarray*}
\begin{eqnarray*} 
\nonumber 
&& =\sum_{q\in \C}  
\sum_{u\in V_{(k)} \atop k \in \Z}    
\mathcal R \prod_{i=1}^l
 \rho^k_i 
\langle w'_i,   
E^{(m_1+\ldots+m_l)}_{W^{(1, \ldots, l)}} \Big(v_1, z_1; \ldots;  
v_{m_1+\ldots+m_l}, z_{m_1+\ldots+m_l};
\nn
\nonumber 
  &&\qquad \qquad Y^{ W^{(i)} }_{ W^{(i)}V' } \left(
P_{q_1+\alpha_1, \ldots, q_l+\alpha_l} 
\left(  
\Phi^{(i)} (
g_i; v_{m_1+\ldots+m_i +1}, z_{m_1+\ldots+m_i +1}; \right. \right. 
\nn 
\nonumber 
 && \qquad \qquad \qquad 
\left. \left. 
 \ldots; v_{m_1+\ldots+m_i +k_i}, z_{m_1+\ldots+m_i +k_i};
 u, \zeta_{1, i}) \right), \zeta_{2, i} \right)  
f_i.\overline{u}  \rangle 
\end{eqnarray*}
\begin{eqnarray*} 
&&
\nonumber 
 = 
\sum_{q\in \C}  
\sum_{u\in V_{(k)} \atop k \in \Z}    
\mathcal R \prod_{i=1}^l \sum_{\widetilde{w}_i \in W^{(i)}}  
 \rho^k_i 
 \langle w'_i, E^{(m_1+\ldots+m_l)}_{W^{(1, \ldots, l)}} \Big(v_1, z_1; \ldots;  
v_{m_1+\ldots+m_l}, z_{m_1+\ldots+m_l};
  \widetilde{w}_i \Big) \rangle  
\nn
\nonumber
&&
\qquad   
\langle w'_i,  Y^{W^{(i)}}_{W^{(i)}V'}  
\left( 
\;P_{q+\alpha}\Big(   \Phi (
g_i; v_{m_1+\ldots+m_i +1}, z_{m_1+\ldots+m_i +1};  \right. 
\nn
\nonumber
&&
\qquad \qquad \left. \ldots; v_{m_1+\ldots+m_i +k_i}, z_{m_1+\ldots+m_i +k_i}; 
u, \zeta_{1, i}) \right), \zeta_{2, i} \Big)  \Big) 
f_i.\overline{u}  \rangle
\end{eqnarray*}
\begin{eqnarray*}
&&
= \sum_{q\in \C}    
 \langle w', E^{(m_1+\ldots+m_l)}_{W^{(1, \ldots, l)}} \Big(
v_1, z_1; \ldots;  
v_{m_1+\ldots+m_l}, z_{m_1+\ldots+m_l};   
\nn
&&
\qquad   
 P_{q+\alpha}\Big(   \Phi^{(1, \ldots, l)} (f_1, \ldots, f_l; g_1, \ldots, g_l;
 v_{m_1+\ldots+m_i+1}, z_{m_1+\ldots+m_i+1};  
\nn 
&&
\qquad \ldots; v_{m_1+\ldots+m_i+ k_1+\ldots+k_l}, z_{m_1+\ldots+m_i+k_1+\ldots+k_l}   
\Big)   
\rangle. 
\end{eqnarray*}
According to Proposition \ref{pupa}, 
as an element of $\W^{(k_1, \ldots, k_l)}_{z_1, \ldots,  z_{m_1+\ldots+m_l+k_1+\ldots+k_l}}$
\begin{eqnarray}
\label{svoloch}
&& \langle w', E^{(m_1+\ldots+m_l)}_{W^{(1, \ldots, l)}} \Big(
v_1, z_1; \ldots;  
v_{m_1+\ldots+m_l}, z_{m_1+\ldots+m_l};    
\nn
&&
\qquad   
 P_{q+\alpha}\Big(   \Phi^{(1, \ldots, l)} (f_1, \ldots, f_l; g_1, \ldots, g_l;
v_{m_1+\ldots+m_i+1}, z_{m_1+\ldots+m_i+1};  
\nn
&&
\qquad \qquad \qquad 
\ldots; v_{m_1+\ldots+m_i+ k_1+\ldots+k_l}, z_{m_1+\ldots+m_i+k_1+\ldots+k_l}) \Big)     
\rangle, 
\end{eqnarray}
is invariant under the action of $\sigma \in S_{m_1+\ldots+m_i+k_1+\ldots+k_l}$.  
 Thus, it possible to use this invariance to show that \eqref{svoloch} reduces to  
\begin{eqnarray*}
&& \langle w', E^{(m_1+\ldots+m_l)}_{W^{(1, \ldots, l)}} \Big(
v_{k_1+1}, z_{k_1+1}; \ldots; v_{k_1+1+m_1}, z_{k_1+1+m_1}; 
\nn
&&
 \qquad  \ldots; 
v_{k_l+1}, z_{k_l+1}; \ldots; v_{k_l+1+m_l}, z_{k_l+1+m_l};    
\nn
&&
\qquad   
 P_{q+\alpha}\Big(   \Phi^{(1, \ldots, l)} (f_1, \ldots, f_l; g_1, \ldots, g_l;
 v_1, z_1;  \ldots; v_{k_1}, z_{k_1};  \ldots; 
\nn
&&
\qquad \qquad \qquad \qquad \qquad \qquad 
v_{k_1+\ldots+k_l}, z_{k_1+\ldots+k_l}) \Big) \Big)  
\rangle  
\nn
&&
=\langle w', E^{(m_1+\ldots+m_l)}_{W^{(1, \ldots, l)}} \Big(
v_{k_1+1, i}, x_{k_1+1, i};  \ldots; v_{k_l+1+m_l}, x_{k_l+1+m_l};    
\nn
&&
\qquad   
 P_{q+\alpha}\Big(   \Phi^{(1, \ldots, l)} (f_1, \ldots, f_l; 
g_1, \ldots, g_l; v_{1, i}, x_{1, i};  \ldots; v_{k_i, i}, x_{k_i, i}) \Big)    
\rangle.  
\end{eqnarray*}
Similarly, for $1 \le i \le l$ 
\begin{eqnarray*}
&& 
\langle w',  E^{(m_1+\ldots+m_l)}_{W^{(1, \ldots, l)}} \Big(
v_1, z_1; \ldots;
v_{m_1+\ldots+m_l}, z_{m_1+\ldots+m_l};  
\nn
&& 
\qquad  P_q \Big(  
  Y^{W^{(i)}}_{W^{(i)}V'} \left(  
\Phi^{(i)} (v_{m_1+\ldots+m_l+1}, z_{m_1+\ldots+m_l+1};  \right. 
\nn
&&  
\qquad \left. \left. 
\ldots; v_{m_1+\ldots+m_l+k_1+\ldots+k_i}, z_{m_1+\ldots+m_l+k_1+\ldots+k_i}); 
u, \zeta_{1, i}) \right), \zeta_{2, i} \right)  
f_i.\overline{u}  \Big) \Big)\rangle,  
\end{eqnarray*} 
correspond to the elements 
of $\W_{z_1, \ldots, z_{m_1+\ldots+m_l+k_1+\ldots+k_i}}$.  
Let us use Proposition \ref{pupa} again and we arrive at 
\begin{eqnarray*}
\nonumber
&& 
\langle w',  E^{(m_1+\ldots+m_l)}_{W^{(1, \ldots, l)}} \Big(
v_{k_i+1, i}, x_{k_i+1, i}; \ldots;  
v_{k_i+m_i}, x_{k_i+m_i};   
\nn
&&
\qquad 
 P_q \Big(  
  Y^{W^{(i)}}_{W^{(i)}V'} \left(  
\Phi^{(i)} (v_{1, i}, x_{1, i};  \ldots; v_{k_i, i}, x_{k_i, i}); u, \zeta_{1, i}) \right), 
f_i.\overline{u}  \Big) \Big) \rangle.  
\end{eqnarray*}  

Next, we prove 
\begin{proposition}
\label{ccc}
The products  
$\Theta
\left( f_1, \ldots, f_l; g_1, \ldots, g_l \right.$;  
 $v_1, z_1; \ldots; v_{\theta_l-r},  z_{\theta_l-r}$;   
$\rho_1$, $\ldots$, $\rho_l$; $\zeta_{1, i}$, $\left. \zeta_{2, i} \right)$
 \eqref{bardos}  
are adapted transversal to $\mu_l-t$ vertex operators.  
\end{proposition}
\begin{proof}
Recall that 
$\Phi^{(i)}(g_i; v_{n_i+1}, z_{n_i+1}; \ldots; v_{n_i+k_i},  z_{n_i+k_i})$, $1\le i \le l$, 
 are adapted transversal to $m_i-t_i$ vertex operators.    
 For the first condition of the adapted transversality: 
 let $l_{1, i}, \ldots, l_{k_i-r_i, i} \in \Z_+$  
such that $l_{1, i}+\ldots +l_{k_i, i}= n_i+k_i-r_i+m_i-t_i$. 
For an arbitrary $w'_i\in W^{(i)}{}'$, denote
\begin{eqnarray}
&& (v_{n_i+1}, \ldots, v_{n_i+k_i}, v_{n_i+k_i+1}, \ldots, v_{n_i+k_i+m_i-t_i})
\nn
&& 
 \qquad =(v_{n_i+1}, \ldots, v_{n_i+k_i}, v'_{n_i+k_i+1}, \ldots, v'_{n_i+k_i+m_i-t_i}), 
\nn
&& (z_{n_i+1}, \ldots, z_{n_i+k_i}, z_{n_i+k_i+1}, \ldots, z_{n_i+k_i+m_i-t_i})
\nn
&& 
\qquad =(z_{n_i+1}, \ldots, z_{n_i+k_i}, z'_{n_i+k_i+1}, \ldots, z'_{n_i+k_i+m_i-t_i}).  
\end{eqnarray} 
 Define  
$\Xi_{j, i}
=
E^{(l_{j,i})}_V(v_{\varkappa_{1, i}}, z_{\varkappa_{1, i}}- \varsigma_{j, i};    
 \ldots; 
v_{\varkappa_{j, i}}, z_{\varkappa_{j,i}}- \varsigma_{j,i}   
 ; \one_V)$,      
where
\begin{eqnarray}
\label{varki}
 \varkappa_{1, i}=l_{1, i}+\ldots +l_{j-1, i}+1, 
\quad  \ldots, \quad  \varkappa_{j, i}=l_{1, i}+\ldots +l_{j-1, i}+l_j,   
\end{eqnarray} 
for $1 \le j \le k_i-r_i$.  
Then the series 
\begin{eqnarray}
\label{Inms}
&& \mathcal R^{1, k_i-r_i}_{m_i-t_i}\left(\Phi^{(i)}\right)= 
R \sum_{r_{1, i}, \ldots, r_{k_i-r_i, i} \in \Z}\langle w'_i,   
\Phi^{(i)} \left( g_i; P_{r_{1, i}}\Xi_{1, i}; \varsigma_{1, i};  \right. 
 \ldots; 
\nn
&&
\qquad \qquad \qquad \qquad \qquad \qquad
\left. P_{r_{k_i-r_i, i}} \Xi_{k_i-r_i, i}, \varsigma_{k_i-r_i, i}\rangle \right)    
\rangle,
\end{eqnarray} 
is absolutely convergent  when 
$|z_{l_{1, i}+\ldots +l_{j-1, i}+p_i}-\varsigma_{j, i}| 
+ |z_{l_{1, i}+\ldots +l_{j'-1, i}+q}-\varsigma_{j', i}|< |\varsigma_{j, i} -\varsigma_{j', i}|$,   
for $j$, $1 \le j' \le k_i-r_i$, $j\ne j'$, 
 and for $1 \le p_i \le l_{j, i}$ and $1 \le q_i\le l_{j', i}$.  
There exist positive integers $N^{k_i-r_i}_{m_i-t_i}(v_{j, i}, v_{j', i})$,   
depending only on $v_{j, i}$ and $v_{j', i}$ 
for $1 \le j$, $j' \le m_i-t_i$, $j\ne j'$, such that  
the sum is analytically extended to a
rational function
in $(z_1, \ldots, z_{n_i+k_i-r_i+m_i-t_i})$,     
 independent of $(\varsigma_{1, i}, \ldots, \varsigma_{k_i-r_i, i})$,     
with the only possible poles at  
$x_{j, i}=x_{j', i}$, of order less than or equal to   
$N^{k_i-r_i}_{m_i-t_i}(v_{j, i}, v_{j', i})$, 
for $j$, $1 \le j' \le k_i-r_i$,  $j \ne j'$.   

Now let us consider the first condition of the definition   
of the adapted transversal for the product 
\eqref{bardos} of 
$\Phi^{(i)}(g_i; v_1, z_1; \ldots; v_{\theta_i},  z_{\theta_i})$    
  with a number of vertex operators.   
We obtain for  
$\Theta\left(f_1, \ldots, f_l; g_1, \ldots, g_l; v_1, z_1; \ldots; 
 v_{\theta_l}, z_{\theta_l}; \rho_1, \ldots, \rho_l\right)$ the following.   
Introduce $l'_1, \ldots, l'_{\theta_l-r} \in \Z_+$,      
 such that $l'_1+\ldots +l'_{\theta_l-r}= \theta_l -r+\mu_l-t$.   
Define   
$\Xi'_{j''} 
=
E^{(l''_{j''})}_V (v_{\varkappa'_1}, z_{\varkappa'_1}- \varsigma'_{j''};  
 \ldots; 
v_{\varkappa'_{j''}}, z_{\varkappa'_{j''}}- \varsigma'_{j''}   
 ; \one_V)$,     
 $\varkappa'_1=l'_1+\ldots +l'_{j''-1}+1$, 
$\ldots$, 
$\varkappa'_{j''}=l'_1+\ldots +l'_{j''-1}+l'_{j''}$,     
for $1 \le j'' \le \theta_l-r$,    
and we take 
$(\zeta'_1, \ldots, \zeta'_{\theta_{k_l}-r} ) 
= (\zeta_1, \ldots, \zeta_{k_1-r_1}; \ldots; 
 \zeta_{n_{l-1}+1}, \ldots, \zeta_{n_l+k_l-r_l})$.   
%
Then we consider 
\begin{eqnarray}
\label{Inmdvadva} 
&& \mathcal R^{1, \theta_l-r}_{\mu_l-r}\left(\Phi^{(1, \ldots, l)}\right)    
= R \sum_{r'_1, \ldots, r'_{\theta_l-r}\in \Z}     
\widehat{\Theta} \left(f_1, \ldots, f_l; g_1, \ldots, g_l; P_{r'_1}\Xi'_1,
 \ldots; \right. 
\nn
&&
\qquad \qquad \qquad \qquad \qquad \qquad
\left. 
P_{r'_{\theta_l-r}} \Xi'_{\theta_l-r}, \varsigma'_{\theta_l-r}\right), 
\end{eqnarray} 
and prove it's absolutely convergence with some conditions.
%
 The condition   
%
$|z_{l'_1+\ldots +l'_{j-1}+p'}$-$\varsigma'_j|$    
$+$ $|z_{l'_1+\ldots +l'_{j'-1}+q'}$- $\varsigma'_{j'}|$ $<$ 
$|\varsigma'_j -\varsigma'_{j'}|$,    
of the absolute convergence for \eqref{Inmdvadva}
 for $1 \le i''$, $j''\le \theta_l-r$, $j\ne j'$,  
 for $1 \le p' \le l'_j$,  
and $1 \le q' \le l'_{j'}$, 
follows from the conditions \eqref{granizy1} and \eqref{granizy2}.  
%
 The action of $e^{\zeta L_{W^{(i)}}(-1) } \;  Y_{W^{(i)}}(.,.)$, $a=1$, $2$, in  
\begin{eqnarray*}
 \langle w_i',  e^{\zeta_1 L_{W^{(i)}} (-1) } \; 
Y_{W^{(i)}}(u, -\zeta)\sum_{ {r_1, \ldots,} \atop {r_{k_i-r_i}\in \Z} }  
\Phi^{(i)}
\left(g_i; P_{r_{1, i}}\Xi_1, \varsigma_1;  \right.  
P_{r_{k_i-r_ik, i}} \Xi_{k_i-r_i}, \varsigma_{k_i-r_i}\rangle)   \rangle,   &&
\end{eqnarray*} 
 does not affect the absolute convergence of \eqref{Inms}.   
 Therefore,  
\begin{eqnarray*}
 && \left|
\mathcal R^{1, \theta_l-r}_{\mu_l-t} 
\left(
\Phi^{(1, \ldots, l)}   \right)   \right|  
\nn
&&  =R \left|
\sum_{r''_1, \ldots, r''_{\theta_l-r}\in \Z}   
\widehat{\Theta} 
\left(
f_1, \ldots, f_l; g_1, \ldots, g_l; P_{r'_1} \Xi'_1, \varsigma'_1;  
 \ldots; 
P_{r'_{\theta_l-r}} \Xi'_{\theta_l-r}, \varsigma'_{\theta_l-r} \right) \right| 
\end{eqnarray*}
\begin{eqnarray*}
\nonumber
&& =\left|
\sum_{u\in V_{(k)} \atop k \in \Z}    
\widehat{\mathcal R}\prod_{i=1}^l
 \rho^k_i 
\langle w'_i,  Y^{W^{(i)}}_{W^{(i)}V'}   
\Big( \sum_{ {r'_1, \ldots, } \atop {r'_{\theta_i-r_i} \in \Z} }       
\Phi^{(i)} \left(
 g_i; P_{r'_1}\Xi'_1, \varsigma'_1;   
\ldots;  \right. \right. 
\nn
&&   
\nonumber
\qquad \qquad \qquad \qquad \left. \left. 
P_{r'_{\theta_i-r_i} } \Xi'_{\theta_i-r_i}, \varsigma'_{\theta_i-r_i}; u, 
\varsigma_{1,i} ) \right),    
 \varsigma_{2, i}\Big)\;  f_i.\overline{u}  \rangle  \right| 
\end{eqnarray*}
\begin{eqnarray*}
&& =\left| 
\sum_{u\in V_{(k)} \atop k \in \Z}    
\widehat{\mathcal R} \prod_{i=1}^l
 \rho_i^k
\langle w'_i,  Y^{W^{(i)}}_{W^{(i)}V'} 
\Big( \sum_{{r_{1, i}, \ldots, } \atop {r_{k_i-r_i, i}
\in \Z}}     
\Phi^{(i)}
\left(
g_i; P_{r_{1, i}}\Xi_{1, i}, \varsigma_{1, i};  
\ldots; 
  \right. \right.  
\nn
&& \left. \left. \left. 
\qquad \qquad \qquad  
P_{r_{k_i-r_i, i}} \Xi_{k_i-r_i, i}, \varsigma_{k_i-r_i, i}; u, \varsigma_{1,i}
\right), 
 \varsigma_{2, i}\right)\;  f_i.\overline{u}  \rangle  \right| 
\end{eqnarray*}
\begin{eqnarray*}
 =\left| 
 \sum_{u\in V_{(k)} \atop k \in \Z}    
\widehat{\mathcal R} \prod_{i=1}^l
 \rho^k_i 
\langle w'_i,  e^{\varsigma_{2, i}L_{W^{(i)}}(-1)} 
Y_{W^{(i)}} ( f_i.\overline{u} ,-\varsigma_{2, i}) \right. 
\qquad \qquad \qquad \qquad \qquad  
&&
\nn
\left. \sum_{{r_{1, i}, \ldots, } \atop {r_{k_i-r_i, i} \in \Z}}     
\Phi^{(i)}\left(g_i; P_{r_{1, i}}\Xi_{1, i}, \varsigma_{1, i};  
\ldots; 
P_{r_{k, i}} \Xi_{k_i-r_i, i}, \varsigma_{k_i-r_i, i}; u, \varsigma_{1,i}\right)\; 
  \rangle  \right| 
  \le \left|\mathcal R^{1, k_i-r_i}_{m_i-t_i}\left(\Phi^{(i)}\right)\right|.   
\end{eqnarray*} 
 We conclude that \eqref{Inmdvadva} 
is absolutely convergent. 
Recall that $N^{k_i-r_i}_{m_i-t_i}(v_{i, i}, v_{j, i})$ are 
the maximal orders of possible poles of 
\eqref{Inmdvadva} at $x_{j, i}=x_{j', i}$. 
From the last expression follows that 
there exist positive integers $N^{\theta_l-r}_{\mu_l-t}(v_{i'', i}, v_{j'', i})$    
 for $1 \le j$, $j' \le k_i-r_i$, $j\ne j'$,      
depending only on $v_{i'', i}$ and $v_{j'', i}$  
for $1 \le i''$, $j'' \le  \theta_{k_l}-r$, $i'' \ne j''$,  
such that 
 the series \eqref{Inmdvadva} 
can be analytically extended to a
rational function
in $(z_1, \ldots, z_{\theta_l-r})$,      
 independent of $(\varsigma'_{1, i}, \ldots, \varsigma'_{\theta_l-r, i})$,     
with extra 
 possible poles at  
 and $z_{j, i}=z'_{j', i'}$, 
of order less than or equal to 
$N^{\theta_l-r}_{\mu_l-t}  
(v_{i'', i}, v_{j'', i})$, for $1 \le i''$, $j'' \le n$,  $i''\ne j''$.   

Now, let us pass to the second condition of the adapted transversal
for $\Phi^{(i)}$  
$(g_i$;  $v_{n_i+1}, z_{n_i+1}; \ldots; v_{n_i+k_i-r_i}, x_{n_i+k_i-r_i}) 
  \in  C^{k_i}_{m_i}\left(V, \W^{(i)} \right)$, and  
$v_{1, i}, \ldots, v_{k_i, i} \in V$, 
$(x_{1, i}, \ldots, x_{k_i+m, i})\in \C$. 
For arbitrary  $w'_i\in W^{(i)}{}'$, the series  
\begin{eqnarray}
\label{Jnm2}
 \mathcal R^{2, k_i-r_i}_{m_i-t_i}\left(\Phi^{(i)}\right)=   
R \sum_{q_i\in \C}\langle w'_i,  
E^{(m_i-t_i)}_{W^{(i)}} \Big(v'{n'_i+1}, z'_{n'_i+1}; 
 \ldots; v'_{n'_i+m_i-t_i}, z'_{n'_i+m_i-t_i};  &&
\nn
 P_{q_i} \left( \Phi^{(i)}(g_i; v_{n'_i+m_i-t_i+1}, z_{n_i+m_i-t_i+1};  
\ldots; v_{n'_i+m_i-t_i+k_i}, z_{n_i+m_i-t_i+k_i}) \right) \Big) \rangle, &&   
\end{eqnarray}
is absolutely convergent when 
$z'_j\ne z'_{j'}$, 
$j\ne j'$, 
$|z'_{j}|>|z'_{j'}|>0$,    
 for $1 \le j \le m_i-t_i$, $m_i+1 \le j' \le k_i+m_i$, 
 and the sum can be analytically extended to a
rational function 
in $(x_{1, i}, \ldots,  x_{k_i+m_i, i})$
 with the only possible poles at 
$x_{j, i}=x_{j', i}$, of orders less than or equal to 
$N^{k_i}_{m_i}(v_{i, i}, v_{j, i})$, 
for $1 \le j$, $j' \le k_i$, $j \ne j'$.    

In the Appendix 
 the definition \eqref{xusnya} of the element 
 $E^{(\mu_l)}_{W^{(1, \ldots, l)}}$ 
for $\Phi^{(1, \ldots, l)} \in 
 \W^{(1, \ldots, l)}_{z'_1, \ldots, z'_{\theta_l-r}}$   
was given. 
With the conditions
$z_{i'', i}\ne z_{j'', i}, \quad i''\ne j''$, 
 $1 \le i \le l$,  
$|z_{i'', i}|>|z_{k''', i}|>0$,   
 for $i''=1, \dots, m_1+\ldots+m_l$, 
and $k'''=m_1+\ldots+m_l+1, \dots, m_1+\ldots+m_l+ k_1+\ldots+k_l$,  
let us define 
\begin{eqnarray}
\label{perda}
 \mathcal R^{2, k_1+\ldots+k_l-r}_{m_1+\ldots+m_l-t}\left(\Phi^{(1, \ldots, l)} \right)  
= R\sum_{{q_1, \ldots, q_l}\in \C} 
 E^{(m_1+\ldots+m_l)}_{W^{(1, \ldots, l)}} \Big(  
v_1, z_1; \ldots; 
 && 
\nn
 \qquad v_{m_1+\ldots+m_l}, z_{m_1+\ldots+m_l};
P_{q_1, \ldots, q_l}\Big( \Phi^{(1, \ldots, l)}(g_1, \ldots, g_l; && 
\nn
 \qquad v_{m_1+\ldots+m_l+1}, z_{m_1+\ldots+m_l+1}; \ldots; &&
\nn
\qquad \qquad \qquad 
  v_{m_1+\ldots+m_l+k_1+\ldots+k_l}, z_{m_1+\ldots+m_l+k_1+\ldots+k_l};   
 \rho_1, \ldots, \rho_l)\Big), \quad &&
\end{eqnarray}
where $P_{q_1, \ldots, q_l}$ stands for projections  
$P_{q_i}: \overline{W}^{(i)} \to \W^{(i)}_{q_i}$
 on the corresponding subspaces in the tensor product $\W^{(1, \ldots, l)}$. 
In the Appendix \eqref{xusnya} defines 
 $E^{(m_1+\ldots+m_l)}_{W^{(1, \ldots, l)}}$. 
In order to get, in particular,
the adapted transversal of an element $\Phi$  
with extra vertex operators, 
$\mathcal R^{2, n}_m(\Phi)$ 
\eqref{Inm}
was introduced in Subsection \ref{composable}. 
We substitute the element $\Phi$
 by an element $\Phi^{(1, \ldots, l)}$ in $\Theta$. 
The absolute convergence of 
$\mathcal R^{2, k_1+\ldots+k_l-r}_{m_1+\ldots+m_l-t}\left(\Phi^{(1, \ldots, l)} \right)$
 defined by \eqref{perda} with \eqref{xusnya} 
provides the adapted transversal condition for $\Phi^{(1, \ldots, l)}$ with 
respect to a number of extra vertex operators in $\W^{(1, \ldots, l)}$. 
Using formulas provem above 
we have 
\begin{eqnarray*}
\left| 
\mathcal R^{2, k_1+\ldots+k_l-r}_{m_1+\ldots+m_l-t}
\left(  \Phi^{(1, \ldots, l)}  \right) 
\right|  
=\left|
\sum_{ {q_1, \ldots, q_l}\in \C} \mathcal R \prod_{i=1}^l 
 \langle w'_i,  E^{(m_i)}_{W^{(i)}} \Big(   
v_1, z_1; \ldots; 
v_{m_i}, z_{m_i}; \right. && 
\nn
\left.
    P_{q_1, \ldots, q_l}
\Big( \Phi^{(i)}(g_i;   
 v_{m_i+1}, z_{m_i+1}; \ldots;  
 v_{m_i+k_i}, z_{m_i+k_i})\Big) \Big) \rangle \right| 
  &&
\end{eqnarray*}
\begin{eqnarray*}
= \left| \sum_{q_1, \ldots, q_l \in \C}    
 \widehat{\mathcal R} \prod_{i=1}^l
\langle w'_i, 
  E^{(m_i)}_{W^{(i)}} \Big(  
v_{1, i}, x_{1, i}; \ldots;   
v_{m_i, i}, x_{m_i, i}; \right. \qquad \qquad \qquad   &&
\nn
  \left.  
P_{q_i} \Big(   
Y^{W^{(i)}}_{W^{(i)}V'} \left( 
\Phi^{(i)}(g_i; v_{m_i+1, i}, x_{m_i+1, i}; \ldots; 
v_{m_i+k_i, i}, x_{m_i+k_i, i}; u, \zeta_{1, i}), \zeta_{2, i} \right)
f_i.\overline{u} 
\; \Big) \Big)\rangle  
 \right| &&
\end{eqnarray*}
\begin{eqnarray*}
&& = \left| 
\sum_{q_1, \ldots, q_l \in \C}    
 \widehat{\mathcal R} \prod_{i=1}^l
\langle w'_i, 
  E^{(m_i)}_{W^{(i)}} \Big(  
v_{1, i}, x_{1, i}; \ldots;   
v_{m_i, i}, x_{m_i, i}; \right. 
\nn 
&& \qquad P_{q_i} \Big(   
e^{\zeta_{2, i}L_{W^{(i)}}(-1)}  
Y_{W^{(i)}}
 \left(f_i.\overline{u}, -\zeta_{2, i} \right) 
 \nn
 && \qquad \left.  
 \Phi^{(i)}(g_i; v_{m_i+1, i}, x_{m_i+1, i}; \ldots; 
v_{m_i+k_i, i}, x_{m_i+k_i, i}; u, \zeta_{1, i})
\; \Big) \Big)\rangle  
 \right| 
\le \left|
\mathcal R^{2, k_i}_{m_i}\left(\Phi^{(i)}\right) \right|,   
\end{eqnarray*}
where the invariance of \eqref{bardos} under 
$\sigma \in S_{m_1+\ldots+m_l-t+k_1+\ldots+k_l-r}$ was used. 
According to the definition,  
$\mathcal R^{2, k_i}_{m_i}\left(\Phi^{(i)}\right)$   
are absolute convergent. 
%
Thus, we infer that  
$\mathcal R^{2, k_1+\ldots+k_l-r}_{m_1+\ldots+m_l-t}\left(\Phi^{(1, \ldots, l}\right)$   
is absolutely convergent, and   
the sum \eqref{Inmdvadva} 
 is analytically extendable to a rational function  
in $(z_1$, $\dots$, $z_{k_1+\ldots+k_l-r+m_1+\ldots+m_l-t})$ 
with the only possible poles at 
$x_{j, i}=x_{j', i}$, and    
at $x_{j, i}=x_{j', i'}$, 
i.e., the only possible poles at 
$z_{i''}=z_{j''}$, of orders less than or equal to  
$N^{k_1+\ldots+k_l}_{m_1+\ldots+m_l}(v_{i'', i}, v_{j'', i})$,  
for $i''$, $j''=1, \ldots, k'''$, $i''\ne j''$.   
This finishes the proof of Proposition \ref{ccc}. 
\end{proof}
Since we have proved that the sequence of 
 products $\widehat{\Theta}$ $\left( f_1 \right.$, 
$\ldots$, $f_l$; $g_1$, $\ldots$., $g_l$;  
 $v_{1, 1}$, $x_{1, 1}$;     
$\ldots$; $v_{k_l, l}$,  $x_{k_l, l}$;   
 $\rho_1$, $\ldots$, $\rho_l$;  $\zeta_{1, i}$, $\zeta_{2, i}$  $5\left. \right)$ is  
adapted transversal to $\mu_l-t$ vertex operators  
\eqref{poper} with the formal parameters identified 
with the local coordinates $c_{j, i}(p''_{j, i})$ 
 around the points $(p'_1, \ldots, p'_{\mu_l-t})$  
on each of the transversal sections
$U_{j, i}$, $1 \le j \le \mu_l-t$, we conclude that  
according to the definition, 
the sequence of products $\widehat{\Theta}$ 
$\left(f_1, \ldots, f_l; g_1, \ldots, g_l \right.$;  
 $v_{1, 1}, x_{1, 1}; \ldots; v_{k_l, l}, x_{k_l, l}$;  
 $\rho_1, \ldots, \rho_l$; $\left. \zeta_{1, i}, \zeta_{2, i} \right)$ 
 belongs to the space    
\begin{eqnarray}
\label{ourbicomplex1110000}
 C^{\theta_l-r}_{\mu_l-t}\left(V, \W^{(1, \ldots, l)} \right)  
  =  \bigcap_{ 
U_{1, i} \stackrel{h_{1, i}} {\hookrightarrow}    
\ldots \stackrel {h_{m_1+\ldots +m_l-1, i}}{\hookrightarrow } U_{m_1+\ldots+m_l}    
\atop 1 \le j \le m_1+\ldots+m_l-t}   
 C^{\theta_l-r}_{(\mu-t)}\left(V, \W^{(1, \ldots, l)} \right) (U_{j, i}), &&       
\nn &&
\end{eqnarray}
where the intersection ranges over all possible $\mu_l-t$-tuples 
of holonomy embeddings $h_{j, i}$, $1 \le j \le  \mu_l-t-1$,     
between transversal sections $U_{1, i}, \ldots, U_{\mu_l-t-1, i}$ 
of the basis $\U$  for $\F$.    
This completes the proof of Proposition \ref{tolsto}. 
\end{proof}
Since the sequence of products 
\eqref{gendef} of $\W^{(i)}$-spaces, $1\le i \le l$, 
gives the tensor products of that spaces, 
the sequence of products \eqref{toporno} of 
the corresponding $C^{k_i}_{m_i}\left(V, \W^{(i)} \right)$-spaces 
belong to the same type of spaces.  
\section{Properties of multiple products sequences}   
\label{pyhva}
Since the sequence of $(\rho_1, \ldots, \rho_l)$-products of elements  
$\Phi^{(i)}$ $(g_i$; $ v_{1, i}, x_{1, i}$; $\ldots$; $v_{k_i,i}, x_{k_i, i})$  
 $\in$  $C^{k_i}_{m_i}\left(V, \W^{(i)} \right)$   
 results in an element of $C^{\theta_l-r}_{\mu_l-t}  
\left(V, \W^{(1, \ldots, l)}, \F \right)$, then the 
corollary below follows
 directly from Proposition \eqref{tolsto}: 

\subsection{Formal parameters invariance}
\label{dondo}
According to Proposition \ref{pupa}, elements of the space 
 \\$\W^{(1, \ldots, l)}_{z_1, \ldots, z_{\theta_l-r} }$     
resulting from the sequence of $(\rho_1, \ldots, \rho_l)$-products 
\eqref{gendef}, \eqref{gendefgen}  are   
invariant with respect to group  
$\left({\rm Aut} \;  \Oo\right)^{\times (\theta_l-r)}_{z_1, \ldots, z_{\theta_l-r} }$ 
of independent changes of the formal parameters. 
It is easy to derive 
\begin{corollary}
For $\Phi^{(i)}(g_i; v_{1,1}, x_{1,1};  \ldots; v_{k_i, i}, x_{k_i, i}) 
 \in C_{m_i}^{k_i}\left(V, \W^{(i)} \right)$    
the sequence 
\begin{eqnarray}
\label{posta}
&& \widehat{\Theta} \left(f_1, \ldots, f_l; g_1, \ldots, g_l;  
 v_{1,1}, x_{1,1};  \ldots; v_{k_l,l},  x_{k_l,l};  
\rho_1, \ldots, \rho_l; \zeta_{1, i}, \zeta_{2, i} \right) 
\nn
&& \qquad \qquad 
=
\left( \Phi^{(i)}(g_i; v_{1,1}, x_{1,1};  \ldots; v_{k_i, i}, x_{k_i, i}) \right)_k, 
\end{eqnarray} 
is invariant with respect to the action of the group 
 $\left({\rm Aut} \; \Oo\right)^{\times (\theta_l-r)}_{z_1, \ldots, z_{\theta_l-r}}$
\begin{eqnarray}
(z_1, \ldots, z_{\theta_l-r})  
\mapsto (\widetilde{z}_1, \ldots, \widetilde{z}_{\theta_l-r})    
= 
(\varrho(z_1), \ldots, \varrho(z_{\theta_l-r})).  && 
\end{eqnarray}  
\hfill $\square$
\end{corollary}
\subsection{Leibniz rule for the multiple product}   
\label{toto}
 In Proposition \ref{tolsto} we proved that 
the sequence of multiple products \eqref{bardos} of 
  spaces $C_{m_i}^{k_i}\left(V, \W^{(i)} \right)$ elements 
 belongs to $C^{\theta_l-r}_{\mu_l-t}$ $\left(V, \W^{(1, \ldots, l)} \right)$.  
Thus, the product admits the action of the coboundary operators 
$\delta^{\theta_l-r}_{\mu_l-t}$ and $\delta^{2-r, i}_{ex-t, i}$  
defined in  
\eqref{deltaproduct} and \eqref{halfdelta}.  
As we showed in Subsection \ref{adapta},  
 in contrast to the case of $\W^{(i)}$-spaces, 
where the sequence of $(\rho_1, \ldots, \rho_l)$-products 
leads to the tensor product $\W^{(1, \ldots, l)}$,  
the products \eqref{toporno} of $C^{k_i}_{m_i}$-spaces result in the same kind of 
space $C^k_m\left(V, \W^{(1, \ldots, l)} \right)$ 
defined on $\W^{(1, \ldots, l)}$.   
%
 The coboundary operators \eqref{deltaproduct}, \eqref{halfdelta}  
 have a version of Leibniz law with respect to the product 
\eqref{bardos}.
We will use it in Section \ref{gv} while deriving the cohomology classes. 
 Recall the notations $n_i$ of Subsection \ref{kapusta}.   
\begin{proposition}
\label{tosya}
For $\Phi^{(i)}(g_i; v_{1, 1}, x_{1, i};  \ldots; v_{k_i, i}, x_{k_i, i})  
\in C_{m_i}^{k_i}\left(V, \W^{(i)} \right)$, $1\le i \le l$,  
the action of the coboundary operator  
$\delta_{\mu_l-t}^{\theta_l-r}$ \eqref{deltaproduct}  
(and $\delta^{2-r, i}_{ex-t, i}$ \eqref{halfdelta})  
 on the sequence of $(\rho_1, \ldots, \rho_l)$-products 
\eqref{bardos}, $l \ge 1$, is given by 
\begin{eqnarray}
\label{leibniz}
 \delta_{\mu_l-t}^{\theta_l-r}    
 \Theta(f_1, \ldots, f_l; g_1, \ldots, g_l; z_1, v_1; 
\ldots; v_{\theta_l-r}, z_{\theta_l-r}; 
\rho_1, \ldots, \rho_l; \zeta_{1, i}, \zeta_{2, i} )_k \qquad \qquad   && 
\nn
 = 
 \sum\limits_{i=1}^l \cdot_{\rho_1, \ldots, \rho_l}  (-1)^{k_i-r_i}
\delta^{k_i-r_i}_{m_i-t_i}  
\Phi^{(i)} (g_i; v_{n_i+1}, z_{n_i+1};  \ldots; v_{n_i+k_i-r_i}, z_{n_i+k_i-r_i})_k. \quad &&
\end{eqnarray}
\end{proposition}
\begin{proof}
Due to \eqref{deltaproduct} 
the action of 
$\delta_{\mu_l-t}^{\theta_l-r}$  
on $\Theta(f_1, \ldots, f_l$; $ g_1, 
\ldots, g_l$; $ z_1, v_1$;$ \ldots$; $  v_{\theta_l-r}, z_{\theta_l-r}$; 
$\rho_1, \ldots, \rho_l$; $ \zeta_{1, i}, \zeta_{2, i} )_k$,  
is given by  
(we assume, as before, that the vertex operator 
$\omega_V (v_j, z_j - z_{j+1})$   
does not act on $(u, \zeta_{1, i})$)   
\begin{eqnarray*}
  && \delta_{\mu_l-t}^{\theta_l-r}  
\Theta(f_1, \ldots, f_l; g_1, \ldots, g_l; v_1, z_1; \ldots; 
 v_{\theta_l-r}, z_{\theta_l-r}; \rho_1, \ldots, \rho_l; \zeta_{1, i}, \zeta_{2, i})_k 
\nn
  &&=
  \sum_{j=1}^{\theta_l-r}(-1)^j \;    
   \Theta (f_1, \ldots, f_l; g_1, \ldots, g_l; 
v_1, z_1; \ldots;  v_{j-1}, z_{j-1}; 
\nn 
  && \omega_V (v_j, z_j - z_{j+1}) v_{j+1}, z_{j+1}; v_{j+2}, z_{j+2};   
\ldots;  v_{\theta_l-r}, z_{\theta_l-r};    
\rho_1, \ldots, \rho_l; \zeta_{1, i}, \zeta_{2, i})_k  
\end{eqnarray*}
\begin{eqnarray*}
   +   
  \Theta (f_1, \ldots, f_l; g_1, \ldots, g_l; \omega_{W^{(1)}} \left(v_1, z_1\right); 
v_2, z_2;  
\ldots; v_{\theta_l-r}, z_{\theta_l-r}; 
 \rho_1, \ldots, \rho_l; \zeta_{1, i}, \zeta_{2, i})_k    &&
\end{eqnarray*}
\begin{eqnarray*}
  && + (-1)^{\theta_l-r+1}      
 \Theta(f_1, \ldots, f_l; g_1, \ldots, g_l; 
\omega_{W^{(l)}}( v_{\theta_l-r+1}, z_{\theta_l-r+1});    
v_1, z_1;  
 \ldots; 
\nn
  && v_{\theta_l-r}, z_{\theta_l-r}; 
 \rho_1, \ldots, \rho_l; \zeta_{1, i}, \zeta_{2, i} )_k. 
\end{eqnarray*} 
Recall the definition of the enumeration $n_i$ of $v$ and $z$-parameters 
defined in Subsection \ref{kapusta}. 
Using \eqref{gendef} we see that the above is equivalent to 
\begin{eqnarray*}
      \sum_{j=1}^{\theta_l-r}(-1)^j        
\mathcal R \prod_{i=1}^l
 \rho^k_i 
 \langle w'_i, Y^{W^{(i)}}_{W^{(i)}V'} \left(  
\Phi^{(i)}(g_i; v_{n_i+1}, z_{n_i+1};  \ldots;     
\right.
 \qquad \qquad \qquad \qquad  &&
\nn 
\left. 
\; \omega_V (v_j, z_j - z_{j+1})\; v_{j+1}, z_{j+1}; 
 v_{j+2}, z_{j+2};   
\ldots;   
v_{n_i+k_i-r_i}, z_{n_i+k_i-r_i};       
 u, \zeta_{1, i} ), 
\zeta_{2, i} \right)\; f_i.\overline{u} \rangle,  &&
\nn
 +   
\mathcal R  \prod_{i=1}^l
 \rho^k_i  \langle w_i', 
  Y^{W^{(i)}}_{W^{(i)}V'} \left( \left(
 \left(\omega_{W^{(1)}} (v_1, z_1) \right)^{\delta_{i, 1}} \right.  \right. 
\qquad \qquad \qquad \qquad \qquad \qquad  &&
\nn
\left. \left. \qquad \qquad 
\Phi^{(i)}(g_i; v_{n_i+1+ \delta_{i, 1}}, z_{n_i+1+\delta_{i, 1}};  \ldots;
v_{n_i+k_i-r_i}, z_{n_i+k_i-r_i};  
  u, \zeta_{1, i})\right),  
\zeta_{2, i} \right)\; f_i.\overline{u}   \rangle  &&
\nn
\qquad  +   (-1)^{\theta_l-r+1}       
\mathcal R  \prod_{i=1}^l
  \rho^k_i   \langle w_i', 
  Y^{W^{(i)}}_{W^{(i)}V'} \left(  \right. 
 \left(\left(\omega_{W^{(l)}}(v_{n_{i+1}+1}, z_{n_{i+1}+1}) 
\right)^{\delta_{i,1}} \right. &&
\end{eqnarray*}
\begin{eqnarray}
\label{partyanki}
&& \quad  \left. \left. 
  \Phi^{(i)}(g; v_{n_i+1}, z_{n_i+1};  \ldots;      
 v_{n_i+k_i-r_1}, z_{n_i+k_i-r_i};  
  u, \zeta_{1, i})\right),  
\zeta_{2, i} \right)\; f_i.\overline{u}  \rangle. 
\end{eqnarray}
Consider the third term in \eqref{partyanki} 
\begin{eqnarray*}
 \sum\limits_{s=2}^l  
\mathcal R  \prod_{i=1}^l
 \rho^k_i \langle w_i', 
  Y^{W^{(i)}}_{W^{(i)}V'} \left( \left(
 \left( \omega_{W^{(s)}}(v_{n_{i+1}+1},  
z_{n_{i+1}+1})   \right)^{\delta_{s, i}} \right. \right. 
\qquad \qquad \qquad &&
\nn
\left. \left. 
\Phi^{(i)}(g_i; v_{n_i+1}, z_{n_i+1};  \ldots;       
 v_{n_i+k_i-r_i}, z_{n_i+k_i-r_i};   
  u, \zeta_{1, i}) \right), 
\zeta_{2, i} \right)\; f_i.\overline{u}   \rangle \qquad  &&
\nn 
=
 \sum\limits_{s=2}^l  
\mathcal R  \prod_{i=1}^l 
 \rho^k_i \langle w_i', e^{ \zeta_{2, i} L_{ W^{(i)} }(-1) }   
  Y_{W^{(i)}} \left( f_i.\overline{u}, -\zeta_{2, i} \right) \;  
\left(
\omega_{W^{(s)}}(v_{n_{i+1}+1}, z_{n_{i+1}+1})  
\right)^{\delta_{s, i}}    &&
\nn
 \Phi^{(i)} (g_i; 
v_{n_i+1}, z_{n_i+1};  \ldots;      
 v_{n_i+k_i-r_1}, z_{n_i+k_i-r_i};     
  u, \zeta_{1, i}) 
  \rangle  \qquad &&     
\nn
=
 \sum\limits_{s=2}^l  
\mathcal R  \prod_{i=1}^l 
 \rho^k_i \langle w_i', e^{ \zeta_{2, i} L_{ W^{(i)} }(-1) }   
\left(
\omega_{W^{(s)}}(v_{n_{i+1}+1}, z_{n_{i+1}+1})  
  \right)^{\delta_{s, i}}  
  Y_{W^{(i)}} \left( f_i.\overline{u}, -\zeta_{2, i} \right)  &&
\nn
\Phi^{(i)} ( 
g_i; 
 v_{n_i+1}, z_{n_i+1};  \ldots;      
 v_{n_i+k_i-r_i}, z_{n_i+k_i-r_i};   
  u, \zeta_{1, i}) 
  \rangle. &&      
\end{eqnarray*}
Due to the definition \eqref{wprop} of the intertwining operator 
and the locality property 
 of vertex operators we obtain 
\begin{eqnarray*} 
 && \sum\limits_{s=2}^l 
\mathcal R  \prod_{i=1}^l 
 \rho^k_i \langle w_i',  
\left(
\omega_{W^{(s)}}(v_{n_{i+1}+1}, z_{n_{i+1}+1}+\zeta_{2, i})  
  \right)^{\delta_{s, i}}  
e^{ \zeta_{2, i} L_{ W^{(i)} }(-1) }  
\nn 
&& \qquad Y_{W^{(i)}} \left( f_i.\overline{u}, -\zeta_{2, i} \right)  
\Phi^{(i)} (
g_i; 
v_{n_i+1}, z_{n_i+1};  \ldots;       
  v_{n_i+k_i-r_i}, z_{n_i+k_i-r_i};     
  u, \zeta_{1, i}) 
  \rangle. 
\end{eqnarray*}
The insertion an arbitrary vertex algebra module  
$W^{(i)}$-basis $\widetilde{w}_i$, and use of   
the definition of the intertwining operator \eqref{wprop} results  
\begin{eqnarray*} 
  &&\sum\limits_{\widetilde{w}_i\in W^{(i)}}  
\sum\limits_{s=2}^l 
\mathcal R  \prod_{i=1}^l 
 \rho^k_i \langle w'_i,   
 \left(\omega_{W^{(s)}} (
 v_{n_{i+1}+1}, z_{n_{i+1}+1}+\zeta_{2, i}) \right)^{\delta_{s, i}} 
\widetilde{w}_i\rangle 
\nn
&& \langle  \widetilde{w}'_i, 
e^{ \zeta_{2, i} L_{ W^{(i)} }(-1) }  
  Y_{W^{(i)}} \left( f_i.\overline{u}, -\zeta_{2, i} \right)  
\Phi^{(i)}(
g_i; 
 v_{n_i+1}, z_{n_i+1};  \ldots;
\nn       
 && \qquad \qquad \qquad \qquad \qquad \qquad \qquad
v_{n_i+k_i-r_i}, z_{n_i+k_i-r_i};  
  u, \zeta_{1, i}) 
  \rangle   
\end{eqnarray*}
\begin{eqnarray*} 
=
\sum\limits_{s=2}^l \sum_{
\widetilde{w}_i\in W^{(i)}
 \atop k \in \Z}    
\mathcal R  \prod_{i=1}^l 
 \rho^k_i 
\langle \widetilde{w}'_i, 
  Y^{W^{(i)}}_{W^{(i)}V'} \left( 
\Phi^{(i)} (
g_i; 
 v_{n_i}, z_{n_i};  \ldots;  \right. \qquad \qquad  \qquad \qquad && 
\nn
   \left. \qquad 
 v_{n_{i+1}-1}, z_{n_{i+1}-1}; 
  u, \zeta_{1, i}), \zeta_{2, i} \right) f_i.\overline{u}  
  \rangle   
\langle w'_i,   
 \left(\omega_{W^{(s)}} (
 v_{n_{i+1}}, z_{n_{i+1}}+\zeta_{2, i}) \right)^{\delta_{s, i}} \widetilde{w}_i\rangle      
&&
\end{eqnarray*}
\begin{eqnarray*} 
&& = 
\sum\limits_{\widetilde{w}_i\in W^{(i)}}  
\sum\limits_{s=2}^l  
\mathcal R  \prod_{i=1}^{l-1} 
 \rho^k_{i+1} 
\langle w'_i, 
\left(\omega_{W^{(s)}} (
 v_{n_{i+1}-1}, z_{n_{i+1}-1}+\zeta_{2, i} \right)^{\delta_{s, i}}
  \widetilde{w}_i \rangle  
\nn
  && \langle  w'_{i+1}, 
  Y^{W^{(i+1)}}_{W^{(i+1)}V'} \left(
\Phi^{(i+1)} (
g_{i+1}; 
 v_{n_{i+1}}, z_{n_{i+1}};  \ldots;   \right. 
\nn   
 && \qquad \qquad \qquad \qquad \left. 
v_{n_{i+2}-1}, z_{n_{i+2}-1};  
  u, \zeta_{1, i+1}), \zeta_{2, i+1} \right) f_{i+1}.\overline{u}   
  \rangle 
\nn 
&&
 = 
\sum\limits_{s=2}^l 
\sum_{
\widetilde{w}_i\in W^{(i)} }   
\mathcal R  \prod_{i=1}^{l-1} 
 \rho^k_{i+1} 
\langle w'_i, 
 \left(\omega_{W^{(s)}} (
 v_{n_{i+1}-1}, z_{n_{i+1}-1}+\zeta_{2, i} \right)^{\delta_{s, i}} 
\nn
 && \qquad \qquad \qquad \qquad \qquad \qquad 
Y^{W{(i)}}_{W^{(i)} W^{(i+1)}}(\widetilde{w}_i, \zeta)\; 
w_{i+1}\rangle  
\end{eqnarray*}
\begin{eqnarray*}
 && \langle  w'_{i+1}, 
  Y^{W^{(i+1)}}_{W^{(i+1)}V'} 
\left( 
\Phi^{(i+1)} ( 
g_{i+1}; 
 v_{n_{i+1}}, z_{n_{i+1}};  \ldots;  \right.
\nn
&& \qquad \qquad \qquad \left.     
 v_{n_{i+2}-1}, z_{n_{i+2}-1};  
  u, \zeta_{1, i+1}), \zeta_{2, i+1} \right) f_{i+1}.\overline{u}   
  \rangle.    
\end{eqnarray*}
Now eliminate the basis $w_{i+1}$ to get
\begin{eqnarray*} 
=   
\sum\limits_{s=1}^l  
\mathcal R  \prod_{i=1}^{l-1} 
 \rho^k_{i+1} 
\langle w'_i, e^{-L_{W^{(s-1)}}(-1)(-z_{n_{i+1}-1}-\zeta_{2, i})} 
e^{L_{W^{(s-1)}}(-1)(-z_{n_{i+1}-1}-\zeta_{2, i})} &&
\nn
\left(\omega_{W^{(s-1)}} (
 v_{n_{i+1}-1},  z_{n_{i+1}-1}+\zeta_{2, i}) \right)^{\delta_{s, i+1}} 
 Y^{W{(i)}}_{W^{(i)} W^{(i+1)}}(\widetilde{w}_i, \zeta) 
\qquad \qquad  \qquad \qquad \qquad \qquad  && 
\nn
  \; Y^{W^{(i+1)}}_{W^{(i+1)}V'} \left(  
\Phi^{(i+1)} (
g_{i+1}; 
 v_{n_{i+1}}, z_{n_{i+1}};  \ldots;      
 v_{n_{i+2}-1}, z_{n_{i+2}-1};  
  u, \zeta_{1, i+1}), \zeta_{2, i+1} \right) f_{i+1}.\overline{u}    
  \rangle   &&
\nn
=   
\sum\limits_{s=1}^l  
\mathcal R  \prod_{i=1}^{l-1} 
 \rho^k_{i+1} 
\langle w'_i, e^{-L_{W^{(s-1)}}(-1)(-z_{n_{i+1}-1}-\zeta_{2, i})} 
  \left(Y^{W^{(i)}}_{W^{(i)}W^{(i)}}  ( \right. 
Y^{W{(i)}}_{W^{(i)} W^{(i+1)}}(\widetilde{w}_i, \zeta) && 
\nn
   Y^{W^{(i+1)}}_{W^{(i+1)}V'} \left(  
\Phi^{(i+1)} (
g_{i+1}; 
 v_{n_{i+1}}, z_{n_{i+1}};  \ldots;      
 v_{n_{i+2}-1}, z_{n_{i+2}-1};  
  u, \zeta_{1, i+1}), \zeta_{2, i+1} \right) f_{i+1}.\overline{u}, \qquad &&
\nn
\left. 
   -\zeta) \right)^{\delta_{s, i+1}} v_{n_{i+1}-1}    
  \rangle   &&
\end{eqnarray*}
\begin{eqnarray*} 
&&
=   
\sum\limits_{s=1}^l  
\mathcal R\prod_{i=1}^{l-1} 
 \rho^k_{i+1} 
\langle w'_i, e^{-L_{W^{(s-1)}}(-1)(-z_{n_{i+1}-1}-\zeta_{2, i})} 
\nn
&&
  \left(Y^{W^{(i)}}_{W^{(i)}W^{(i)}}  ( 
Y^{W{(i)}}_{W^{(i)} W^{(i+1)}}(\widetilde{w}_i, \zeta)  
e^{L_{W^{(i+1)}}(-1)(-\zeta_{2, i+1})}
   Y_{W^{(i+1)}}(v_{n_{i+1}-1},  \zeta) \right.
\nn
&&
\left.  
\Phi^{(i+1)} (
g_{i+1}; 
 v_{n_{i+1}}, z_{n_{i+1}};  \ldots;      
 v_{n_{i+2}-1}, z_{n_{i+2}-1};  
  u, \zeta_{1, i+1} ) f_{i+1}.\overline{u}, -\zeta)
   \right)^{\delta_{s, i+1}}    
  \rangle   
\nn
&&
=   
\sum\limits_{s=1}^l   
\mathcal R \prod_{i=1}^{l-1} 
 \rho^k_{i+1} 
\langle w'_{i+1}, e^{-L_{W^{(i)}}(-1)(-z_{n_{i+1}-1}-\zeta_{2, i})} 
\nn
&&
e^{L_{W^{(i+1)}}(-1)(-\zeta_{2, i+1})}
   Y_{W^{(i+1)}}(v_{n_{i+1}-1},  \zeta)
\nn
&&
 \left.  
\Phi^{(i+1)} (
g_{i+1}; 
 v_{n_{i+1}}, z_{n_{i+1}};  \ldots;      
 v_{n_{i+2}-1}, z_{n_{i+2}-1};  
  u, \zeta_{1, i+1}) f_{i+1}.\overline{u}, -\zeta)
   \right)^{\delta_{s, i+1}} 
  \rangle,   
\end{eqnarray*}
where $\zeta=-z_{n_{i+1}-1}-\zeta_{2, i}$. 
Above we have made use of 
 the commutativity of $L_{W^{(i)}}(-1)$ and $L_{W^{(i+1)}}(-1)$, 
and the formula relating the intertwining operators in the adjoint positions. 
Due to locality 
of vertex operators, 
and arbitrariness of $v_{k+1}\in V$ and $z_{k+1}$, 
it is always possible to take 
$\omega_{W^{(s-1)}} (
 v_{n_{i+1}-1}, z_{n_{i+1}-1}+\zeta_{2, i-1} - \zeta_{2, i+1})  
  =\omega_{W^{(s-1)}}(v_{n_{i+1}}, z_{n_{i+1}})$,     
for $v_{n_{i+1}}= v_{n_{i+1}-1}$, 
 $z_{n_{i+1}}= z_{n_{i+1}-1}+\zeta_{2, i-1} - \zeta_{2, i+1}$.   
We repete the same operations with the second term of \eqref{partyanki}. 
Combining the action of $\delta^{k_i}_{m_i}$ on $\Phi^{(i)}$,  
gives \eqref{leibniz} due to \eqref{gendef}, \eqref{gendefgen}.  
The statement of the proposition for $\delta^{2, i}_{ex, i}$
 \eqref{halfdelta} can be checked in the similar way.  
\end{proof}
Next, we prove the following 
\begin{proposition}
The sequence of products \eqref{bardos} extends the property  
\eqref{deltacondition} 
of the families of cochain  
complexes 
\eqref{hat-complex} and \eqref{hat-complex-half}  
to all sequences of products 
$\cdot_{\rho_1, \ldots, \rho_l}$ $C^{k_i}_{m_i}\left(V, \W^{(i)} \right)$,   
$k_i \ge0$, $m_i \ge0$, $1 \le i \le l$.  
\end{proposition}
\begin{proof}
For $\Phi^{(i)} \in C^{k_i}_{m_i}\left(V, \W^{(i)} \right)$ 
 we proved in Proposition \ref{tolsto} that the sequence of products  
$\cdot_{\rho_1, \ldots, \rho_l}\left( \Phi^{(i)} \right)$
 belongs to the spaces   
$C^{\theta_l-r}_{\mu_l-t}\left(V, \W^{(i)} \right)$.   
Using \eqref{leibniz} and 
the cochain property for $\Phi^{(i)}$ we see that 
\begin{eqnarray*}
\delta^{\theta_l-r+1}_{\mu-t-1} \circ  
 \delta^{\theta_l-r}_{\mu_l-t} \left( \cdot_{\rho_1, \ldots, \rho_l} \Phi^{(i)} \right)=0,   
\quad 
\delta^{2-r}_{ex-t} \circ \delta^{1-r}_{2-t} \left(
\cdot_{\rho_1, \ldots, \rho_l} \Phi^{(i)} \right)=0.  
\end{eqnarray*}
Thus, the cochain property extends to the sequence of 
 $(\rho_1, \ldots, \rho_l)$-products   
$\cdot_{\rho_1, \ldots, \rho_l}$ 
$\left( C^{k_i}_{m_i}\left(V\right. \right.$, $\left.\left. \W^{(i)} \right) \right)$.  
\end{proof}
Finally, for elements of the spaces $C^{2, i}_{ex, i}\left(V, \W^{(i)} \right)$
 we obtain 
\begin{corollary}
The product of elements of the spaces $C^2_{ex}$ $\left(V, \W^{(ex)} \right)$ 
and  $C^{k_i}_{m_i}$ $\left(V, \W^{(i)} \right)$ is given by   
\eqref{bardos}, 
\begin{eqnarray*}
\cdot_{\rho_1, \ldots, \rho_l}: 
\times_{i=1}^{l_1} C^{2, i}_{ex, i} \left(V, \W^{(i)} \right)
 \times_{j=1}^{l_2}C^{k_i, i}_{m_i} \left(V, \W^{(i)} \right) 
\to C^{k_1+\ldots+k_{l_2}+ 2l_1 -r, i}_{m_i-t, i} \left(V, \W^{(i)} \right),  &&
\end{eqnarray*}
\begin{equation}
\label{pupa3}
\cdot_{\rho_1, \ldots, \rho_l}: \times_{i=1}^l C^{2, i}_{ex, i} \left(V, \W^{(i)} \right) 
\to C^{4-r, i}_{0, i} \left(V, \W^{(i)} \right).    
\end{equation}
\end{corollary}
\begin{proof}
The number of formal parameters
 in the product \eqref{bardos}   is  
$k_1+\ldots+k_{l_2}+2l_1-r$. 
That follows from Proposition \eqref{derga}.   
Consider the product \eqref{bardos} for  
$C^{2, i}_{ex, i} \left(V, \W^{(i)} \right)$ and  
$C^{k_i}_{m_i} \left(V, \W^{(i)} \right)$.  
As in the proof of Proposition \ref{tolsto}, 
the total number $m_i-t$ of vertex operators the product 
$\Theta$ is adapted transversal is preserved.
Thus, we have to checked that on the right hand side of \eqref{pupa3} 
the number of vertex operators adapted transversal becomes $m_i-t$.   
\end{proof}
\section{The multiple-product cohomology classes}
\label{gv}
In this Section proofs of the main results of this paper are provided. 
In particular, we find invariant classes
 associated to the sequences of multiple products 
for a vertex algebra cohomology  
 for codimension one foliations.  
\subsection{The cohomology classes}
\label{cohomological}
In this Subsection, we introduce the 
cohomology classes 
 for codimension one foliations  
on complex curves associated to 
a grading-restricted vertex operator algebra.  
The cohomology classes for a codimension one foliation \cite{G, CM, Ko}  
were introduced starting with 
an extra transversality condition on 
 differential forms defining a foliation, and leading to 
the integrability condition.   
The elements of $\mathcal E$ in \eqref{deltaproduct} and $\mathcal E_{ex}$  
 are elements  
of spaces $C^{1, i}_{\infty, i}\left(V, \W^{(i)} \right)$ 
adapted transversal to 
an infinite number of vertex operators. 
The actions of coboundary operators $\delta^{k_i}_{m_i}$  
and $\delta^{2, i}_{ex, i}$ in \eqref{deltaproduct} and  
\eqref{halfdelta} are written as products  
 similar to as differential forms in Frobenius theorem \cite{G}.  
Using the sequence of multiple products  
we introduce cohomology     
classes of the form  
that are counterparts of the Godbillon class.      

 We call a map 
$\Phi^{(i)} \in C_{m_i}^{k_i}\left(V, \W^{(i)} \right)$,     
closed if it represents a closed connection 
$\delta^{k_i}_{m_i} \Phi^{(i)}=G\left(\Phi^{(i)}\right)=0$.  
For $m_i \ge 1$, we call it exact if there exists  
$\Psi^{(i)} \in  C_{m_i-1}^{k_i+1}\left(V, \W^{(i)} \right)$,    
such that 
$\Psi^{(i)}(v'_1, z'_1; \ldots; v'_{k_i+1}, z'_{k_i+1})=
\delta^{k_i}_{m_i} \Phi^{(i)}(v_1, z_1; \ldots; v_{k_i}, z_{k_i})$,  
 i.e., $\Psi^{(i)}$ is the form of a connection.   
For $\Phi^{(i)} \in C^{k_i}_{m_i}\left(V, \W^{(i)} \right)$
 we call the cohomology class of mappings 
 $\left[ \Phi^{(i)} \right]$ 
the set of all closed forms that differ from $\Phi^{(i)}$ by an  
exact mapping, 
i.e., for $\Lambda^{(i)} \in C^{k_i-1}_{m_i+1}\left(V, \W^{(i)} \right)$,   
$\left[ \Phi^{(i)} \right]= \Phi^{(i)} + \delta^{k_i-1}_{m_i+1} \Lambda^{(i)}$.     
The cohomology classes constructed in this paper
 are vertex algebra cohomology analogues  
 of the Godbillon class 
\cite{Ko}   
for codimension one foliations on complex curves.  
\subsection{Transversality conditions}
\label{multpartprod}
In this Subsection we consider the general classes of cohomology invariants which arise from 
 the definition of the product of pairs of  
$C^{k_i}_{m_i}\left(V, \W^{(i)} \right)$-spaces.  
Under a natural extra condition, the families 
 cochain complexes \eqref{hat-complex} and \eqref{hat-complex-half}  
allow us to establish relations among elements of $C^{k_i}_{m_i}\left(V, \W^{(i)} \right)$-spaces.  
By analogy with the notion of the integrability for differential forms \cite{G}, 
 we use here the notion of the transversality for the spaces of a complex. 
  
For the families cochain complexes \eqref{hat-complex} and \eqref{hat-complex-half}   
let us require that for cochain  complex spaces 
$C_{m_{i_j}}^{k_{i_j}}\left(V, \W^{(i_j)} \right)$,   
$1 \le i_1 < \ldots< i_j \le l$, $1 \le j \le k \le l$   
 there exist subspaces 
$\widetilde{C}_{m_i}^{k_i}
\left(V, \W^{(i)} \right)
\subset C_{m_i}^{k_i}\left(V, \W^{(i)} \right)$,  
 such that for   
 $\Phi^{(i_j)} \in \widetilde{C}_{m_{i_j}}^{k_{i_j}}\left(V, \W^{(i_j)} \right)$,  
and $1 \le n \le l$, 
$\left( \ldots, \delta^{k_{i_1}}_{m_{i_1}} \Phi^{(i_1)},  
\ldots, \delta^{k_{i_k}}_{m_{i_k}}      
\Phi^{(i_k)}, \ldots \right)=0$.    
Then we call the set of subspaces 
$\left\{ \widetilde{C}_{m_i}^{k_i}
\left(V, \W^{(i)} \right) \right\}$ 
orthogonal for all spaces 
 $C_{m_i}^{k_i}\left(V, \W^{(i)} \right)$, $i\ne i_j$ with respect to 
the product $\eqref{gendefgen}$.  
%
 Namely,  
$\delta^{k_{i_1}}_{m_{i_1}} \Phi^{(i_1)}$, $\ldots$, 
$\delta^{k_{i_j}}_{m_{i_j}} \Phi^{(i_j)}$,
 are supposed to be transversal to all other multiplicands  
 with respect to the product   
\eqref{product}.  
We call this the generalized transversality condition for mappings 
of the  families cochain complexes  
\eqref{hat-complex} and \eqref{hat-complex-half}.

In particular, the simplest case of the transversality  
is defined for some $1 \le i, p \le l$ by 
\begin{equation}
\label{ortho} 
\left( \ldots, \left(\delta^{k_i}_{m_i} \right)^{\delta_{i, p}}  
\Phi^{(i)}, \ldots \right)=0. 
\end{equation}
Note that in the case of differential forms considered on a smooth manifold, 
the Frobenius theorem for a distribution provides the transversality condition \cite{G}.   
The fact that both sides of a differential relation belong 
to the same cochain  complex space, applies limitations 
to possible combinations of $(k_i, m_i)$, $1 \le i \le j \le l$. 
Below we derive the algebraic relations occurring from the transversality condition on 
the families of cochain complexes \eqref{hat-complex} and \eqref{hat-complex-half}.  
Taking into account the correspondence    
with ${\rm \check {C}}$ech-de Rham complex due to \cite{CM}, 
we reformulate the derivation of the product-type invariants
in the vertex algebra terms. 
%
Recall that the Godbillon--Vey cohomology class \cite{G} 
is considered on codimension one foliations of 
three-dimensional smooth manifolds. 
In this paper, we supply its analogue for complex curves.  
 According to the definition \eqref{ourbicomplex} we 
have $m_i$-tuples of one-dimesional transversal sections. 
 In each section we attach one vertex operator
 $\omega_{W^{(i)}}(u_j, w_j)$, $u_{m_i} \in V$, $w_{m_i} \in U_{m_i, i}$,   
$1\le i \le l$, $1 \le j \le m_i$.    
Similarly to the differential forms setup, 
a mapping $\Phi^{(i)} \in 
 C_{m_i}^{k_i} \left(V, \W^{(i)} \right)$ defines  
a codimension one foliation. 
As we see from \eqref{gendef} and \eqref{leibniz}   
it satisfies the properties 
similar as differential forms do. 
 
Now, let us explain how we understand powers of 
 an element of $\W^{(i)}_{x_{1, i}, \ldots, x_{k_i, i}}$
in the multiple product \eqref{gendefgen}. 
Denote by 
$\Phi^{(i)}_{j_s}$ $=$ $\Phi^{(i)}$ $(g_i$;  $v_{1, i}$, $x_{1, i}$;   
$\ldots$; $v_{k_i, i}$, $)$ an element of 
$\W^{(i)}_{x_{1, i}, \ldots, x_{k_i, i}}$       
placed at a position $1 \le j_s \le l$, $1 \le s \le k$.  
We then have 
\begin{eqnarray}
\label{parasnya}
\left(\ldots, \left(\Phi^{(i)} \right)^k, \ldots \right)  
=  
\left(\ldots, \Phi^{(i)}_{j_1}, 
\ldots, \Phi^{(i)}_{j_2}, 
\ldots, \Phi^{(i)}_{j_r}, \ldots \right), \quad && 
\end{eqnarray}
with $\Phi^{(i)}$ placed at some positions 
$(j_1, \ldots, j_k)$.  
 
Letus introduce another kind of transversality conditions. 
We call the order ${\ord}\; \Phi$ of an element $\Phi$ in a product of the form 
\eqref{gendef} the number of appearance of $\Phi$. 
For two elements $\Phi$, $\Psi$ we can also define the mutual order 
as ${\ord}\; (\Phi, \Psi)=|{\ord}\; \Phi - {\ord}\; \Psi|$.  
\subsection{The commutator multiplications} 
\label{comult}
In this Subsection we define 
further multiple products  
of elements of the spaces $C^{k_i}_{m_i}\left(V, \W^{(i)} \right)$, $1 \le i \le l$,  
suitable for the formulation of cohomology invariants. 

For a set of indices 
$(i_1, i_2, i_{1, 2}, i_3, \ldots, i_{1, \ldots, l-1}, i_l)$   
ranging in $[1, \ldots, l]$, 
and corresponding complex parameters 
$(\rho_1, \rho_2, \rho_{1,2}, \ldots, \rho_{1, 2, \ldots, l-1}, \rho_l)$,   
 let us define  
 the additional 
 multiple products of elements    
  $\Phi^{(i)}$ $(g_i$; $v_{n_i+1}$, $z_{n_i+1}$;  
 $\ldots$ ; $v_{n_i+k_i-r_i}$, $z_{n_i+k_i- r_i})$      
$\in$ $C^{k_i}_{m_i}$ $\left(V \right.$, $\left. \W^{(i)} \right)$,  
as follows (for clarity of presentation, 
we omit here explicit dependence on the automorphism element, 
vertex algebra elements, formal parameters, and additional $\zeta$-parameters)  
 \begin{eqnarray}
\label{defproduct}
&&
*_{(i_1, i_2, i_{1, 2}, i_3, \ldots, i_{1, \ldots, l-1}, i_l)}:
 \times_{i=1}^l \W^{(i_p)}_{z_{1, i_p}, \ldots, z_{k_p, i_p}}  
\to \W^{(1, \ldots, l)}_{z_{k_1}, \ldots, z_{\theta_{l-r}}}, \;    
\end{eqnarray}
\begin{eqnarray*}
 && *_{(i_1, i_2, i_{1, 2}, i_3, \ldots, i_{1, \ldots, l-1}, i_l)}
\left(\Phi^{(i)}\right)_{1 \le i \le l}   
\nn
&& \quad = \left[ \left[ \ldots \left[\left[ \Phi^{(i_1)}
,_{ \cdot_{\rho_{i_1}, \rho_{i_2} }} \Phi^{(i_2)}\right]  
 ,_{ \cdot_{\rho_{i_{1, 2}}, \rho_{i_3} }} \Phi^{(i_3)}\right]  \ldots \right] 
   ,_{ \cdot_{\rho_{1, \ldots, l-1}, \rho_{i_l} }} \Phi^{(i_l)}\right],   
\end{eqnarray*}
where the brackets denote the commutator with respect to the  
$\cdot_{i_p, i_q}$-product defined on 
$\W^{(i_p)}_{z_{1, i_p}, \ldots, z_{k_p, i_p}} 
\times \W^{(i_q)}_{z_{1, i_q}, \ldots, z_{k_q, i_q}}$, 
$\left[\Phi^{(i_p)},_{i_p, i_q} \Phi^{(i_q)}\right] =  
\Phi^{(i_p)} \cdot_{\rho_{i_p}, \rho_{i_q}} \Phi^{(i_q)}   
- \Phi^{(i_q)} \cdot_{\rho_{i_q}, \rho_{i_p}} \Phi^{(i_p)}$,  
with respect to the $\cdot_{ \rho_{i_p}, \rho_{i_q} }$-product \eqref{gendef}. 

We are able to use also the total 
$(i_1$, $i_2$, $i_{1, 2}$, $\ldots$, $i_{1, 2, \ldots, i_{l-1}}$, $i_l)$-symmetrization 
\begin{equation}
\label{syma}
{\rm Sym}
\left(*_{(i_1, i_2, i_{1, 2}, \ldots, i_{1, 2, \ldots, i_{l-1}}, i_l)}
\left(\Phi^{(i)}\right)_{1 \le i \le l}\right),  
\end{equation}   
of the product \eqref{defproduct}.  
The form of \eqref{syma} is not unique of cause. 
We are able to form other types of products resulting from the products \eqref{gendef}. 
Nevertheless, \eqref{syma} is suitable for computation of cohomology invariants of 
foliations. 
Due to the properties of the maps $\Phi^{(i)}\in C_{m_i}^{k_i}\left(V, \W^{(i)} \right)$,  
$1 \le i \le l$ we obtain 
\begin{lemma}
\label{propo}
The products \eqref{syma}  
belong to the space  
$C_{\mu_l- t }^{\theta_l-r}$  
$\left(V \right.$ , $\W^{(1, \ldots, l)}$, $\left.\F \right)$.       
\hfill $\square$
\end{lemma}

For $i_p=i_q$, a self-dual bilinear pairing $\langle .,.\rangle$ 
for $W^{(i_p)}$, and 
$(g_{i_p}$; $v_{n_{i_p}}$, $z_{n_{i_p}}$;   
 $\ldots ; v_{n_{i_p+1}-1}, z_{n_{i_p+1}-1})$ $=$   
    $(g_{i_q}; v_{n_{i_q}}, z_{n_{i_q}}; 
 \ldots ; v_{n_{i_q+1}-1}, z_{n_{i_q+1}-1})$,  
 the product 
\begin{eqnarray}
\label{fifi}
&&
\Phi^{(i_p)} (g_{i_q}; v_{n_{i_q}}, z_{n_{i_q}}; 
 \ldots ; v_{n_{i_q+1}-1}, z_{n_{i_q+1}-1})
\nn
&&
\qquad *_{i_p,  i_q}  
 \Phi^{(i_p)} (g_{i_p}; v_{n_{i_p}}, z_{n_{i_p}}; 
 \ldots ; v_{n_{i_p+1}-1}, z_{n_{i_p+1}-1}) =0. 
\end{eqnarray}  
The product \eqref{defproduct} allows to introduce cohomology invariants 
associated with the condition \eqref{fifi} on $\Phi^{(i)}$.  
Namely, it is easy to prove the following 
\begin{proposition}
\label{kontora}
For the cochain  complex \eqref{hat-complex} elements 
$\Phi^{(i)} \in C^{k_i}_{m_i} \left(V, \W^{(i)} \right)$ 
 satisfying \eqref{fifi}
and the transversality condition 
\begin{equation}
\label{dundak}
\delta^{k_{i_s}}_{m_{i_s}} \partial_t \Phi^{(i_s)} 
*_{i_s, i_{s'}} \delta^{k_{i_{s'}}}_{m_{i_{s'}}}\Phi^{(i_{s'})}=0,  
\end{equation}
with $i_s, i_{s'}=i_p, i_q, i_r$, 
there exist the classes
 of non-vanishing cohomology invariants of the form 
$\left[ \delta^{k_{i_p}}_{m_{i_p}} \Phi^{(i_p)} 
*_{i_p, i_q} \left(\partial_t \Phi^{(i_q)}\right)^\beta 
*_{i_{p, q}, i_r} \Phi^{(i_r)} \right]$,    
not depending on the choice of $\Phi^{(i_s)}$. 
In particular, for the short complex \eqref{hat-complex-half}, 
one has 
$\left[ \delta^{1, i_p}_{2, i_p} \Phi^{(i_p)} \right.$ $*_{i_p, i_q}$      
 $\left(\Phi^{(i_q)} \right)$ $\left. *_{i_{p, q}, i_r} \right]$, 
$\left[ \delta^{0, i_p}_{3, i_p} \Lambda^{(i_p)} *_{i_p, i_q}    
 \left(\Lambda^{(i_q)} \right)^\beta 
*_{i_{p,q}, i_r} \Lambda^{(i_r)} \right]$,   
 are invariant, i.e., 
they do not depend on the choices of 
$\Phi^{(i_s)}$ $\in$ $C^{1, i_s}_{2, i_p}$ $\left(V, \W^{(i_s)} \right)$,    
$\Lambda^{(i_s)}$ $\in$ $C^{0, i_s}_{3, i_s}$ $\left(V, \W^{(i_s)} \right)$. 
\hfill $\square$ 
\end{proposition} 
\subsection{Proof of Theorem \ref{talasa}} 
\label{proofthe}
Now we show that the analog of the integrability condition 
provides the generalizations of the product-type invariants   
for codimension one foliations on complex curves. 
Here we give a proof of the main statement of this paper, 
Theorem \ref{talasa} formulated in the Introduction. 
\begin{proof}
Suppose we consider products containing elements 
$\Phi^{(i_s)}$, $\Psi^{(i_{s})}$  
$\in$ $C^{k_{i_s}}_{m_{i_s}}$ $\left(V \right.$, $\left. \W^{(i_s)} \right)$,
with $i_s=i, i', i''$, 
 with the mutual orders satisfying 
$\ord \left(\delta^{k_{i_s}}_{m_{i_s}} \Phi^{(i_s)}, \Psi^{(i_{s'})}\right)<m+k-1$. 
For elements $\Phi^{(i_s)} 
\in C^{k_{i_s}}_{m_{i_s}}\left(V, \W^{(i_s)} \right)$,   
for $1 \le i_s \le n$, 
let us start with 
 the foliation $\F$  
 transversality condition \cite{Ko}  
\begin{equation}
\label{prontosan}
\left( \delta^{k_{i_s}}_{m_{i_s}} \partial_t \Phi^{(i_s)},  
 \delta^{k_{i_{s'}}}_{m_{i_{s'}}} \Phi^{(i_{s'})} \right)=0. 
\end{equation} 
for any pair of $i_s$ and $i_{s'}$,  $1 \le i_s, i_{s'} \le n$. 
Then, due to associativity of the products \eqref{gendef}, \eqref{gendefgen} 
and the definition \eqref{parasnya} of an $\W$-element powers 
it follows that 
\begin{equation}
\label{kabkab}
\left( \delta^{k_{i_s}}_{m_{i_s}} \partial_t \Phi^{(i_s)},  
 \delta^{k_{i_{s'}}}_{m_{i_{s'}}} \left( \Phi^{(i_{s'})}  \right)^k \right)=0, 
\quad 
\left( \delta^{k_{i_s}}_{m_{i_s}} \partial_t \Phi^{(i_s)},  
 \left(  \delta^{k_{i_{s'}}}_{m_{i_{s'}}} \Phi^{(i_{s'})} \right)^k \right)=0.  
\end{equation}
It is clear that if one of multiplicand in the product \eqref{gendef} 
is zero then the product vanishes. 
Let us show that the invariant \eqref{pampadur} is closed. 
Due to \eqref{prontosan} (\eqref{kabkab} 
\begin{eqnarray*} 
&& \delta.\left( 
 \left( \delta^{k_i}_{m_i} \Phi^{(i)} \right)^m,     
   \left(\partial_t \Phi^{(i')}\right)^\beta, 
   \left( \Phi^{(i'')} \right)^k  
   \right), 
\nn
&&= \left(  
 (-1)^{k_i+1} \delta^{k_i+1}_{m_i-1}. \left( \delta^{k_i}_{m_i} \Phi^{(i)} \right)^m,    
   \left(\partial_t \Phi^{(i')}\right)^\beta, 
   \left( \Phi^{(i'')} \right)^k  
  \right)
\nn
&&+\left( 
  \left( \delta^{k_i}_{m_i} \Phi^{(i)} \right)^m,   
  (-1)^{k_{i'}} \delta^{k_{i'}}_{m_{i'}}. \left(\partial_t \Phi^{(i')}\right)^\beta, 
   \left( \Phi^{(i'')} \right)^k   
    \right)
\nn
&& +\left( 
 \left( \delta^{k_i}_{m_i} \Phi^{(i)} \right)^m,   
 \left(\partial_t \Phi^{(i')}\right)^\beta, 
  (-1)^{k_{i''}}\delta^{k_{i''}}_{m_{i''}}. \left( \Phi^{(i'')} \right)^k  
    \right)=0,  
\end{eqnarray*}
i.e., \eqref{pampadur} is closed.  
Let us show   
 non-vanishing property of \eqref{pampadur}.  
Indeed, suppose 
$\left(    
 \left(\delta^{k_i}_{m_i} \Phi^{(i)}\right)^m,   
 \left(\partial_t \Phi^{(i')} \right)^\beta,   
   \left( \Phi^{(i'')} \right)^k     
   \right)=0$.  
 Then there exists 
$\Gamma^{(i)} \in C^n_\mu \left(V, \W^{(i)} \right)$, 
such that 
 $P^{(i)}_{(i, i, i', i'')}\delta^{k_i}_{m_i} \Phi^{(i)} =\left(      
 \Gamma^{(i)},  \left(\delta^{k_i}_{m_i} \Phi^{(i)}\right)^{m-1},     
  \left(\partial_t \Phi^{(i')} \right)^\beta,    
   \left( \Phi^{(i'')} \right)^k     
     \right)$,  
where $P^{(i)}_{(i, i, i', i'')}$ is the projection 
$P^{(i)}_{(i, i, i', i'')}: \W^{(i)} \to \W^{(i, i, i', i'')}$.  
 Both sides of the last equalities should belong to the same cochain  complex space.  
Indeed, 
$k_i+1 = n +(m-1)(k_i+1) + \beta k_{i'}+ k k_{i''}$, 
$m_i-1= \mu + (m-1)(m_i-1) + \beta m_{i'}+ k m_{i''}$. 
For a non-vanishing expression, 
$n$ or $\mu$ should be negative. 
Then we obtain  
$(2-m)k_i -m +1   - \beta k_{i'} - k k_{i''}<0$, 
and 
$(2-m)m_i + m-1 - \beta m_{i'} - k m_{i''} <0$. 
Now let us show that \eqref{pampadur} 
 is an invariant, i.e., 
it does not depend on the choice of 
$\Phi^{(i)} \in C^{k_i}_{m_i} \left(V, \W^{(i)} \right)$. 
 Substitute elements the $\Phi^{(i)}$, $\Phi^{(i')}$, $\Phi^{(i'')}$  by  
 elements added by 
$\eta^{(i)} \in C^{k_i}_{m_i} \left(V, \W^{(i)} \right)$, 
$\eta^{(i')} \in C^{k_{i'}}_{m_{i'}} \left(V, \W^{(i')} \right)$,
 $\eta^{(i'')} \in C^{k_{i''}}_{m_{i''}} \left(V, \W^{(i'')} \right)$, correspondingly.   
Since the multiple product is associative, 
we obtain 
\begin{eqnarray*}
 \left(  
 \left(\delta^{k_i}_{m_i}  \Phi^{(i)} + \delta^{k_i}_{m_i} \eta^{(i)}\right)^m,   
  \partial_t \left( \Phi^{(i')} + \eta^{(i')}\right)^\beta, 
    \left( \Phi^{(i'')} + \eta^{(i'')} \right)^k     
   \right)
\qquad \qquad \qquad \qquad \qquad \qquad &&
\nn
= \sum\limits_{{j=0, \atop j'=0}}^{m, k} C^{j, j'}_{m, k}
\left(  \left(\delta^{k_i}_{m_i}  \Phi^{(i)}\right)^j, 
 \left(\delta^{k_i}_{m_i} \eta^{(i)}\right)^{m-j},    
 \partial_t  \left( \Phi^{(i')} + \eta^{(i')}\right)^\beta, 
   \left(\Phi^{(i'')}\right)^k,  
\left(\eta^{(i'')}\right)^{k-j'} \right),  \qquad  &&
\end{eqnarray*}
where 
$C^{j, j'}_{m, k}=\left({m \atop j}\right)\left({k \atop j'} \right)$. 
The expression above splits in two parts relative to 
$\Phi^{(i)}$ and $\eta^{(i)}$. 
\begin{eqnarray*}
\left( \left(\delta^{k_i}_{m_i}  \Phi^{(i)} \right)^m,
  \left( \partial_t \Phi^{(i')} \right)^\beta, 
\left(\Phi^{(i'')} \right)^k 
  \right) 
+ \left( \left(\delta^{k_i}_{m_i}  \Phi^{(i)} \right)^m,
 \partial_t \left( \eta^{(i')}\right)^\beta, 
\left(\Phi^{(i'')} \right)^k 
  \right) \qquad \qquad \qquad && 
\end{eqnarray*}
\begin{eqnarray*}
+ \sum\limits_{{j=1, \atop j'=1}}^{m, k} C^{j, j'}_{m, k} 
\left( \left(\delta^{k_i}_{m_i}  \Phi^{(i)}\right)^{m-j},  
 \left(\delta^{k_i}_{m_i} \eta^{(i)}\right)^j,    
 \partial_t  \left( \Phi^{(i')} \right)^\beta, 
  \left(\Phi^{(i'')}\right)^{k-j'},  
\left(\eta^{(i'')}\right)^{j'} \right)  \qquad \qquad  &&
\nn
+\sum\limits_{{j=1, \atop j'=1}}^{m, k} C^{j, j'}_{m, k} 
\left(  \left(\delta^{k_i}_{m_i}  \Phi^{(i)}\right)^{m-j},  
 \left(\delta^{k_i}_{m_i} \eta^{(i)}\right)^j,    
 \partial_t  \left(  \eta^{(i')}\right)^\beta, 
  \left(\Phi^{(i'')}\right)^{k-j'},  
\left(\eta^{(i'')}\right)^{j'} \right).  \qquad \qquad  &&
\end{eqnarray*}
The terms except the first two vanish due to the mutual order condition 
of required in the Theorem. 
Then one can see that the cohomology class of \eqref{pampadur} is preserved. 
Similarly we show that
$\left( \left(\delta^{1, i}_{2, i} \Phi^{(i)} \right)^m \right.$,   
$\left(\partial_t \Phi^{(i')}\right)$,  
$\left. \left(\Phi^{(i'')}\right)^k \right)$  
and 
$\left(  \left(\delta^{0, i}_{3, i} \Lambda^{(i)} \right)^m,   
 \left(\partial_t \Lambda^{(i')}\right)^\beta, 
\left(\Lambda^{(i'')}\right)^k \right)$,   
 are invariant, i.e., 
it does not depend on the choices of 
$\Phi^{(i_s)} \in C^{1, i_s}_{2, i_s} \left(V, \W^{(i_s)} \right)$,    
$\Lambda^{(i_s)} \in C^{0, i_s}_{3, i_s} \left(V, \W^{(i)} \right)$, 
with  $i_s=i, i', i''$,
 satisfying the transversality condition \eqref{prontosan}
with the corresponding values of $i_s$, $i_{s'}$.  
\end{proof}

In this paper we provide results concerning complex curves. 
 They generalize to the case of higher dimensional complex manifolds.  
\section*{Acknowledgment} 
The author is supported by 
the Institute of Mathematics, Academy of Sciences of the Czech Republic (RVO 67985840). 
\section{Appendix: vertex operator algebras and matrix elements}
\label{grading}
In this Appendix we recall basic properties of 
grading-restricted vertex algebras \cite{Huang}
 and their modules. 
A vertex algebra   
$(V,Y_V,\mathbf{1}_V)$, 
\cite{FHL, K} 
  is a $\Z$-graded complex vector space  
$V = \coprod_{n\in\Z}\,V_{(n)}$,  $\dim V_{(n)}<\infty$,   
 for each $n\in \Z$.    
It is endowed with the linear map  
$Y_V:V\rightarrow {\rm End \;}(V)[[z,z^{-1}]]$,  
where $z$ is a formal parameter, 
and a 
distinguished vector $\mathbf{1}_V\in V$.   
The evaluation of $Y_V$ on $v\in V$ is called the vertex operator 
$Y_V(v)\equiv Y_V(v,z) = \sum_{n\in\Z}v(n)z^{-n-1}$,  
with components 
$(Y_V(v))_n=v(n)\in {\rm End \;}(V)$, where $Y_V(v,z)\mathbf{1}_V = v+O(z)$.
For the definition of a grading-restricted vertex algebra
and a grading-restricted generalized vertex algebra module we refer a reader 
to \cite{Huang}. 

 A quasi-conformal grading-restricted vertex algebra $V$ module 
 $W$ vector $A$ is called  
 primary of conformal dimension $\Delta(A) \in  \mathbb Z_+$ if  
 $L_W(k) A 
=
 0$,  $k > 0$,   
 $L_W(0) A 
=
 \Delta(A) A$. 
For $z' \in \C$, 
that vertex operators satisfy 
the translation property 
 $Y_W(u, z) = e^{-z' L_W(-1)} Y_W(u, z+z') e^{z' L_W(-1)}$.  
For $v\in V$, and $w \in W$, one defines 
 the intertwining operator 
\begin{eqnarray}
\label{interop}
 Y_{WV}^W: V\to W,  \quad 
v   \mapsto  Y_{WV}^W(w, z) v, 
\\
\label{wprop}
Y_{WV}^W(w, z) v= e^{zL_W(-1)} Y_W(v, -z) w. 
\end{eqnarray}
With the grading operator $L_W{(0)}$, 
the conjugation property 
for $a\in \C$ is 
\begin{equation}
\label{aprop}
 a^{L_W(0)} \; Y_W(v,z) \; a^{-L_W(0) }= Y_W \left(a^{L_W(0)} v, az\right).
\end{equation}

Now we recall 
  definitions and some properties of 
matrix elements for a grading-restricted vertex algebra $V$ \cite{Huang}.  
Let $W$ be a grading-restricted generalized $V$-module.   
In this paper we consider elements 
 $\Phi(g; v_1, z_1; \ldots; v_l, z_l) \in \W$, $l \ge 0$,  
endowed with an automorphism group ${\rm Aut}(V)$ elements $g$. 
Note that we assume that in $\Phi(g; v_1, z_1; \ldots; v_l, z_l)$
an automorphism $g$ acts first on elements 
of the corresponding module $W$. 
  The $\overline{W}$-valued function 
is given by 
\begin{eqnarray}
\label{pusnya}
&&
E^{(n)}_W\left(v_1, z_1; \ldots; v_n, z_n; 
\Phi(g; v'_1, z'_1; \ldots; v'_l, z'_l)\right)
\nn
&&
\qquad = E\left(\omega_W(v_1, z_1) \ldots \omega_W(v_n, z_n) \;  
\Phi(g; v'_1, z'_1; \ldots; v'_l, z'_l)\right),     
\end{eqnarray}
where 
$\omega_W(dz^{\wt(v)}\otimes v,z)=Y_W(dz^{\wt(v)}\otimes v,z)$,  
and an element $E(.)\in \overline{W}$ is given by   
$\langle w', E(g; \alpha) \rangle =R \langle w', g.\alpha \rangle$,   
$\alpha \in \overline{W}$ (here we use the notation of Subsection \ref{nahuy}).  
Here a group element $g$ is supposed to act both on $v'_j$, $1 \le j \le l$, 
and $v_i$, $1 \le i \le n$. 

For a number $l$ of generalized vertex algebra $V$-modules $W^{(i)}$,   
denote $\Phi^{(1, \ldots, l)} \in
 \W^{(1, \ldots, l)}_{z_1, \ldots, z_{k_1+\ldots+k_l-r}}$.  
 Then we define similarly  
\begin{eqnarray}
\label{xusnya}
E^{(m_1, \ldots, m_l)}_{W^{(1, \ldots, l)}}
\left(v_1, z_1; \ldots; v_{m_1+\ldots+m_l}, z_{m_1+\ldots+m_l}; 
\qquad \qquad \qquad \qquad \qquad \qquad 
\right. 
 && 
\nn
  \Phi^{(1, \ldots, l)}(g_1, \ldots, g_l; 
v_{m_1+\ldots+m_l+1}, z_{m_1+\ldots+m_l+1}; \ldots; 
\qquad \qquad    &&
\nn
 \left. v_{m_1+\ldots+m_l+k_1+\ldots+k_l}, z_{m_1+\ldots+m_l+k_1+\ldots+k_l})\right) 
\qquad   &&
\nn 
=
\sum_{u\in V_{(k)}, \; k \in \Z } \;
 \widehat{\mathcal R} \prod_{i=1}^l \rho_i^k  \langle w'_i, 
E^{(m_i)}_{W^{(i)}} 
\left(v_{1, i}, x_{1, i}; \ldots; v_{m_i, i}, x_{m_i, i};  \right. && 
\nn
\left. Y^{W^{(i)}}_{W^{(i)}V'} \left(
\Phi^{(i)}(g_i; v_{m_i+1, i}, x_{m_i+1, i}; \ldots; 
v_{m_i+k_i, i}, x_{m_i+k_i, i}; u, \zeta_{1, i}), \zeta_{2, i}  \right)
f_i.\overline{u} \right)  \rangle,  \quad   &&  
\end{eqnarray}
where $v_j$, $z_j$, $1 \le j \le m_1+\ldots+m_l+k_1+\ldots+k_l-r$ are 
vertex algebra elements and 
formal parameters for $\Phi^{(1, \ldots, l)}$, 
and $v_{i', i}$, $x_{i', i}$, $1 \le i' \le k_i-r_i$ are vertex algebra elements 
and formal parameters of $\Phi^{(i)}$.  
The form of \eqref{xusnya} is inspired 
by the adapted transversal condition for $\Phi^{(1, \ldots, l)}$. 
One defines also
$E^{W; (n)}_{WV'}$ $(\Phi(g$; $v'_1$, $z'_1$; $\ldots$; $v'_l$, $z'_l)$; 
$v_1$, $z_1$; $\ldots$; $v_n$, $z_n)$  
$=$ $E^{(n)}_W(v_1, z_1; \ldots;  v_n, z_n; 
\Phi(g; v'_1, z'_1; \ldots; v'_l, z'_l))$, 
which is an element of $\widetilde{W}_{z_1, \ldots, z_n}$.
In addition to that above, we define
$\left(E^{(l_1)}_{V;\;\one} \otimes \ldots \otimes E^{(l_n)}_{V;\;\one}\right).
\Phi:  
V^{\otimes m+n}\to  
\widetilde{W}_{z_1,  \ldots, z_{m+n}}$,  
\begin{eqnarray}
\label{kiznya}
\left(E^{(l_1)}_{V;\;\one}\otimes \ldots \otimes E^{(l_n)}_{V;\;\one}\right).
\Phi(g; v_1, z_1; \ldots; v_{m+n-1}, z_{m+n-1}) \qquad \qquad \qquad    &&
\nn
=E\left( \Phi \left(g; E^{(l_1)}_{V; \one}(v_1, z_1; \ldots;  v_{l_1}, z_{l_1});  \ldots;  
\right. \right.\qquad \qquad \qquad \qquad &&
\nn
 \left. \left. 
E^{(l_n)}_{V; \one}
(v_{l_1+\ldots +l_{n-1}+1}, z_{l_1+\ldots +l_{n-1}+1};  \ldots;  
 v_{l_1+\ldots +l_{n-1}+l_n}, z_{l_1+\ldots +l_{n-1}+l_n})\right)\right), &&  
\end{eqnarray}
and $E^{(m)}_W.
\Phi: V^{\otimes m+n}\to 
\widetilde{W}_{z_1, \ldots, z_{m+n-1}}$, 
  given by 
\begin{eqnarray*}
&& E^{(m)}_W.
\Phi(g; v_1, z_1; \ldots; v_{m+n}, z_{m+n}) 
\nn
&&
\qquad =E\left(E^{(m)}_W\left( v_1, z_1; \ldots; v_m, z_m; 
\Phi(g; v_{m+1}, z_{m+1}; \ldots; v_{m+n}, z_{m+n})\right)\right).  
\end{eqnarray*}
For $l_1=\ldots=l_{i-1}=l_{i+1}=1$, $l_i=m-n-1$, $1 \le i \le n$,  
 by $E^{(l_i)}_{V;\;\one}.\Phi$ we 
denote $(E^{(l_1)}_{V;\;\one}\otimes \ldots 
\otimes E^{(l_n)}_{V;\;\one}).\Phi$, 
(this notation is different that of \cite{Huang}).  
In \cite{Huang} the algebra of $E$-operators was derived. 

\end{document}